\documentclass[hidelinks,onefignum,onetabnum]{siamart251216}

\usepackage{hyperref}       
\usepackage{url}            
\usepackage{booktabs}       
\usepackage{thmtools}       
\usepackage{amsmath}        
\usepackage{amssymb}        
\usepackage{bm}             
\usepackage{mathrsfs}       
\numberwithin{equation}{section}

\usepackage{nicefrac}       
\usepackage{xcolor}
\usepackage{iftex}

\usepackage{upquote}
  \usepackage[]{microtype}
  \UseMicrotypeSet[protrusion]{basicmath} 
\usepackage{bookmark}
\usepackage{xurl} 
\usepackage{cleveref}       

\usepackage{graphicx}
\usepackage{doi}
\usepackage{subcaption}
\usepackage{tikz-cd}
\usepackage{multirow}
\headers{Stabilization-Free \(\mathbf{H}\left(\operatorname{\mathbf{curl}}\right)\) and \(\mathbf{H}\left(\operatorname{div}\right)\)-Conforming VEM}{Yuxuan Liao et al.}

\title{Stabilization-Free \(\mathbf{H}\left(\operatorname{\mathbf{curl}}\right)\) and \(\mathbf{H}\left(\operatorname{div}\right)\)-Conforming Virtual Element Method}
\newsiamremark{remark}{Remark}
\newsiamremark{hypothesis}{Hypothesis}
\newsiamthm{claim}{Claim}
\newsiamthm{assumption}{Assumption}
\newsiamremark{fact}{Fact}
\author{Yuxuan Liao\thanks{Department of Electronic Engineering, Tsinghua University, 100084 Beijing, People's Republic of China.}
\and Xue Feng\thanks{Corresponding authors. Department of Electronic Engineering, Tsinghua University, 100084 Beijing, People's Republic of China.
  (\email{x-feng@tsinghua.edu.cn}, \email{yidonghuang@tsinghua.edu.cn}).}
\and Yidong Huang\footnotemark[2]}

\usepackage{amsopn}

\newcommand{\normmm}[1]{{\left\vert\kern-0.25ex\left\vert\kern-0.25ex\left\vert #1 
   \right\vert\kern-0.25ex\right\vert\kern-0.25ex\right\vert}}
\newcommand{\bfitgreek}[1]{\bm{\mathit{#1}}}

\ifpdf
\hypersetup{
  pdftitle={Stabilization-Free \(\mathbf{H}\left(\operatorname{\mathbf{curl}}\right)\) and \(\mathbf{H}\left(\operatorname{div}\right)\)-Conforming Virtual Element Method},
  pdfauthor={Yuxuan Liao et al.},
}
\fi

\begin{document}

\maketitle

\begin{abstract}
  Standard Virtual Element Method (VEM) requires stabilization terms that significantly affect the numerical computation performance. In this work, we propose a stabilization-free VEM for general order \(\mathbf{H}(\operatorname{\mathbf{curl}})\) and \(\mathbf{H}(\operatorname{div})\)-conforming spaces by constructing novel serendipity projectors and corresponding serendipity spaces with minimum number of DoFs. Our approach handles the full De Rham complex chain in \(\mathbb{R}^3\) while preserving essential properties including boundary continuity and commutativity. Since the number of DoFs are minimized, computational overhead is greatly reduced. The optimal approximation properties are rigorously proven and validated through Maxwell eigenvalue problems with numerical experiments.
\end{abstract}



\section{Introduction}

The Virtual Element Method (VEM)\cite{beiraodaveigaVirtualElementMethod2023}, introduced as a natural extension of the finite element method (FEM), has gained considerable attention for its exceptional capability to handle general star-shaped polytopes. Such flexibility offers significant advantages over conventional finite element methods, particularly for complex geometries and adaptive mesh refinement. For comprehensive reviews of VEM's advantages and applications, we refer to \cite{antoniettiVirtualElementMethod2022a}.

However, its geometric flexibility comes at a cost. Unlike traditional FEM, the virtual element function space \(\mathbf{V}(K)\) is no longer purely polynomial, necessitating projections onto polynomial spaces for numerical computation. Specifically, to solve weak formulations of the form \(a(u,v)=b(v)\), the local discrete bilinear form must be reformulated as \(a_{K}(\bfitgreek{\Pi}_{k}u,\bfitgreek{\Pi}_{k}v)\), where \(\bfitgreek{\Pi}_{k}:\mathbf{V}(K)\to \mathbb{P}_{k}(K)\) denotes the polynomial projection.

The fundamental challenge in standard VEM arises from the instability of this projection: since typically \(\dim\mathbb{P}_{k}^{d}(K) < \dim\mathbf{V}(K)\), the inequality \(\|\bfitgreek{\Pi}_{k}v\|_{K}\gtrsim\|v\|_{K}\) fails to hold. Consequently, stabilization terms of the form
\[
  a_{K}(v,v)\lesssim S_{K}(v,v)\lesssim a_{K}(v,v)\ \forall v\in \ker\bfitgreek{\Pi}_{k}
\]
are essential to maintain the coercivity. These stabilization terms present several critical drawbacks: they significantly impact numerical stability and condition numbers\cite{beiraodaveigaStabilityAnalysisVirtual2017}, require careful problem-specific design (especially for nonlinear problems\cite{wriggersVirtualElementMethods2024}), and exhibit isotropic behavior that proves problematic for anisotropic coefficients or meshes\cite{berroneComparisonStandardStabilization2022}.

\textbf{Existing Stabilization-Free Approaches.} Recent efforts to eliminate stabilization terms have pursued two distinct strategies.

The first is to modify the virtual
element space, so that a stable \(\mathrm{L}^2\)-polynomial
projection can be computed. The first work for \(\mathrm{H}^1\)-conforming virtual element
spaces is \cite{berroneLowestOrderStabilization2023}. It utilizes the lowest order
\(\mathrm{H}^1\)-seminorm projector \(\bfitgreek{\Pi}^\nabla _{1}\), and construct
the virtual element space with functions equivalent under this
projector:
\begin{equation}
  \begin{split}
    \mathbf{V}_{1, l}\left(F\right):=\left\{v\in \mathrm{H}^1\mid \left(\bfitgreek{\Pi}^\nabla_{1}v-v,p\right)_{F}=0\
    \forall p\in \mathbb{P}_{l+1}(F), \Delta v\in \mathbb{P}_{l+1}(F),\right. \\
    \left.v_{\mid E} \in \mathbb{P}_1(E)\ \forall E \in \partial F, v \in C^0(\partial F)\right\}.
  \end{split}
\end{equation}
Then the \(\mathrm{L}^2\)-projection of \(\nabla v\) to
\(\mathbb{P}^2_{l}\left(F\right)\) is computable, so
\(\left(\nabla u,\nabla v\right)_{F}\ \forall u,v\in \mathbf{V}_{1,l}\left(F\right)\)
is computed as
\(\left(\bfitgreek{\Pi}_{l}\nabla u,\bfitgreek{\Pi}_{l}\nabla v\right)_{F}\).
This discrete bilinear form requires no stabilization since
\(\left\|\nabla v\right\|_{F}\approx\left\|\bfitgreek{\Pi}_{l}\nabla v\right\|_{F}\). This technique, while can be extended to general order spaces, inevitably requires the internal DoFs for \(k>1\). Since \(\dim \mathbb{P}_{l}^{2}(F)\geq \dim \mathbf{V}(F)\) is necessary for \(\bfitgreek{\Pi}_{l}\) to be stable, these internal DoFs, though can be eventually eliminated by static condensation, still require \(l\) much larger than it truly needs to be, which would greatly increase the computational cost.
While in \cite{chenStabilizationfreeSerendipityVirtual2023}
this technique is further extended to general order spaces, by replacing
\(\bfitgreek{\Pi}^\nabla _{1}\) with the serendipity projector
\(\bfitgreek{\Pi}^{\mathrm{S}}_{k}\). Later in \cite{xuStabilizationfreeVirtualElement2023,xuStabilizationfreeVirtualElement2024}, the same technique is applied to nonlinear two-dimensional (2D) and three-dimensional (3D) finite elasticity problems, which greatly exhibit the advantages of stabilization-free VEM. Since the serendipity projector can be computed using only the boundary DoFs, it overcomes the internal DoFs issue. However, the trivially chosen serendipity projector, do not preserve many crucial properties required by other weak formulations. For example, the lack of commutativity with differential operators, makes it unapplicable to mixed formulations.

The other strategy is to modify the projection space. In \cite{chenVirtualElementMethods2024}, an
\(\mathbf{H}\left(\operatorname{div}\right)\)-conforming macro finite element
\(\mathbf{V}^{\operatorname{div}}_{k-1}\left(K\right)\) is constructed,
and the \(\mathrm{L}^2\)-projection from the conventional virtual element space
\(\mathbf{V}_{k}\left(K\right)\) to
\(\mathbf{V}^{\operatorname{div}}_{k-1}\left(K\right)\) is stable and
computable. But this method would dilute the advantages of VEM to some extent, since it requires simplical tessellation of polyhedra
and the construction of corresponding broken polynomial spaces, while VEM with the original form, is designed to handle more general polytopes.

\textbf{Our Contribution.} In this work, we introduce a fundamentally different stabilization-free VEM that addresses the limitations of existing approaches while preserving all essential mathematical and computational properties. Our method is built upon four key innovations:

\emph{Complete De Rham Complex Treatment:} Unlike previous works that handle \(\mathrm{H}^1\) space, we construct commutative serendipity projectors for the entire De Rham complex chain in \(\mathbb{R}^3\). This unified treatment not only ensures seamless integration across different differential operators, but also can be systematically extended to higher-dimensional De Rham complexes by recursively applying the same technique to boundaries of each dimension.

\emph{Preservation of Essential Mathematical Properties:} Our serendipity spaces rigorously maintain boundary continuity, operator conformity, exact commutativity with differential operators, and B-compatibility\cite{boffiMixedFiniteElement2013}.

\emph{Computational Efficiency:} Our approach achieves stabilization-free computation using only serendipity projections computed with minimum number of DoFs. This reduces the stable projector order significantly, for general order virtual element spaces. Notably, even for stabilized treatment, the reduction in DoFs is also substantial as it is applicable to eigenvalue problems where internal DoFs cannot be eliminated by static condensation.

The resulting stabilization-free VEM preserves all advantages of the original method while eliminating stabilization-related complications, making it particularly valuable for mixed formulations, nonlinear problems, and applications requiring robust numerical behavior across diverse mesh configurations.

\section{Preliminaries}

We utilize standard Sobolev spaces whose detailed definitions can be found in \cite{evansPartialDifferentialEquations2010}. For a Lipschitz domain \(D \subseteq \mathbb{R}^d\) with \(d=1,2,3\):

\begin{itemize}
  \item \(\mathrm{L}^2(D)\): space of square-integrable functions
  \item \(\mathrm{H}^s(D)\): space of functions with square-integrable derivatives up to order \(s\geq0\)
\end{itemize}

We use bold upright typeface for vector-valued spaces:
\begin{itemize}
  \item \(\mathbf{H}^s(\operatorname{\mathbf{curl}}, D)\): functions whose curl belongs to \(\mathbf{H}^s(D)\)
  \item \(\mathbf{H}^s(\operatorname{div}, D)\): functions whose divergence belongs to \(\mathrm{H}^s(D)\)
\end{itemize}

For any function space \(X\), we denote:
\begin{itemize}
  \item \(\|\cdot\|_{X}\) and \(|\cdot|_{X}\): norms and seminorms equipped to \(X\)
  \item \(\|\cdot\|_{s,D}\) and \(|\cdot|_{s,D}\): norms and seminorms for \(\mathrm{H}^s(D)\) or \(\mathbf{H}^s(D)\)
  \item \((\cdot,\cdot)_{D}\): \(\mathrm{L}^2\)-inner product in domain \(D\)
\end{itemize}
When \(s=0\), the subscript \(s\) is omitted. When the domain is clear from context, the subscript \(D\) is omitted. The subscript \(0\) denotes the space of functions with zero trace, and the subscript \(\operatorname{div}0\) denotes the space of functions with zero divergence.

For linear spaces \(X\subseteq Y\), \(Y/X\) denotes the subspace of \(Y\) naturally isomorphic to the quotient space in conventional notation.

For polynomial spaces, we adopt the following notation:
\begin{itemize}
  \item \(\mathbb{P}^q_{k,d}(D)\): polynomials from \(\mathbb{R}^d\) to \(\mathbb{R}^q\) of degree at most \(k\)
  \item \(\mathbb{P}^q_{k_1+1:k_2,d}(D)\): \(\mathbb{P}^q_{k_2,d}(D)/\mathbb{P}^q_{k_1,d}(D)\)
\end{itemize}
Following standard convention: \(q\) is omitted when \(q=1\), \(d\) is omitted when clear from context, and \(k\) is omitted for arbitrarily large polynomial degree.

We reserve specific notation for these entities:
\begin{itemize}
  \item \(\mathbf{V}\): finite element spaces
  \item \(\bfitgreek{\Pi}\): general projectors
  \item \(\bfitgreek{\Pi}^q_{k,d}\): \(\mathrm{L}^2\)-projector onto \(\mathbb{P}^q_{k,d}\)
\end{itemize}
The subscripts \(q,k,d\) follow the same omission rules as polynomial spaces.
\begin{itemize}
  \item \(\boldsymbol{I}\): interpolators
  \item \(\boldsymbol{E}\): identity transforms

  \item \(K,F,E,V\): polyhedron (body), face, edge, and vertex (node), respectively
  \item \(\boldsymbol{n}\): outward unit normal vector of face \(F\)
  \item \(\boldsymbol{t}\): unit tangent vector of edge \(E\)
  \item \(\nabla_F\): surface nabla operator on face \(F\)
  \item \(\boldsymbol{v}^F\): tangential component of \(\boldsymbol{v}\) on face \(F\)
\end{itemize}

For positive quantities \(a\) and \(b\):
\begin{itemize}
  \item \(a \lesssim b\): there exists constant \(c>0\) such that \(a \leq cb\)
  \item \(a \gtrsim b\): there exists constant \(c>0\) such that \(a \geq cb\)
  \item \(a \approx b\): both \(a \lesssim b\) and \(b \lesssim a\) hold
\end{itemize}
All constants are required to be independent of the mesh size parameter \(h\).

\section{Serendipity Virtual Element Spaces: Construction, Interpolation, and Stability Analysis}

In this section, we first introduce the serendipity projector and the corresponding serendipity space (\cref{def:vs}) and the standard virtual element spaces (\cref{def:vn,def:vse,def:vsf}). Then we construct their serendipity counterparts with our novel serendipity projectors (\cref{def:vse,def:vsf,def:vsn}) and prove the approximation properties (\cref{thm:vf,thm:ve,thm:vn}). Finally, we establish the stability of the high-order polynomial projection (\cref{thm:st}) and discuss how the corresponding order \(l\) should be chosen in practice.

For the sake of simplicity, in this section we omit the subscript \(h\) for mesh size and \(K\) for element, when no ambiguity arises. In addition, \cref{ass:mesh} is assumed throughout this paper.

\begin{definition}[General Serendipity Projector and Space]
  \label{def:vs}Let \(\mathbf{V}\) be a finite element space and \(\dim\mathbf{V}=N_{\mathrm{V}}\),
  \(\mathscr{S}\subseteq \mathbf{V}\) be the subspace we want to keep and \(\dim\mathscr{S}=N_{\mathrm{P}}\). Let
  \(\left\{\mathcal{F}_{i}\right\}_{i=1}^{N_{\mathrm{V}}}\) be the DoFs of
  \(\mathbf{V}\). Define the vector of DoFs
  \[\mathcal{D}_{N_{\mathrm{V}}}\left(\boldsymbol{p}\right):=\left[\mathcal{F}_{1}\left(\boldsymbol{p}\right),\dots,\mathcal{F}_{N_{\mathrm{V}}}\left(\boldsymbol{p}\right)\right]^\top.\]
  With necessary re-order of the DoFs, we call the first
  \(N_{\mathrm{S}}\) DoFs \(\mathscr{S}\)-identifying, if:
  \begin{equation}
    \mathcal{D}_{N_{\mathrm{S}}}\left(\boldsymbol{p}\right)=0 \implies \boldsymbol{p}=0\ \forall \boldsymbol{p}\in \mathscr{S}.
  \end{equation}

  Then serendipity projector \(\bfitgreek{\Pi}^{\mathrm{S}}:\mathbf{V}\rightarrow\mathscr{S}\) is defined as:
  \begin{equation}
    \mathcal{D}_{N_{\mathrm{S}}}\left(\bfitgreek{\Pi}^{\mathrm{S}}\boldsymbol{v}-\boldsymbol{v}\right)^\top\mathbf{G}\mathcal{D}_{N_{\mathrm{S}}}\left(\boldsymbol{p}\right)=0\ \forall \boldsymbol{p}\in \mathscr{S},
  \end{equation}
  where \(\mathbf{G}\) is some symmetric and positive definite matrix. And the serendipity space is defined as:
  \begin{equation}
    \mathbf{V}^\mathrm{S}=\left\{\boldsymbol{v} \in \mathbf{V} \mid \mathcal{D}_{ N_{\mathrm{S}}+1: N_{\mathrm{V}}}(v)=\mathbf{C}\mathcal{D}_{ N_{\mathrm{S}}}\left(\boldsymbol{v}\right)\right\},
    \label{eq:defvs}
  \end{equation}
  where \(\mathbf{C}\) is some coefficient matrix.
\end{definition}

We hereafter call
\(\left\{\mathcal{F}_{i}\right\}_{i=1}^{N_{\mathrm{S}}}\) ``preserved
DoFs'', and the remaining ``reduced DoFs''.

\begin{remark}
  \label{rem:vs}The major difference of definition here compared to that introduced in \cite{veigaSerendipityFaceEdge2017}, is that we do not necessarily determine the reduced DoFs with the serendipity projector equivalence. This provides us with greater flexibility in constructing spaces that meet the requirements, e.g., conformity, accuracy, compatibility with other spaces, and so on.
\end{remark}

Next we recall the results of the standard virtual element spaces
for \(\mathrm{H}^1,\mathbf{H}\left(\operatorname{\mathbf{curl}}\right)\) and
\(\mathbf{H}\left(\operatorname{div}\right)\)\cite{beiraodaveigaVirtualElementMethod2023}.

Let \(\mathcal{T}_{h}\) be a mesh partitioning of the computational
domain \(\Omega\). \(h_{E},h_{F},h_{K}\) are the diameters of the edge, face and polyhedron, respectively. The mesh size parameter \(h\) is defined as the maximum of \(h_{K}\).
The following mesh regularity is assumed:

\begin{assumption}[Mesh Regularity]
  \label{ass:mesh}There exists a uniform constant \(\rho>0\) such that, for every
  polyhedron \(K\in \mathcal{T}_{h}\),
  \begin{enumerate}
    \def\labelenumi{(\roman{enumi})}
    \item
          \(K\) is star-shaped with respect to a ball of radius
          \(\geq \rho h_K\);
    \item
          every face \(F\) of \(\partial K\) is star-shaped with respect to a
          disk of radius \(\geq \rho h_F\);
    \item
          for every face \(F\) of \(\partial K\), every edge \(E\) of
          \(\partial F\) satisfies \( h_E \geq \rho h_F \geq \rho^2 h_K\).
  \end{enumerate}
\end{assumption}

\begin{definition}[Face Virtual Element Space]
  Define:
  \begin{equation}
    \begin{aligned}
      \mathbf{V}_{k, k_{\mathrm{d}}, k_{\mathrm{r}}}^{\mathrm{f}}(K):= & \left\{\boldsymbol{v}\in\mathbf{L}^2(K) \mid \boldsymbol{v} \cdot \boldsymbol{n}_F \in \mathbb{P}_k(F)\ \forall F\in \partial K,\right.           \\
                                                                       & \ \left.\nabla \cdot \boldsymbol{v} \in \mathbb{P}_{k_{\mathrm{d}}}(K), \nabla \times \boldsymbol{v} \in\mathbb{P}_{k_{\mathrm{r}}}^3(K)\right\},
    \end{aligned}
  \end{equation}
  with DoFs:
  \begin{subequations}
    \label{eq:dofvf}
    \begin{align}
       & \left(\boldsymbol{v} \cdot \boldsymbol{n}, p_{k}\right)_{F}                                &  & \forall F\subseteq\partial K,p_{k}\in\mathbb{P}_k(F),\label{eq:dofvf1}                       \\
       & \left(\boldsymbol{v}, \nabla p_{k_{\mathrm{d}}}\right)_{K}                                 &  & \forall p_{k_{\mathrm{d}}}\in\mathbb{P}_{k_{\mathrm{d}}}(K),\label{eq:dofvf2}                \\
       & \left(\boldsymbol{v}, \boldsymbol{x}_{K} \times \boldsymbol{p}_{k_{\mathrm{r}}}\right)_{K} &  & \forall \boldsymbol{p}_{k_{\mathrm{r}}}\in\mathbb{P}_{k_{\mathrm{r}}}^3(K).\label{eq:dofvf3}
    \end{align}
  \end{subequations}
\end{definition}

\begin{definition}[Edge Virtual Element Space]
  Define:
  \begin{equation}
    \begin{aligned}
      \mathbf{V}_{k, k_{\mathrm{d}}, k_{\mathrm{r}}}^{\mathrm{e}}(F):= & \left\{\boldsymbol{v} \in \mathbf{L}^2(F)\mid \boldsymbol{v}\cdot \boldsymbol{t}\in \mathbb{P}_{k}(E)\ \forall E \subseteq \partial F,\right.                        \\
                                                                       & \ \left. \nabla_{F} \cdot \boldsymbol{v} \in \mathbb{P}_{k_{\mathrm{d}}}(F), \nabla_{F} \times \boldsymbol{v} \in \mathbb{P}_{k_{\mathrm{r}}}\left(F\right)\right\},
    \end{aligned}
  \end{equation}
  \begin{equation}
    \begin{aligned}
      \mathbf{V}_{\boldsymbol{\beta},k_{\mathrm{d}},\mu_{\mathrm{r}}}^{\mathrm{e}}(K):= & \left\{\boldsymbol{v} \in \mathbf{L}^2(K)\mid \boldsymbol{v}^F \in \mathbf{V}_{\boldsymbol{\beta}}^{\mathrm{e}}(F)\ \forall F \subseteq \partial K,\right.                                                                                                               \\
                                                                                        & \ \left.\nabla \cdot \boldsymbol{v} \in \mathbb{P}_{k_{\mathrm{d}}}(K), \nabla \times \boldsymbol{v} \in \mathbf{V}^{\mathrm{f}}_{\boldsymbol{\mu}}\left(K\right), \boldsymbol{v} \cdot \boldsymbol{t}_E \text { is continuous } \forall E \subseteq \partial F\right\},
    \end{aligned}
  \end{equation}
  where
  \(\boldsymbol{\beta}=\left(\beta,\beta_{\mathrm{d}},\beta_{\mathrm{r}}\right), \boldsymbol{\mu}=\left(\beta_{\mathrm{r}},-1,\mu_{\mathrm{r}}\right)\),
  with DoFs:
  \begin{subequations}\label{eq:dofve}
    \begin{align}
       & \left(\boldsymbol{v}\cdot\boldsymbol{t}, p_{\beta}\right)_{E}
       &                                                                                                        & \forall E\subseteq\partial K, p_{\beta}\in\mathbb{P}_{\beta}(E),\label{eq:dofve1}                           \\
       & \left(\boldsymbol{v}, \boldsymbol{x}_{F} p_{\beta_{\mathrm{d}}}\right)_{F}
       &                                                                                                        & \forall F\subseteq\partial K, p_{\beta_{\mathrm{d}}}\in\mathbb{P}_{\beta_{\mathrm{d}}}(F),\label{eq:dofve2} \\
       & \left(\boldsymbol{v}, \nabla_{F}\times p_{\beta_{\mathrm{r}}}\right)_{F}
       &                                                                                                        & \forall F\subseteq\partial K, p_{\beta_{\mathrm{r}}}\in\mathbb{P}_{\beta_{\mathrm{r}}}(F),\label{eq:dofve3} \\
       & \left(\nabla\times\boldsymbol{v}, \boldsymbol{x}_{K}\times\boldsymbol{p}_{\mu_{\mathrm{r}}}\right)_{K}
       &                                                                                                        & \forall \boldsymbol{p}_{\mu_{\mathrm{r}}}\in\mathbb{P}^3_{\mu_{\mathrm{r}}}(K),\label{eq:dofve4}            \\
       & \left(\boldsymbol{v}, \boldsymbol{x}_{K} p_{k_{\mathrm{d}}}\right)_{K}
       &                                                                                                        & \forall p_{k_{\mathrm{d}}}\in\mathbb{P}_{k_{\mathrm{d}}}(K).\label{eq:dofve5}
    \end{align}
  \end{subequations}
\end{definition}

\begin{definition}[Node Virtual Element Space]
  \label{def:vn}Let \(k_{\mathrm{v}}\leq k_{\mathrm{f}}\), define:
  \begin{equation}
    \begin{aligned}
      \mathbf{V}^{\mathrm{n}}_{k,k_{\mathrm{f}},k_{\mathrm{v}}}\left(K\right):= & \left\{v\in H^1(K) \mid v_{\mid E}\in \mathbb{P}_{k}\left(E\right)\ \forall E\subseteq \partial F\ \forall F\subseteq \partial K,\right. \\
                                                                                & \ \left.\nabla v \in \mathbf{V}_{\boldsymbol{\beta},k_{\mathrm{v}},-1}^{\mathrm{e}}\left(K\right)\right\},
    \end{aligned}
  \end{equation}
  where \(\boldsymbol{\beta}=\left(k-1,k_{\mathrm{f}},-1\right)\), with
  DoFs:
  \begin{subequations}\label{eq:dofvn}
    \begin{align}
       & v\left(V\right)                        &  & \forall V\subseteq\partial K,\label{eq:dofvn1}                                                      \\
       & \left(v, p_{k-2}\right)_{E}            &  & \forall E\subseteq\partial K, p_{k-2}\in\mathbb{P}_{k-2}(E),\label{eq:dofvn2}                       \\
       & \left(v, p_{k_{\mathrm{f}}}\right)_{F} &  & \forall F\subseteq\partial K, p_{k_{\mathrm{f}}}\in\mathbb{P}_{k_{\mathrm{f}}}(F),\label{eq:dofvn3} \\
       & \left(v, p_{k_{\mathrm{v}}}\right)_{K} &  & \forall p_{k_{\mathrm{v}}}\in\mathbb{P}_{k_{\mathrm{v}}}(K).\label{eq:dofvn4}
    \end{align}
  \end{subequations}
\end{definition}

Then we define the local serendipity projectors and spaces. The aim is to construct a space \(\mathbf{V}^{\mathrm{S}}_{k,l}(K)\), such that:
\begin{itemize}
  \item \(\mathbb{P}^{q}_{k,d}(K)\subseteq \mathbf{V}^{\mathrm{S}}_{k,l}(K)\) can be identified by the serendipity projector \(\bfitgreek{\Pi}^{\mathrm{S}}_{k}\) with the preserved DoFs,
  \item \(\bfitgreek{\Pi}_{l}\) is computable combined with the reduced DoFs, and stable if \(l\) is sufficiently large,
  \item exact commuting diagrams (\cref{thm:excom}) hold.
\end{itemize}

\begin{definition}[Novel Serendipity Projectors]
  \label{def:sp}Given sufficiently large \(k_{\mathrm{r}}\), define \(\bfitgreek{\Pi}^{\mathrm{Sf}}_{k}:\mathbf{H}(\operatorname{div},K)\rightarrow\mathbb{P}^3_{k}(K)\) as:
  \begin{subequations}\label{eq:sf}
    \begin{align}
       & \left(\left(\boldsymbol{v}-\bfitgreek{\Pi}^{\mathrm{Sf}}_{k}\boldsymbol{v}\right)\cdot \boldsymbol{n}, \nabla \times \boldsymbol{p}_{k+1} \cdot \boldsymbol{n}\right)_{\partial K}=0 &  & \forall \boldsymbol{p}_{k+1}\in\mathbb{P}_{k+1}^3(K),\label{eq:sf1}                       \\
       & \left(\nabla\cdot\left(\boldsymbol{v}-\bfitgreek{\Pi}^{\mathrm{Sf}}_{k}\boldsymbol{v}\right), p_{k-1}\right)_{K}=0                                                                   &  & \forall p_{k-1}\in\mathbb{P}_{k-1}(K),\label{eq:sf2}                                      \\
       & \left(\boldsymbol{v}-\bfitgreek{\Pi}^{\mathrm{Sf}}_{k}\boldsymbol{v}, \boldsymbol{x}_{K} \times \boldsymbol{p}_{k_{\mathrm{r}}}\right)_{K}=0                                         &  & \forall \boldsymbol{p}_{k_{\mathrm{r}}}\in\mathbb{P}_{k_{\mathrm{r}}}^3(K).\label{eq:sf3}
    \end{align}
  \end{subequations}

  Given sufficiently large \(\beta_{\mathrm{d}}\), define \(\tilde{\bfitgreek{\Pi}}^{\mathrm{Se}}_{k}:\mathbf{H}(\operatorname{rot},F) \rightarrow\mathbb{P}^2_{k}(F)\) as:
  \begin{subequations}\label{eq:sef}
    \begin{align}
       & \left(\left(\boldsymbol{v}-\tilde{\bfitgreek{\Pi}}^{\mathrm{Se}}_{k}\boldsymbol{v}\right)\cdot\boldsymbol{t}, \nabla p_{k+1} \cdot \boldsymbol{t}\right)_{\partial F}=0
       &                                                                                                                                                                         & \forall p_{k+1}\in\mathbb{P}_{k+1}(F),                               \\
       & \left(\nabla_{F}\times \left(\boldsymbol{v}-\tilde{\bfitgreek{\Pi}}^{\mathrm{Se}}_{k}\boldsymbol{v}\right), p_{k-1}\right)_{F}=0
       &                                                                                                                                                                         & \forall p_{k-1}\in\mathbb{P}_{k-1}(F),                               \\
       & \left(\boldsymbol{v}-\tilde{\bfitgreek{\Pi}}^{\mathrm{Se}}_{k}\boldsymbol{v}, \boldsymbol{x}_{F} p_{\beta_{\mathrm{d}}}\right)_{F}=0
       &                                                                                                                                                                         & \forall p_{\beta_{\mathrm{d}}}\in\mathbb{P}_{\beta_{\mathrm{d}}}(F).
    \end{align}
  \end{subequations}

  Given sufficiently large \(k_{\mathrm{d}}\) and \(\mu_{\mathrm{r}}\), define \(\bfitgreek{\Pi}^{\mathrm{Se}}_{k}:\mathbf{H}(\operatorname{\mathbf{curl}},K)\rightarrow\mathbb{P}^3_{k}(K)\) as:
  \begin{subequations}\label{eq:se}
    \begin{align}
       & \left(\left(\boldsymbol{v}-\bfitgreek{\Pi}^{\mathrm{Se}}_{k}\boldsymbol{v}\right)\times\boldsymbol{n},\nabla p_{k+1}\times\boldsymbol{n}\right)_{\partial K}=0
       &                                                                                                                                                                                              & \forall F\subseteq\partial K, p_{k+1}\in\mathbb{P}_{k+1}(F),\label{eq:se1}                     \\
       & \left(\nabla\times\left(\boldsymbol{v}-\bfitgreek{\Pi}^{\mathrm{Se}}_{k}\boldsymbol{v}\right)\cdot \boldsymbol{n}, \nabla\times\boldsymbol{p}_{k} \cdot \boldsymbol{n}\right)_{\partial K}=0
       &                                                                                                                                                                                              & \forall\boldsymbol{p}_{k}\in\mathbb{P}_{k}^3(K),\label{eq:se2}                                 \\
       & \left(\boldsymbol{v}-\bfitgreek{\Pi}^{\mathrm{Se}}_{k}\boldsymbol{v}, \boldsymbol{x}_{K} p_{k_{\mathrm{d}}}\right)_{K}=0
       &                                                                                                                                                                                              & \forall p_{k_{\mathrm{d}}}\in\mathbb{P}_{k_{\mathrm{d}}}(K),\label{eq:se3}                     \\
       & \left(\nabla \times\left(\boldsymbol{v}-\bfitgreek{\Pi}^{\mathrm{Se}}_{k}\boldsymbol{v}\right),\boldsymbol{x}_{K}\times \boldsymbol{p}_{\mu_{\mathrm{r}}}\right)_K=0
       &                                                                                                                                                                                              & \forall \boldsymbol{p}_{\mu_{\mathrm{r}}}\in \mathbb{P}^3_{\mu_{\mathrm{r}}}(K).\label{eq:se4}
    \end{align}
  \end{subequations}

  Define \(\tilde{\bfitgreek{\Pi}}^{\mathrm{Sn}}_{k}:\mathrm{H}^1(F)\rightarrow\mathbb{P}_{k}(F)\) as:
  \begin{subequations}
    \begin{align}
       & \left(\partial_{\boldsymbol{t}}\left(v-\tilde{\bfitgreek{\Pi}}^{\mathrm{Sn}}_{k}v\right), \partial_{\boldsymbol{t}} p_{k}\right)_{\partial F}=0   &  & \forall p_{k}\in\mathbb{P}_{k}(F), \\
       & \left(v-\tilde{\bfitgreek{\Pi}}^{\mathrm{Sn}}_{k}v, \boldsymbol{n}\cdot \boldsymbol{x}_F\right)_{\partial F}=0.
    \end{align}
  \end{subequations}

  Define \(\bfitgreek{\Pi}^{\mathrm{Sn}}_{k}:\mathrm{H}^1(K)\rightarrow\mathbb{P}_{k}(K)\) as:
  \begin{subequations}
    \begin{align}
       & \left(\nabla_F\left(v-\bfitgreek{\Pi}^{\mathrm{Sn}}_{k}v\right), \nabla_F p_{k}\right)_{\partial K}=0   &  & \forall p_{k}\in\mathbb{P}_{k}(K), \\
       & \left(v-\bfitgreek{\Pi}^{\mathrm{Sn}}_{k}v, \boldsymbol{n}\cdot \boldsymbol{x}_K\right)_{\partial K}=0.
    \end{align}
  \end{subequations}
\end{definition}

\begin{proposition}
  \label{prop:sf}\(\bfitgreek{\Pi}^{\mathrm{Sf}}_{k}\) is well-defined.
\end{proposition}

\begin{proof}
  Replacing \(\boldsymbol{v}-\bfitgreek{\Pi}^{\mathrm{Sf}}_{k}\boldsymbol{v}\) in \cref{eq:sf} by \(\boldsymbol{p}\in \mathbb{P}^3_{k}(K)\), we need to check that these conditions imply \(\boldsymbol{p}=0\). Since \cref{eq:sf2} implies that \(\nabla\cdot\boldsymbol{p}=0\). Combining \cref{eq:sf1}, \(\boldsymbol{p}\) has zero normal trace. In fact, since the following equivalent Poission problem:
  \begin{equation}
    \begin{cases}
      \Delta s=0                   & \text{in }K,          \\
      \partial_{\boldsymbol{n}}s=0 & \text{on }\partial K,
    \end{cases}
  \end{equation}
  has the only solution \(\nabla s=0\), \(\boldsymbol{p}\) is not a gradient of a polynomial. By \cref{eq:pdcomp1}, we deduce that \(\boldsymbol{p}=\boldsymbol{x}_K\times\boldsymbol{r}_{k-1}\) for some \(\boldsymbol{r}_{k-1}\in\mathbb{P}_{k-1}^3(K)\). Finally \cref{eq:sf3} with sufficiently large \(k_{\mathrm{r}}\) ensures \(\boldsymbol{p}=0\).
\end{proof}

\begin{proposition}
  \label{prop:se}\(\tilde{\bfitgreek{\Pi}}^{\mathrm{Se}}_{k}\) and \(\bfitgreek{\Pi}^{\mathrm{Se}}_{k}\) are well-defined.
\end{proposition}

\begin{proof}
  The proof for \(\tilde{\bfitgreek{\Pi}}^{\mathrm{Se}}_{k}\) can be found in \cite[Sections 5 and 6]{beiraodaveigaVirtualElementApproximation2017}.

  Replacing \(\boldsymbol{v}-\bfitgreek{\Pi}^{\mathrm{Se}}_{k}\boldsymbol{v}\) in \cref{eq:se} by \(\boldsymbol{p}\in \mathbb{P}^3_{k}(K)\), we need to check that these conditions imply \(\boldsymbol{p}=0\). By \cref{prop:sf}, \cref{eq:se2,eq:se4} imply that \(\nabla\times\boldsymbol{p}=0\). Combining \cref{eq:se1}, \(\boldsymbol{p}\) has zero tangent trace. Similarly, since the following equivalent curl-curl problem:
  \begin{equation}
    \begin{cases}
      \nabla\times\nabla\times\boldsymbol{s}=0                      & \text{in }K,          \\
      \left(\nabla\times\boldsymbol{s}\right)\times\boldsymbol{n}=0 & \text{on }\partial K,
    \end{cases}
  \end{equation}
  has the only solution \(\nabla\times\boldsymbol{s}=0\), \(\boldsymbol{p}\) is not a curl of a polynomial. By \cref{eq:pdcomp2}, we deduce that \(\boldsymbol{p}=\boldsymbol{x}_{K}r_{k-1}\) for some \(r_{k-1}\in\mathbb{P}_{k-1}(K)\). Finally \cref{eq:se3} with sufficiently large \(k_{\mathrm{d}}\) ensures \(\boldsymbol{p}=0\).
\end{proof}

\begin{proposition}
  \(\tilde{\bfitgreek{\Pi}}^{\mathrm{Sn}}_{k}\) and \(\bfitgreek{\Pi}^{\mathrm{Sn}}_{k}\) are well-defined.
\end{proposition}

\begin{proof}
  See \cite[Proposition 5.1]{beiraodaveigaVirtualElementApproximation2017}, and proof there can be easily adapted to our definition.
\end{proof}

It is worth noting that there is much flexibility in choosing the orders of the DoFs for each VEM space. For example, the lowest order \(\mathbf{H}(\operatorname{div})\)-conforming VEM space with \(k,k_{\mathrm{d}}=k,k_{\mathrm{r}}=k-1\) resembles the Raviart-Thomas element, while the choice of \(k,k_{\mathrm{d}}=k-1,k_{\mathrm{r}}=k-1\) resembles the Brezzi-Douglas-Marini element. Similar analogies can be drawn for the \(\mathbf{H}(\operatorname{\mathbf{curl}})\)-conforming and \(\mathrm{H}^1\)-conforming VEM spaces with appropriate choices of DoF orders. Without loss of generality, we will focus on the following commonly used choices in defining serendipity spaces. Note that if \(k=0\), \(k_{\mathrm{d}}\) for \(\mathbf{V}^{\mathrm{f}}\), \(\beta_{\mathrm{r}}\) for \(\mathbf{V}^{\mathrm{e}}\) should also be set to \(0\) instead of \(-1\), to ensure the linear independence of the boundary DoFs.

\begin{definition}[Serendipity Face Virtual Element Spaces]
  \label{def:vsf}Let \(k,k_{\mathrm{r}}\) identify \(\bfitgreek{\Pi}^{\mathrm{Sf}}_{k}\), then define \(\mathbf{V}_{k,l}^{\mathrm{Sf}}(K)\) as:
  \begin{equation}\label{eq:equivsf}
    \left\{\boldsymbol{v} \in \mathbf{V}_{k, k-1, l-1}^{\mathrm{f}}(K)\mid\left(\boldsymbol{v}-\bfitgreek{\Pi}^{\mathrm{Sf}}_{k}\boldsymbol{v}, \boldsymbol{x}_{K} \times \boldsymbol{p}_{k_{\mathrm{r}}}\right)_{K}=0\ \forall \boldsymbol{p}_{k_{\mathrm{r}}}\in\mathbb{P}_{k_{\mathrm{r}}+1:l-1}^3(K)\right\}.
  \end{equation}
\end{definition}

As for the edge space, to ensure boundary continuity, the serendipity boundary space must be independently defined on each face.

\begin{definition}[Serendipity Edge Virtual Element Space]
  \label{def:vse}Let \(\beta=k,\beta_{\mathrm{r}}=k-1,\beta_{\mathrm{d}}\) identify \(\tilde{\bfitgreek{\Pi}}^{\mathrm{Se}}_{k}\) on each face \(F\subseteq \partial K\), then define \(\mathbf{V}^{\mathrm{Se}}_{k,l}(F)\) as:
  \begin{equation}\label{eq:equivsef}
    \left\{\boldsymbol{v} \in \mathbf{V}^{\mathrm{e}}_{k,l-1,k-1}\left(F\right)\mid \left(\boldsymbol{v}-\tilde{\bfitgreek{\Pi}}^{\mathrm{Se}}_{k}\boldsymbol{v}, \boldsymbol{x}_{F} p_{\beta_{\mathrm{d}}}\right)=0\ \forall p_{\beta_{\mathrm{d}}}\in\mathbb{P}_{\beta_{\mathrm{d}}+1:l-1}(F)\right\}.
  \end{equation}

  In the body, let \(\boldsymbol{\beta}, k_{\mathrm{d}}, \mu_{\mathrm{r}}\) identify \(\bfitgreek{\Pi}^{\mathrm{Se}}_{k}\), then define \(\mathbf{V}^{\mathrm{Se}}_{k,l}\) as:
  \begin{equation}\label{eq:equivse}
    \begin{aligned}
      &\left\{\boldsymbol{v}\in \mathbf{V}_{\boldsymbol{\beta},k_{\mathrm{d}},\mu_{\mathrm{r}}}^{\mathrm{e}}(K)\mid \boldsymbol{v}_h^F \in \mathbf{V}^{\mathrm{Se}}_{k,l_F}(F)\ \forall F \subseteq \partial K,\right.            \\
                                                              & \ \left(\boldsymbol{v}-\bfitgreek{\Pi}^{\mathrm{Se}}_{k}\boldsymbol{v}, \boldsymbol{x}_{K} p_{k_{\mathrm{d}}}\right)_K=0\ \forall p_{k_{\mathrm{d}}}\in\mathbb{P}_{k_{\mathrm{d}}+1:l-1}(K), \\
      &\ \left.\left(\nabla \times\left(\boldsymbol{v}-\bfitgreek{\Pi}^{\mathrm{Se}}_{k}\boldsymbol{v}\right),\boldsymbol{x}_K\times \boldsymbol{p}_{\mu_{\mathrm{r}}}\right)_K=0\ \forall \boldsymbol{p}_{\mu_{\mathrm{r}}}\in \mathbb{P}^3_{\mu_{\mathrm{r}}+1:l-2}(K)\right\}.
    \end{aligned}
  \end{equation}
\end{definition}

\begin{definition}[Serendipity Node Virtual Element Space]
  \label{def:vsn}Define \(\mathbf{V}_{k, l}^{\mathrm{Sn}}(K)\) as:
  \begin{equation}
    \begin{aligned}
      &\left\{v \in \mathbf{V}^{\mathrm{n}}_{k, l-2, l-2}(K)\mid \nabla v \in \mathbf{V}_{k-1,l-1}^{\mathrm{Se}}\left(K\right)\right\}.
    \end{aligned}
  \end{equation}
\end{definition}

Now we will show the interpolation properties of these serendipity
spaces. The proof is just reiterating the proof of the standard virtual element space interpolation properties presented in \cite{daveigaInterpolationStabilityEstimates2022}, with our serendipity projector replacing the polynomial projector in the proof, and finally applying the interpolation properties of the serendipity projector, which will be shown later.

First, we recall some essential properties of the virtual element spaces.

\begin{theorem}[Exact Commuting Diagrams]
  \label{thm:excom}In an element \(K\), we have the following exact commuting diagrams:
  \begin{equation}
    \begin{tikzcd}
      \mathrm{H}^1 \arrow[r, "\nabla"] \arrow[d, "\boldsymbol{I}^{\mathrm{Sn}}_{k}"] & \mathbf{H}(\operatorname{\mathbf{curl}})  \arrow[r, "\nabla\times"] \arrow[d, "\boldsymbol{I}^{\mathrm{Se}}_{k-1}"]
      & \mathbf{H}(\operatorname{div})   \arrow[r, "\nabla\cdot"] \arrow[d, "\boldsymbol{I}^{\mathrm{Sf}}_{k-2}"]
      & \mathrm{L}^2              \arrow[d, "\bfitgreek{\Pi}_{k-3}"] \\
      \mathbf{V}^{\mathrm{Sn}}_{k} \arrow[r, "\nabla"] \arrow[d, "\bfitgreek{\Pi}^{\mathrm{Sn}}_{k}"]               & \mathbf{V}^{\mathrm{Se}}_{k-1}               \arrow[r, "\nabla\times"] \arrow[d, "\bfitgreek{\Pi}^{\mathrm{Se}}_{k-1}"]
      & \mathbf{V}^{\mathrm{Sf}}_{k-2}              \arrow[r, "\nabla\cdot"] \arrow[d, "\bfitgreek{\Pi}^{\mathrm{Sf}}_{k-2}"]
      & \mathbb{P}_{k-3} \arrow[d, "\boldsymbol{E}"]\\
      \mathbb{P}_{k} \arrow[r, "\nabla"]                & \mathbb{P}^{3}_{k-1}               \arrow[r, "\nabla\times"]
      & \mathbb{P}^{3}_{k-2}              \arrow[r, "\nabla\cdot"]
      & \mathbb{P}_{k-3}\\
    \end{tikzcd}
  \end{equation}
\end{theorem}

\begin{proof}
  See \cite[Theorem 8.2]{daveigaDivCurlConforming2014} for the exact commuting properties of the virtual element spaces. The commuting properties of the serendipity projectors follow trivially from the definition of the serendipity projectors, i.e., \cref{def:vsn,def:vse,def:vsf}.
\end{proof}

\begin{theorem}[Inverse Estimates]
  \label{thm:invest}
  There exists some \(1/2<s\leq 1\) depending on the shape regularity of \(K\) such that the following inverse estimates hold true:
  \begin{subequations}
    \begin{align}
       & \left\|\nabla \times \boldsymbol{v}\right\|_{K}\lesssim h^{-1}_{K}\left\|\boldsymbol{v}\right\|_{K} & \forall \boldsymbol{v}\in \mathbf{V}^{\mathrm{Se}}_{k,l}\left(K\right), \label{eq:invest1}                                                          \\
       & \left\|\nabla \cdot \boldsymbol{v}\right\|_{K}\lesssim h^{-1}_{K}\left\|\boldsymbol{v}\right\|_{K}  & \forall \boldsymbol{v}\in \mathbf{V}^{\mathrm{Sf}}_{k,l}\left(K\right), \label{eq:invest2}                                                          \\
       & |\boldsymbol{v}|_{s,K} \lesssim h^{-s}_{K}\left\|\boldsymbol{v}\right\|_{K}                         & \forall \boldsymbol{v}\in \mathbf{V}^{\mathrm{Sn}}_{k,l}\left(K\right) \text{ or } \mathbf{V}^{\mathrm{Sf}}_{k,l}\left(K\right). \label{eq:invest3}
    \end{align}
  \end{subequations}
\end{theorem}

\begin{proof}
  For the first two estimates, see the proof of \cite[(30)]{daveigaInterpolationStabilityEstimates2022}, and the proof can be easily adapted to our definition. Indeed, the bubble function technique is applicable to any finite dimensional space. Since for \(\boldsymbol{v} \in \mathbf{V}^{\mathrm{Sf}}_{k,l}(K)\), \(\nabla \cdot \boldsymbol{v}\) and \(\nabla \times \boldsymbol{v}\) are both polynomials, the inverse estimates for them hold true by the standard polynomial inverse estimates. For \(\boldsymbol{v} \in \mathbf{V}^{\mathrm{Se}}_{k,l}(K)\), by noting \(\nabla \times \boldsymbol{v} \in \mathbf{V}^{\mathrm{Sf}}_{k-1,l-1}(K)\), and applying the previous inverse estimates, the inverse estimates for \(\nabla \times \boldsymbol{v}\) hold true.

  For the third estimate, see the proof of \cite[Lemma 3]{beiraodaveigaVirtualElementsMaxwell2022}. Just consider the auxiliary Laplacian or curl-curl problem and prove by using the regularity of the solution, embedding properties of \(\mathbf{H}\left(\operatorname{div},K\right)\cap \mathbf{H}\left(\operatorname{\mathbf{curl}},K\right)\), and previous inverse estimates.
\end{proof}

\begin{theorem}[Basic Bounds]
  The following bounds hold true\cite[(77, 105)]{daveigaInterpolationStabilityEstimates2022}:

  For each \(\boldsymbol{v}\in \mathbf{V}_{k, k_{\mathrm{d}}, k_{\mathrm{r}}}^{\mathrm{f}}(K)\), we have:
  \begin{equation}\label{eq:boundf}
    \|\boldsymbol{v}\|_{K}\lesssim h_{K}\left\|\nabla \cdot \boldsymbol{v}\right\|_{K}+h^{\frac{1}{2}}_{K}\|\boldsymbol{v}\cdot \boldsymbol{n}\|_{\partial K}+\sup_{\boldsymbol{p}_{k_{\mathrm{r}}}\in \mathbb{P}^{3}_{k_{\mathrm{r}}}\left(K\right)}\frac{\left(\boldsymbol{v},\boldsymbol{x}_{K} \times \boldsymbol{p}_{k_{\mathrm{r}}}\right)_{K}}{\left\|\boldsymbol{x}_{K} \times \boldsymbol{p}_{k_{\mathrm{r}}}\right\|_{K}}.
  \end{equation}

  And for each \(\boldsymbol{v}\in \mathbf{V}_{\boldsymbol{\beta},k_{\mathrm{d}},\mu_{\mathrm{r}}}^{\mathrm{e}}(K)\), we have:
  \begin{equation}\label{eq:bounde}
    \begin{aligned}
      \|\boldsymbol{v}\|_{K} & \lesssim\sum_{F\subseteq \partial K}\left(h^{\frac{3}{2}}_{K}\left\|\nabla \times \boldsymbol{v}\cdot \boldsymbol{n}\right\|_{F}+h_{K}\left\|\boldsymbol{v}\cdot \boldsymbol{t}\right\|_{\partial F}+\sup_{p_{\beta_{\mathrm{d}}}\in \mathbb{P}_{\beta_{\mathrm{d}}}\left(F\right)} \frac{h^{\frac{1}{2}}_{K}\left(\boldsymbol{v},\boldsymbol{x}_{F}p_{\beta_{\mathrm{d}}}\right)_{F}}{\left\|\boldsymbol{x}_{F}p_{\beta_{\mathrm{d}}}\right\|_{F}}\right)                                             \\
                             & \quad+\sup_{\boldsymbol{p}_{\mu_{\mathrm{r}}}\in \mathbb{P}^{3}_{\mu_{\mathrm{r}}}\left(K\right)}\frac{\left(\nabla \times \boldsymbol{v},\boldsymbol{x}_{K} \times \boldsymbol{p}_{\mu_{\mathrm{r}}}\right)_{K}}{\left\|\boldsymbol{x}_{K} \times \boldsymbol{p}_{\mu_{\mathrm{r}}}\right\|_{K}}+\sup_{p_{k_{\mathrm{d}}}\in \mathbb{P}_{k_{\mathrm{d}}}\left(K\right)}\frac{\left(\boldsymbol{v},\boldsymbol{x}_K p_{k_{\mathrm{d}}}\right)_{K}}{\left\|\boldsymbol{x}_K p_{k_{\mathrm{d}}}\right\|_{K}}.
    \end{aligned}
  \end{equation}
\end{theorem}

Next, we define some norms induced by the serendipity projectors:
\begin{definition}[Serendipity Norms]
  Define the following norms on \(\mathbb{P}^3_{k}(K)\):
  \begin{subequations}
    \begin{align}
       & \begin{aligned}
           \normmm{\boldsymbol{v}}_{\mathrm{Sf},k,K}:= & \alpha\sup_{p_{k-1}\in \mathbb{P}_{k-1}\left(K\right)} \frac{h_{K}\left(\nabla \cdot \boldsymbol{v},p_{k-1}\right)_{K}}{\left\|p_{k-1}\right\|_{K}}                                                                                                                                                  \\
                                                       & +\sup_{\boldsymbol{p}_{k+1}\in \mathbb{P}_{k+1}^{3}\left(K\right)} \frac{h^{\frac{1}{2}}_{K}\left(\boldsymbol{v}\cdot \boldsymbol{n},\nabla \times \boldsymbol{p}_{k+1}\cdot \boldsymbol{n}\right)_{\partial K}}{\left\|\nabla \times \boldsymbol{p}_{k+1}\cdot \boldsymbol{n}\right\|_{\partial K}} \\
                                                       & +\sup_{\boldsymbol{p}_{\mu_{\mathrm{r}}}\in \mathbb{P}_{\mu_{\mathrm{r}}}^{3}\left(K\right)}\frac{\left(\boldsymbol{v},\boldsymbol{x}_K \times \boldsymbol{p}_{\mu_{\mathrm{r}}}\right)_{K}}{\left\|\boldsymbol{x}_K \times \boldsymbol{p}_{\mu_{\mathrm{r}}}\right\|_{K}},
         \end{aligned}\label{eq:normvf}                                          \\
       & \begin{aligned}
           \normmm{\boldsymbol{v}}_{\mathrm{Se},k,K}:= & \beta h_{K}\normmm{\nabla \times \boldsymbol{v}}_{\mathrm{Sf},k-1,K}+\sup_{p_{k+1}\in \mathbb{P}_{k+1}\left(K\right)} \frac{h^{\frac{1}{2}}_{K}\left(\boldsymbol{v}\times\boldsymbol{n},\nabla p_{k+1}\times\boldsymbol{n}\right)_{\partial K}}{\left\|\nabla p_{k+1}\times\boldsymbol{n}\right\|_{\partial K}} \\
                                                       & +\sup_{p_{k_{\mathrm{d}}}\in \mathbb{P}_{k_{\mathrm{d}}}\left(K\right)}\frac{\left(\boldsymbol{v},\boldsymbol{x}_K p_{k_{\mathrm{d}}}\right)_{K}}{\left\|\boldsymbol{x}_K p_{k_{\mathrm{d}}}\right\|_{K}},
         \end{aligned}\label{eq:normve}
    \end{align}
  \end{subequations}
  where \(\alpha\) and \(\beta\) are chosen such that the coefficients in \cref{eq:const1,eq:const2} are positive.

  Note that as functionals, they are meaningful for smooth enough functions.
\end{definition}

\begin{lemma}
  \label{lem:sfsubbound}The following bound holds true: \(\forall \boldsymbol{q}_{k+1}\in\mathbb{P}_{k+1}^{3}(K)\)
  \begin{equation}\label{eq:contrasf}
    \begin{aligned}
       & h^{\frac{1}{2}}_{K}\left\|\nabla \times \boldsymbol{q}_{k+1}\cdot \boldsymbol{n}\right\|_{\partial K}+\sup_{\boldsymbol{p}_{\mu_{\mathrm{r}}}\in \mathbb{P}_{\mu_{\mathrm{r}}}^{3}\left(K\right)}\frac{\left(\nabla \times \boldsymbol{q}_{k+1},\boldsymbol{x}_K \times \boldsymbol{p}_{\mu_{\mathrm{r}}}\right)_{K}}{\left\|\boldsymbol{x}_K \times \boldsymbol{p}_{\mu_{\mathrm{r}}}\right\|_{K}} \\
       & \gtrsim\left\|\nabla \times\boldsymbol{q}_{k+1}\right\|_{K}.
    \end{aligned}
  \end{equation}
\end{lemma}

\begin{proof}
  It suffices to prove the claim fixing \(h_{K}=1\), and then use a scaling argument. By the mesh assumption, there are concentric balls \(\tilde{B}\) and \(B\), with radius \(\rho\) and \(1+\rho\), satisfying:
  \begin{equation}
    \tilde{B}\subseteq K\subseteq B.
  \end{equation}

  Let \(S_{N_{\mathrm{v}}}\) denote the set of polytopes with \(N_{\mathrm{v}}\) vertices satisfying \cref{ass:mesh}. Setting the center of \(\tilde{B}\) and \(B\) at the origin, we can define an injective mapping \(\mathscr{M}:S_{N_{\mathrm{v}}}\to \mathbb{R}^{3N_{\mathrm{v}}}\) by:
  \begin{equation}
    K\mapsto\left[v_{1,x},v_{1,y},v_{1,z},\dots,v_{N_{\mathrm{v}},x},v_{N_{\mathrm{v}},y},v_{N_{\mathrm{v}},z}\right]^{\top},
  \end{equation}
  where \(v_{n,x},v_{n,y},v_{n,z}\) are the \(x,y,z\) coordinates of the \(n\)-th vertex. Note that we here assume the admissibility of the polytopes and well-defined order of the vertices, though it is not the case when \(S\) is only categorized by vertex number. But since the following proof is applicable to any further categorization of \(S_{N_{\mathrm{v}}}\), we resort to avoiding the tedious task.

  Now we will prove the claim by contradiction. If \cref{eq:contrasf} were false, there exists a sequence of polytopes \(\{K_{n}\in S_{N_{\mathrm{v}}}\}_{n\in \mathbb{N}}\), and a sequence of polynomials \(\{\nabla \times \boldsymbol{q}_{n}\in \nabla \times \mathbb{P}^{3}_{k+1}(B)\}_{n\in \mathbb{N}}\), such that:
  \begin{equation}\label{eq:pfcontrasf}
    \|\nabla \times \boldsymbol{q}_{n}\|_{K_{n}}=1,
    \|\nabla \times \boldsymbol{q}_{n}\cdot \boldsymbol{n}\|_{\partial K_{n}}\le\frac{1}{n},
    \sup_{\boldsymbol{p}_{\mu_{\mathrm{r}}}\in\mathbb{P}^{3}_{\mu_{\mathrm{r}}}\left(K_n\right)}
    \frac{(\nabla \times \boldsymbol{q}_{n},\,\boldsymbol{x}_K\times\boldsymbol{p}_{\mu_{\mathrm{r}}})_{K_{n}}}
    {\|\boldsymbol{x}_K\times\boldsymbol{p}_{\mu_{\mathrm{r}}}\|_{K_{n}}}\le\frac{1}{n}
    \ \forall n\in\mathbb{N}.
  \end{equation}

  Since \(\mathscr{M}(S_{N_{\mathrm{v}}})\) is bounded and closed, there exists a subsequence \(\{\mathscr{M}(K_{n_{i}})\}_{i\in \mathbb{N}}\) converging to \(\mathscr{M}(K)\), which corresponds to \(K\) as the limit of \(\{K_{n_{i}}\}_{i\in \mathbb{N}}\) in the sense that all vertex coordinates converge. Since:
  \begin{equation}
    \|\nabla \times \boldsymbol{q}_{n}\|_{\tilde{B}}\le\|\nabla \times \boldsymbol{q}_{n}\|_{K_{n}}=1
  \end{equation}

  is bounded, there exists another subsequence of \(\{\nabla \times \boldsymbol{q}_{n_{i}}\}_{i\in \mathbb{N}}\) converging to \(\nabla \times \boldsymbol{q}\).

  By standard polynomial properties, we have:
  \begin{equation}
    1=\|\nabla \times \boldsymbol{q}\|_{K}\le\|\nabla \times \boldsymbol{q}\|_{B}\lesssim\|\nabla \times \boldsymbol{q}\|_{\tilde{B}},
  \end{equation}
  which implies \(\nabla \times \boldsymbol{q}\neq 0\) in \(\tilde{B}\). But, following the same argument as in the proof to \cref{prop:sf}, the fact that last two terms in \cref{eq:pfcontrasf} diminish, implies \(\nabla \times \boldsymbol{q} =  0\) in \(K\) hence in \(\tilde{B}\). This is a contradiction, and thus the claim is proved.
\end{proof}

\begin{lemma}
  \label{lem:sfbound} For given \(s> 1/2\), the following bounds hold true:
  \begin{subequations}
    \begin{align}
       & \normmm{\boldsymbol{s}_{k}}_{\mathrm{Sf},k,K}\gtrsim\left\|\boldsymbol{s}_{k}\right\|_K                                                                     & \forall \boldsymbol{s}_{k}\in\mathbb{P}_{k}^{3}(K),\label{eq:lsf1}                                                     \\
       & \normmm{\boldsymbol{v}_{h}}_{\mathrm{Sf},k,K}\lesssim\left\|\boldsymbol{v}_{h}\right\|_{K}                                                                  & \forall \boldsymbol{v}_{h}\in\mathbf{V}^{\mathrm{Sf}}_{k, l}\left(K\right),\label{eq:lsf2}                             \\
       & \normmm{\boldsymbol{v}}_{\mathrm{Sf},k,K}\lesssim\|\boldsymbol{v}\|_{K}+h^{s}_{K}|\boldsymbol{v}|_{s,K}+h_{K}\left\|\nabla \cdot \boldsymbol{v}\right\|_{K} & \forall \boldsymbol{v}\in \mathbf{H}^{s}\left(K\right)\cap \mathbf{H}\left(\operatorname{div},K\right).\label{eq:lsf3}
    \end{align}
  \end{subequations}
\end{lemma}

\begin{proof}
  We first prove \cref{eq:lsf1}. Since \(\boldsymbol{s}_{k}\in\mathbb{P}_{k}^{3}(K)\). Decompose \(\boldsymbol{s}_{k}=\nabla \times \boldsymbol{q}_{k+1}+\boldsymbol{x}r_{k-1}\) for some \(\boldsymbol{q}_{k+1}\in\mathbb{P}_{k+1}^{3}(K)\) and \(r_{k-1}\in\mathbb{P}_{k-1}(K)\) and supply the supremum in \cref{eq:normvf} with the corresponding components, we have:
  \begin{equation}
    \begin{aligned}
      \normmm{\boldsymbol{s}_{k}}_{\mathrm{Sf},k,K} & \geq\alpha h_{K}\left\|\nabla \cdot \boldsymbol{s}_{k}\right\|_{K}+\frac{h^{\frac{1}{2}}_{K}\left(\left(\nabla \times \boldsymbol{q}_{k+1}+\boldsymbol{x}_K r_{k-1}\right)\cdot \boldsymbol{n},\nabla \times \boldsymbol{q}_{k+1}\cdot \boldsymbol{n}\right)_{\partial K}}{\left\|\nabla \times \boldsymbol{q}_{k+1}\cdot \boldsymbol{n}\right\|_{\partial K}} \\
                                                    & \quad+\sup_{\boldsymbol{p}_{\mu_{\mathrm{r}}}\in \mathbb{P}_{\mu_{\mathrm{r}}}^{3}\left(K\right)}\frac{\left(\nabla \times \boldsymbol{q}_{k+1},\boldsymbol{x}_K\times \boldsymbol{p}_{\mu_{\mathrm{r}}}\right)_{K}}{\left\|\boldsymbol{x}_K \times \boldsymbol{p}_{\mu_{\mathrm{r}}}\right\|_{K}}                                                                            \\
                                                    & \geq\alpha h_{K}\left\|\nabla \cdot \boldsymbol{s}_{k}\right\|_{K}+h^{\frac{1}{2}}_{K}\left(\left\|\nabla \times \boldsymbol{q}_{k+1}\cdot \boldsymbol{n}\right\|_{\partial K}-\left\|\boldsymbol{x}_K r_{k-1}\cdot \boldsymbol{n}\right\|_{\partial K}\right)                                                                                                 \\
                                                    & \quad+\sup_{\boldsymbol{p}_{\mu_{\mathrm{r}}}\in \mathbb{P}_{\mu_{\mathrm{r}}}^{3}\left(K\right)}\frac{\left(\nabla \times \boldsymbol{q}_{k+1},\boldsymbol{x}_K\times \boldsymbol{p}_{\mu_{\mathrm{r}}}\right)_{K}}{\left\|\boldsymbol{x}_K \times \boldsymbol{p}_{\mu_{\mathrm{r}}}\right\|_{K}}.
    \end{aligned}
  \end{equation}

  Now we handle each term separately. First we have:
  \begin{equation}\label{eq:investdiv}
    h^{\frac{1}{2}}_{K}\left\|\boldsymbol{x}_K r_{k-1}\cdot \boldsymbol{n}\right\|_{\partial K}\stackrel{\cref{eq:polyinv1}}{\lesssim} \left\|\boldsymbol{x}_K r_{k-1}\right\|_{K}\stackrel{\cref{eq:iso3}}{\lesssim} h_{K}\left\|\nabla \cdot\left(\boldsymbol{x}_K r_{k-1}\right)\right\|_{K}=h_{K}\left\|\nabla \cdot \boldsymbol{s}_{k}\right\|_{K}.
  \end{equation}

  This, and \cref{lem:sfsubbound} lead to:
  \begin{equation}\label{eq:const1}
    \begin{aligned}
      \normmm{\boldsymbol{s}_{k}}_{\mathrm{Sf},k,K} & \geq\left(\gamma-C_{1}\right)h_{K}\left\|\nabla \cdot \boldsymbol{s}_{k}\right\|_{K}+C_{2}\left\|\nabla \times \boldsymbol{q}_{k+1}\right\|_{K} \\
                                                    & \gtrsim\left\|\boldsymbol{x}_K r_{k-1}\right\|_{K}+\left\|\nabla \times \boldsymbol{q}_{k+1}\right\|_{K}                                           \\
                                                    & \gtrsim\left\|\boldsymbol{s}_{k}\right\|_{K}.
    \end{aligned}
  \end{equation}

  We prove \cref{eq:lsf2} as follows:
  \begin{equation}
    \begin{aligned}
      \normmm{\boldsymbol{v}_{h}}_{\mathrm{Sf},k,K} & \stackrel{\cref{eq:normvf}}{\lesssim} h^{\frac{1}{2}}_{K}\left\|\boldsymbol{v}_{h}\cdot \boldsymbol{n}\right\|_{\partial K}+\left\|\boldsymbol{v}_{h}\right\|_{K}+h_{K}\left\|\nabla \cdot \boldsymbol{v}_{h}\right\|_{K} \\
                                                    & \stackrel{\cref{eq:polyinv2},\cref{eq:vtrace2}}{\lesssim}\left\|\boldsymbol{v}_{h}\right\|_{K}+h_{K}\left\|\nabla \cdot \boldsymbol{v}_{h}\right\|_{K}                                                                    \\
                                                    & \stackrel{\cref{eq:invest2}}{\lesssim}\left\|\boldsymbol{v}_{h}\right\|_{K}.
    \end{aligned}
  \end{equation}

  Finally, we prove \cref{eq:lsf3} as follows:
  \begin{equation}
    \begin{aligned}
      \normmm{\boldsymbol{v}}_{\mathrm{Sf},k,K} & \stackrel{\cref{eq:normvf}}{\lesssim} h^{\frac{1}{2}}_{K}\left\|\boldsymbol{v}\cdot \boldsymbol{n}\right\|_{\partial K}+\|\boldsymbol{v}\|_{K}+h_{K}\left\|\nabla \cdot \boldsymbol{v}\right\|_{K} \\
                                                & \stackrel{\cref{eq:trace1}}{\lesssim}\|\boldsymbol{v}\|_{K}+h^{s}_{K}|\boldsymbol{v}|_{s,K}+h_{K}\left\|\nabla \cdot \boldsymbol{v}\right\|_{K}.
    \end{aligned}
  \end{equation}
\end{proof}

\begin{theorem}
  \label{thm:errorsf}For each \(\boldsymbol{v} \in \mathbf{H}^s(K)\cap \mathbf{H}(\operatorname{div},K), 1/2 <s \leq k+1\), we have
  \begin{equation}
    \left\|\boldsymbol{v}-\bfitgreek{\Pi}^{\mathrm{Sf}}_{k} \boldsymbol{v}\right\|_K \lesssim h_K^s|\boldsymbol{v}|_{s, K}+h_K\left\|\nabla\cdot \boldsymbol{v}\right\|_K .
  \end{equation}

  The second term on the right-hand side can be neglected if \(s \geq 1\).
\end{theorem}

\begin{proof}
  For any \(\boldsymbol{p}_{k}\) in \(\mathbb{P}_{k}^{3}(K),k=\max\left\{0,[s]\right\}\), by applying \cref{lem:sfbound}, and note that \(\normmm{\bfitgreek{\Pi}^{\mathrm{Sf}}_{k} \cdot}_{\mathrm{Sf},k,K}=\normmm{\cdot}_{\mathrm{Sf},k,K}\), we have:
  \begin{equation}
    \begin{aligned}
      \left\|\boldsymbol{v}-\bfitgreek{\Pi}^{\mathrm{Sf}}_{k} \boldsymbol{v}\right\|_K & \lesssim\left\|\boldsymbol{v}-\boldsymbol{p}_{k}\right\|_{K}+\left\|\bfitgreek{\Pi}^{\mathrm{Sf}}_{k}\left(\boldsymbol{p}_{k}-\boldsymbol{v}\right)\right\|_{K}                                               \\
                                                                                       & \lesssim\left\|\boldsymbol{v}-\boldsymbol{p}_{k}\right\|_{K}+\normmm{\bfitgreek{\Pi}^{\mathrm{Sf}}_{k}\left(\boldsymbol{p}_{k}-\boldsymbol{v}\right)}_{\mathrm{Sf},k,K}                                       \\
                                                                                       & =\left\|\boldsymbol{v}-\boldsymbol{p}_{k}\right\|_{K}+\normmm{\boldsymbol{p}_{k}-\boldsymbol{v}}_{\mathrm{Sf},k,K}                                                                                            \\
                                                                                       & \lesssim\left\|\boldsymbol{v}-\boldsymbol{p}_{k}\right\|_{K}+h^{s}_{K}\left|\boldsymbol{v}-\boldsymbol{p}_{k}\right|_{s,K}+h_{K}\left\|\nabla \cdot\left(\boldsymbol{v}-\boldsymbol{p}_{k}\right)\right\|_{K} \\
    \end{aligned}
  \end{equation}

  If \(s \geq 1\), the standard polynomial approximation estimates yield:
  \begin{equation}
    \left\|\boldsymbol{v}-\bfitgreek{\Pi}^{\mathrm{Sf}}_{k} \boldsymbol{v}\right\|_K \lesssim h_K^s|\boldsymbol{v}|_{s, K}.
  \end{equation}

  Instead if \(1/2 < s < 1\), take \(\boldsymbol{p}_{0}=\bfitgreek{\Pi}_{0}\boldsymbol{v}\). By Poincaré inequality, we have:
  \begin{equation}
    \left\|\boldsymbol{v}-\bfitgreek{\Pi}^{\mathrm{Sf}}_{k} \boldsymbol{v}\right\|_K \lesssim h_K^s|\boldsymbol{v}|_{s, K}+h_K\left\|\nabla\cdot \boldsymbol{v}\right\|_K .
  \end{equation}
\end{proof}

\begin{theorem}[Interpolation Properties of Serendipity Face Virtual
    Element Spaces]
  \label{thm:vf}
  Let \(\boldsymbol{I}^{\mathrm{Sf}}_{k}\) be the interpolator to \(\mathbf{V}_{k, l}^{\mathrm{Sf}}(K)\),
  for each \(\boldsymbol{v} \in \mathbf{H}^s(K)\cap\mathbf{H}^r(\operatorname{div},K), 1/2<s\leq k+1,0 \leq r\leq k\), we have:
  \begin{subequations}
    \begin{align}
       & \left\|\boldsymbol{v}-\boldsymbol{I}^{\mathrm{Sf}}_{k} \boldsymbol{v}\right\|_K \lesssim h_K^s|\boldsymbol{v}|_{s, K}+h_K\|\nabla \cdot \boldsymbol{v}\|_K,\label{eq:thmvf1}               \\
       & \left\|\nabla \cdot\left(\boldsymbol{v}-\boldsymbol{I}^{\mathrm{Sf}}_{k} \boldsymbol{v}\right)\right\|_K \lesssim h_K^r \left|\nabla \cdot \boldsymbol{v}\right|_{r, K} .\label{eq:thmvf2}
    \end{align}
  \end{subequations}
  The second term on the right-hand side of \cref{eq:thmvf1} can be
  neglected if \(s \geq 1\).
\end{theorem}

\begin{proof}
  Inequality \cref{eq:thmvf2} follows from \cref{thm:excom} and standard polynomial approximation
  properties.

  Inequality \cref{eq:thmvf1} follows from the same argument as in the proof of \cite[Theorem 4.2]{daveigaInterpolationStabilityEstimates2022}, with \(\mathit{\Pi}^{0,E}_{k}\) replaced by \(\bfitgreek{\Pi}^{\mathrm{Sf}}_{k}\) and \cref{thm:errorsf} applied.
\end{proof}

In the subsequent proofs, we will adopt a more concise exposition for sections analogous to those presented earlier.

\begin{lemma}
  \label{lem:sesubbound}The following bound holds true:
  \begin{equation}
    \begin{aligned}
       & h^{\frac{1}{2}}_{K}\left\|\nabla q_{k+1}\times \boldsymbol{n}\right\|_{\partial K}+\sup_{p_{k_{\mathrm{d}}}\in \mathbb{P}_{k_{\mathrm{d}}}\left(K\right)}\frac{\left(\nabla q_{k+1},\boldsymbol{x}p_{k_{\mathrm{d}}}\right)_{K}}{\left\|\boldsymbol{x}p_{k_{\mathrm{d}}}\right\|_{K}} \\
       & \gtrsim\left\|\nabla q_{k+1}\right\|_{K}.
    \end{aligned}
  \end{equation}
\end{lemma}

\begin{proof}
  Similar to the proof of \cref{lem:sfsubbound}.
\end{proof}

\begin{lemma}
  \label{lem:sebound}For given \(s,r> 1/2\), the following bound holds true:
  \begin{subequations}
    \begin{align}
       & \normmm{\boldsymbol{s}_{k}}_{\mathrm{Se},k,K}\gtrsim\left\|\boldsymbol{s}_{k}\right\|_{K}                                            & \forall \boldsymbol{s}_{k}\in \mathbb{P}^{3}_{k}\left(K\right),\label{eq:lse1}             \\
       & \normmm{\boldsymbol{v}_{h}}_{\mathrm{Se},k,K}\lesssim\left\|\boldsymbol{v}_{h}\right\|_{K}                                           & \forall \boldsymbol{v}_{h}\in \mathbf{V}^{\mathrm{Se}}_{k,l}\left(K\right),\label{eq:lse2} \\
       &
      \begin{aligned}
        \normmm{\boldsymbol{v}}_{\mathrm{Se},k,K}\lesssim & \|\boldsymbol{v}\|_{K}+h^{s}_{K}|\boldsymbol{v}|_{s,K}                                                       \\
                                                          & +h_{K}\left\|\nabla \times \boldsymbol{v}\right\|+h_{K}^{r+1}\left|\nabla \times \boldsymbol{v}\right|_{r,K}
      \end{aligned}
       & \forall \boldsymbol{v}\in \mathbf{H}^{s}\left(K\right)\cap \mathbf{H}^{r}\left(\operatorname{\mathbf{curl}},K\right).\label{eq:lse3}
    \end{align}
  \end{subequations}
\end{lemma}

\begin{proof}
  We first prove \cref{eq:lse1}. Since \(\boldsymbol{s}_{k}\in\mathbb{P}^{3}_{k}(K)\), decompose \(\boldsymbol{s}_{k}=\nabla q_{k+1}+\boldsymbol{x}\times \boldsymbol{r}_{k-1}\) for some \(\boldsymbol{q}_{k+1}\in\mathbb{P}^{3}_{k+1}(K)\) and \(r_{k-1}\in\mathbb{P}_{k-1}(K)\), and supply the supremum in \cref{eq:normve} with the corresponding components, we have:
  \begin{equation}
    \begin{aligned}
      \normmm{\boldsymbol{s}_{k}} & \geq \beta h_{K}\normmm{\nabla \times \boldsymbol{s}_{k}}_{\mathrm{Sf},k-1,K}+\frac{h^{\frac{1}{2}}_{K}\left(\left(\nabla q_{k+1}+\boldsymbol{x}_K \times \boldsymbol{r}_{k-1}\right)\times \boldsymbol{n},\nabla q_{k+1}\times \boldsymbol{n}\right)_{\partial K}}{\left\|\nabla q_{k+1}\times \boldsymbol{n}\right\|_{\partial K}} \\
                                  & \quad+\sup_{p_{k_{\mathrm{d}}}\in \mathbb{P}_{k_{\mathrm{d}}}\left(K\right)}\frac{\left(\nabla q_{k+1},\boldsymbol{x}_K p_{k_{\mathrm{d}}}\right)_{K}}{\left\|\boldsymbol{x}_K p_{k_{\mathrm{d}}}\right\|_{K}}                                                                                                                                    \\
                                  & \geq  \beta h_{K}\normmm{\nabla \times \boldsymbol{s}_{k}}_{\mathrm{Sf},k-1,K}+h^{\frac{1}{2}}_{K}\left(\left\|\nabla q_{k+1}\times \boldsymbol{n}\right\|_{\partial K}-\left\|\boldsymbol{x}_K \times \boldsymbol{r}_{k-1}\times \boldsymbol{n}\right\|_{\partial K}\right)                                                        \\
                                  & \quad+\sup_{p_{k_{\mathrm{d}}}\in \mathbb{P}_{k_{\mathrm{d}}}\left(K\right)}\frac{\left(\nabla q_{k+1},\boldsymbol{x}_K p_{k_{\mathrm{d}}}\right)_{K}}{\left\|\boldsymbol{x}_K p_{k_{\mathrm{d}}}\right\|_{K}}.
    \end{aligned}
  \end{equation}

  We have:
  \begin{equation}
    \begin{aligned}
      h^{\frac{1}{2}}_{K}\left\|\boldsymbol{x}_K \times \boldsymbol{r}_{k-1}\times \boldsymbol{n}\right\|_{\partial K} & \stackrel{\cref{eq:polyinv1}}{\lesssim} \left\|\boldsymbol{x}_K \times \boldsymbol{r}_{k-1}\right\|_{K}                            \\
                                                                                                                     & \stackrel{\cref{eq:iso2}}{\lesssim} h_{K}\left\|\nabla \times\left(\boldsymbol{x}_K \times \boldsymbol{r}_{k-1}\right)\right\|_{K} \\
                                                                                                                     & =h_{K}\left\|\nabla \times \boldsymbol{s}_{k}\right\|_{K}.
    \end{aligned}
  \end{equation}
  This, and \cref{lem:sesubbound} lead to:
  \begin{equation}\label{eq:const2}
    \begin{aligned}
      \normmm{\boldsymbol{s}_{k}}_{\mathrm{Se},k,K} & \stackrel{\cref{eq:lsf1}}{\geq} C_{1}\left(\beta-C_{2}\right)h_{K}\left\|\nabla \times \boldsymbol{s}_{k}\right\|_{K}+C_{3}\left\|\nabla q_{k+1}\right\|_{K} \\
                                                    & \gtrsim\left\|\boldsymbol{x}_K \times \boldsymbol{r}_{k-1}\right\|_{K}+\left\|\nabla q_{k+1}\right\|_{K}                                                        \\
                                                    & \gtrsim\left\|\boldsymbol{s}_{k}\right\|_{K}.
    \end{aligned}
  \end{equation}

  We prove \cref{eq:lse2} as follows:
  \begin{equation}
    \begin{aligned}
      \normmm{\boldsymbol{v}_{h}}_{\mathrm{Se},k,K} & \stackrel{\cref{eq:normve},\cref{eq:lsf2}}{\lesssim} h^{\frac{1}{2}}_{K}\left\|\boldsymbol{v}_{h}\times \boldsymbol{n}\right\|_{\partial K}+\left\|\boldsymbol{v}_{h}\right\|_{K}+h_{K}\left\|\nabla \times \boldsymbol{v}_{h}\right\|_{K} \\
                                                    & \stackrel{\cref{eq:polyinv2},\cref{eq:vtrace3}}{\lesssim}\left\|\boldsymbol{v}_{h}\right\|_{K}+h_{K}\left\|\nabla \times \boldsymbol{v}_{h}\right\|_{K}                                                                                    \\
                                                    & \stackrel{\cref{eq:invest1}}{\lesssim}\left\|\boldsymbol{v}_{h}\right\|_{K}.
    \end{aligned}
  \end{equation}

  Finally, we prove \cref{eq:lse3} as follows:
  \begin{equation}
    \begin{aligned}
      \normmm{\boldsymbol{v}}_{\mathrm{Se},k,K} & \stackrel{\cref{eq:normve}}{\lesssim} h^{\frac{1}{2}}_{K}\left\|\boldsymbol{v}\times \boldsymbol{n}\right\|_{\partial K}+\|\boldsymbol{v}\|_{K}+h_{K}\normmm{\nabla \times \boldsymbol{v}}_{K}                   \\
                                                & \stackrel{\cref{eq:lsf3},\cref{eq:trace1}}{\lesssim}\|\boldsymbol{v}\|_{K}+h^{s}_{K}|\boldsymbol{v}|_{s,K}+h_{K}\left\|\nabla \times \boldsymbol{v}\right\|_{K}+h_{K}^{r+1}|\nabla \times \boldsymbol{v}|_{r,K}.
    \end{aligned}
  \end{equation}
\end{proof}

\begin{theorem}
  \label{thm:errorse}For each \(\boldsymbol{v} \in \mathbf{H}^s(K)\cap \mathbf{H}^r(\operatorname{\mathbf{curl}},K), 1/2 <s \leq k+1, 1/2<r\leq k\), we have:
  \begin{equation}
    \left\|\boldsymbol{v}-\bfitgreek{\Pi}^{\mathrm{Se}}_{k} \boldsymbol{v}\right\|_K \lesssim h_K^s|\boldsymbol{v}|_{s, K}+ h_K^{r+1}|\nabla\times \boldsymbol{v}|_{r, K} + h_K\left\|\nabla\times \boldsymbol{v}\right\|_K .
  \end{equation}
  On the right-hand side, the second term can be neglected if \(s \geq 3/2\), and the third term can be neglected if \(s \geq 1\).
\end{theorem}

\begin{proof}
  Similar to the proof of \cref{thm:errorsf}.
\end{proof}

\begin{theorem}[Interpolation Properties of Serendipity Edge Virtual
    Element Space]
  \label{thm:ve}
  Let \(\boldsymbol{I}^{\mathrm{Se}}_{k}\) be the interpolator to \(\mathbf{V}^{\mathrm{Se}}_{k,l}(K)\),
  for each \(\boldsymbol{v} \in \mathbf{H}^s(K)\cap \mathbf{H}^r(\operatorname{\mathbf{curl}},K)\), \(1/2 <s \leq k+1, 1/2<r\leq k\), we have:
  \begin{subequations}
    \begin{align}
       & \left\|\boldsymbol{v}-\boldsymbol{I}^{\mathrm{Se}}_{k} \boldsymbol{v}\right\|_K \lesssim h_K^s|\boldsymbol{v}|_{s, K}+h_K^{r+1}|\nabla\times \boldsymbol{v}|_{r, K}+h_K\|\nabla\times \boldsymbol{v}\|_K,\label{eq:thmve1} \\
       & \left\|\nabla\times\left(\boldsymbol{v}-\boldsymbol{I}^{\mathrm{Se}}_{k} \boldsymbol{v}\right)\right\|_K \lesssim h_K^r|\nabla\times \boldsymbol{v}|_{r, K} .\label{eq:thmve2}
    \end{align}
  \end{subequations}
  On the right-hand side of \cref{eq:thmve1}, the second term can be neglected if \(s \geq 3/2\), and the third term can be neglected if \(s \geq 1\).
\end{theorem}

\begin{proof}
  Inequality \cref{eq:thmve2} follows from \cref{thm:excom} and \cref{thm:vf}.

  Inequality \cref{eq:thmve1} follows from the same argument as in the proof of \cite[Theorems 4.5 and 4.6]{daveigaInterpolationStabilityEstimates2022}, with \(\mathit{\Pi}^{E}_{\bar{k}}\) replaced by \(\bfitgreek{\Pi}^{\mathrm{Se}}_{k}\) and \cref{thm:errorse} applied.
\end{proof}

\begin{theorem}[Interpolation Properties of Serendipity Node Virtual
    Element Space]
  \label{thm:vn}Let \(\boldsymbol{I}^{\mathrm{Sn}}_{k}\) be the interpolator to \(\mathbf{V}_{k,l}^{\mathrm{Sn}}(K)\),
  for each \(\nabla v \in \mathrm{H}^s(K), 1 / 2<s\leq k\), we have:
  \begin{equation}
    \left\|\nabla \left(v-\boldsymbol{I}^{\mathrm{Sn}}_{k} v\right)\right\|_K \lesssim h_K^s \left|\nabla v\right|_{s, K}.
  \end{equation}
\end{theorem}

\begin{proof}
  The inequality simply follows from \cref{thm:excom} and and \cref{thm:ve}.
\end{proof}

We notice that the \(\mathrm{L}^2\)-polynomial projector is computable, not only of order \(k\), but also of any order \(l\) on the serendipity space, which leads to the following main result.

\begin{theorem}[Existence of Stable Polynomial Projection]
  \label{thm:st}
  Let \(\mathbf{V}^{\mathrm{S}}_{l}\) denote the serendipity face or edge space,
  and \(\mathbf{V}^{\mathrm{S}}_{l}\) is defined by the same DoFs
  \(\left\{\mathcal{F}_{i}\right\}_{i=1}^{N_{\mathrm{S}}}\). Then
  \(\exists l, \text{ s.t. }\left\|\bfitgreek{\Pi}_{l}\boldsymbol{v}\right\|_{K}\approx \|\boldsymbol{v}\|_{K}\ \forall \boldsymbol{v} \in\mathbf{V}^{\mathrm{S}}_{l}\).
\end{theorem}

\begin{proof}
  The detailed proof is deferred to the Supplementary Material.
\end{proof}

Finally, it remains to discuss how large \(l\) should be for each element
\(K\). \Cref{thm:st} guarantees the existence of sufficiently large projection
orders for both the face space and the edge space.
The proof is nonconstructive and the resulting lower bounds are far from
optimal. The
dimension-counting discussion below still provides a practical way to estimate
the required order. Note that in the following discussion, we assume that
\(K\) is convex and does not possess specific symmetries.

Since the polynomial part is invariant under \(\bfitgreek{\Pi}_{l}\), we can focus on the non-polynomial \(\boldsymbol{v}\) with \(\bfitgreek{\Pi}_{k}^{\mathrm{S}}\boldsymbol{v}=0\). For convenience, define \(\pi_{k,d},\rho_{k,d}, \gamma_{k,d}\) as in \cref{eq:dimdef}, \(\eta\) as the number of faces of \(K\), and \(N_{\mathrm{S}}\) the number of DoFs of the serendipity space.

For \(\boldsymbol{v} \in \mathbf{V}^{\mathrm{Sf}}_{k,l}\), by using \cref{eq:equivsf,eq:sf}, we notice that
\[\left(\boldsymbol{v},\boldsymbol{p}_{l}\right)_{K}=\left(\boldsymbol{v}\cdot \boldsymbol{n},q_{l+1}\right)_{\partial K},\]
i.e., only these \(\nabla q_{l+1}\) with nonzero trace contribute, which means, according to \cite[Chapter 8.8]{veigaSerendipityFaceEdge2017}, we ideally have:
\begin{equation}
  \min_l\{\gamma_{l,3}-\pi_{l-\eta+1,3}+3\pi_{k,3}, 3\pi_{l,3}\}\geq N_{\mathrm{S}}.
\end{equation}

For \(\mathbf{V}^{\mathrm{Se}}_{k,l}\), similar arguments yield
\[\left(\boldsymbol{v},\boldsymbol{p}_{l}\right)_{K}=\left(\boldsymbol{v},\boldsymbol{n}\times \boldsymbol{q}_{l+1}\right)_{\partial K},\]
and only those rotational representatives with nonzero tangential trace
contribute. By \cref{eq:iso2}, we may identify the rotational family
\(\nabla\times\mathbb{P}^{3}_{l+1}(K)\) with
\(\boldsymbol{x}_{K}\times\mathbb{P}^{3}_{l}(K)\), whose dimension is
\(\rho_{l,3}\). Inside this family, zero tangential trace actually forces zero
boundary values: if
\(\widetilde{\boldsymbol{q}}_{l}=\boldsymbol{x}_{K}\times\boldsymbol{r}_{l}\)
and \(\boldsymbol{n}_{F}\times\widetilde{\boldsymbol{q}}_{l}=0\) on a face
\(F\), then \(\widetilde{\boldsymbol{q}}_{l}\) is parallel to
\(\boldsymbol{n}_{F}\), while
\(\widetilde{\boldsymbol{q}}_{l}\cdot\boldsymbol{x}_{K}=0\). Since
\(\boldsymbol{x}_{K}\cdot\boldsymbol{n}_{F}\) is a nonzero constant on
\(F\), this implies \(\widetilde{\boldsymbol{q}}_{l}=0\) on \(F\). Hence the
zero-trace subfamily consists precisely of those rotational representatives
vanishing on the whole boundary. Therefore each component is divisible by the
boundary bubble of degree \(\eta\), so this subfamily is a boundary-bubble
multiple of a rotational family of order \(l-\eta\), and therefore has
dimension \(\rho_{l-\eta,3}\). Hence, analogously to the face case, we
ideally have:
\begin{equation}
  \min_l\{\rho_{l,3}-\rho_{l-\eta,3}+3\pi_{k,3}, 3\pi_{l,3}\} \geq N_{\mathrm{S}}.
\end{equation}

\begin{remark}
  From the above discussion, a considerable part of \(\mathbb{P}_l^3\) has no contribution to the projection due to the definition of the reduced DoFs. If only one of the serendipity spaces is used in an application, one may consider an alternative definition to make the full use of \(\mathbb{P}_l^3\). For example, denote \(\boldsymbol{v}=\nabla\times\boldsymbol{u}+\boldsymbol{w}\) and \cref{eq:equivsf} can be defined by \(\left(\boldsymbol{u}-\bfitgreek{\Pi}^{\mathrm{Se}}\boldsymbol{u},\nabla\times\boldsymbol{p}_{k_{\mathrm{r}}}\right)_K=0\). Note that \(\bfitgreek{\Pi}^{\mathrm{Sf}}\boldsymbol{w}\in \boldsymbol{x}_K\mathbb{P}_{k-1}\perp \boldsymbol{x}_K \times \boldsymbol{p}_{k_{\mathrm{r}}} \). However, such definition will trivialize the \(\bfitgreek{\Pi}_l\) on \(\mathbf{V}^{\mathrm{Se}}_{k,l}\), so it is not adopted in the main text.
\end{remark}

\section{Computational Costs}

We report only per-element asymptotic scalings as functions of \(k,l\), since absolute costs depend on implementation details.

\emph{Identification of DoFs and \(l\).} Indeed, no separate test is needed in identifying the DoFs or \(l\), as the small linear systems assembled for the projectors become singular precisely when identifiability fails. The initial order of preserved DoFs \cref{eq:dofvf3,eq:dofve2,eq:dofve4,eq:dofve5} are set following \cite[Chapters 3.4, 7.2, 8.8]{veigaSerendipityFaceEdge2017}. For \(k=0,1\) there are no such internal DoFs. The initial guess for \(l\) is set according to the discussion above. Notably, these initial choices are generally sufficient for tested non-regular meshes.

\emph{Computation of the projectors.} Each element requires (i) Numerical integration of moments and coordinate transforms to form the matrices and (ii) the solution of small dense linear systems. Using the strategy of \cite{antoniettiFastNumericalIntegration2018}, an order-\(2k\) 3D polynomial integration costs \(O((2k)^{3})\) operations. But the solution of \(n\)-by-\(n\) dense linear systems costs \(O(n^{3})\) operations, which in our scenario dominates. Then:
\begin{itemize}
  \item \(\bfitgreek{\Pi}^{\mathrm{S}}_{k}\): system of size \(\dim \mathbb{P}_{k}^3=3\pi_{k,3}\) costs \(O((3\pi_{k,3})^{3})\sim O((3k^3)^{3})\) operations.
  \item \(\bfitgreek{\Pi}_{l}\): system of size \(\dim \mathbb{P}_{l}^3=3\pi_{l,3}\) (3 uncoupled blocks after splitting by monomial parity) costs \(O(\pi_{l,3}^{3})\sim O(l^{9})\) operations.
\end{itemize}

\emph{Comparison.} Versus the stabilized serendipity VEM\cite{veigaSerendipityFaceEdge2017}, the only extra work is replacing \(\bfitgreek{\Pi}_{k}\) by \(\bfitgreek{\Pi}_{l}\). But our advantages are:
\begin{itemize}
  \item Less DoFs in solving the global system.
  \item If most elements derive from a small set of prototype shapes forming a tiling (repeated cell geometries), the local projection matrices can be precomputed once per prototype and reused. This reduces the per-mesh cost from \(O(N_{\text{cells}}\, n^{3})\) to \(O(N_{\text{shapes}}\, n^{3}+N_{\text{cells}}\, mn)\), where \(n=3\pi_{k,3}\) or \(3\pi_{l,3}\), and \(m\) is the number of DoFs per element.
\end{itemize}

In \cref{tab:l-by-k-polygon} we list sufficient \(l\) for several \(k\) on non-regular convex polygons, but note that we only need them for 2D applications. For polyhedra, it is inconvenient to list all possible configurations, so we instead report the numerical statistics in \cref{sec:exp-cost}. It shows that the required \(l\) is often close to \(k\), especially for small \(k\).

\begin{table}[t]
  \centering
  \caption{Sufficient \(l\) for \(\mathbf{V}^{\mathrm{Sf}}_{k,l}\) (or \(\mathbf{V}^{\mathrm{Se}}_{k,l}\)) on non-regular convex polygons with \(m\) edges.}
  \label{tab:l-by-k-polygon}
  \begin{tabular}{c|c|ccccccc}
    \hline
    \multicolumn{2}{c|}{} & \multicolumn{7}{c}{\(m\)} \\
    \cline{3-9}
    \multicolumn{2}{c|}{}           & 3 & 4 & 5 & 6 & 7 & 8 & 9 \\
    \hline
    \multirow{4}{*}{\(k\)} & 0 & 1 & 1 & 1 & 1 & 2 & 2 & 2 \\
                             & 1 & 1 & 2 & 2 & 2 & 3 & 3 & 3 \\
                             & 2 & 2 & 3 & 3 & 3 & 4 & 4 & 4 \\
                             & 3 & 3 & 4 & 4 & 4 & 5 & 5 & 5 \\
    \hline
  \end{tabular}
\end{table}

\section{Model Problem}

In this section, we approximate the Maxwell eigenvalue problem using the proposed stabilization-free virtual element method. The problem is chosen for two reasons:
\begin{itemize}
  \item It offers solutions with known eigenvalues and eigenfunctions of different regularities, which allows us to precisly evaluate the convergence of the proposed method.
  \item The stabilization-free technique for both the \(\mathbf{H}(\operatorname{\mathbf{curl}})\) and \(\mathbf{H}(\operatorname{div})\)-conforming spaces are relevant, so that it simultaneously demonstrates the applicability of the proposed method to both spaces.
\end{itemize}

The problem is: find eigenvalues \(\lambda\in\mathbb{R}\) and eigenfunctions \(\boldsymbol{u}\neq 0\), satisfying:
\begin{equation}
  \begin{aligned}
    \nabla \times \nabla \times \boldsymbol{u} & = \lambda \boldsymbol{u} & \text{in } \Omega,          \\
    \nabla \cdot \boldsymbol{u}                & = 0                      & \text{in } \Omega,          \\
    \boldsymbol{u} \times \boldsymbol{n}       & = 0                      & \text{on } \partial \Omega,
  \end{aligned}
\end{equation}
where \(\Omega\) is the domain, \(\lambda=\left(\omega/c\right)^2\), \(\omega\) is the angular frequency, and \(c=1\) is the speed of light.

The variational form is: find eigenvalues \(\lambda \in \mathbb{R}\) and eigenfunctions \(\boldsymbol{u}\neq 0 \in \mathbf{H}_{0,\operatorname{div}0}\left(\operatorname{\mathbf{curl}},\Omega\right)\), satisfying:
\begin{equation}
    a(\boldsymbol{u},\boldsymbol{v}):=\left(\nabla \times\boldsymbol{u},\nabla \times \boldsymbol{v}\right)=\lambda\left(\boldsymbol{u},\boldsymbol{v}\right)  \quad \forall \boldsymbol{v}\in \mathbf{H}_{0,\operatorname{div}0}\left(\operatorname{\mathbf{curl}},\Omega\right).
  \label{eq:continuous}
\end{equation}

Before discretization, we define the global conforming discrete spaces:
\begin{equation}
  \begin{aligned}
     & \mathbf{V}^{\mathrm{Sf}}_{h}\left(\Omega\right):=\left\{\boldsymbol{v}_{h}\in \mathbf{H}\left(\operatorname{div},\Omega\right)\mid \boldsymbol{v}_{h}\in \mathbf{V}^{\mathrm{Sf}}_{k-1,l_K-1}\left(K\right)\ \forall K\in \mathcal{T}_{h}\right\},       \\
     & \mathbf{V}^{\mathrm{Se}}_{h}\left(\Omega\right):=\left\{\boldsymbol{v}_{h}\in \mathbf{H}\left(\operatorname{\mathbf{curl}},\Omega\right)\mid \boldsymbol{v}_{h}\in \mathbf{V}^{\mathrm{Se}}_{k,l_K}\left(K\right)\ \forall K\in \mathcal{T}_{h}\right\}, \\
     & \mathbf{V}^{\mathrm{Sn}}_{h}\left(\Omega\right):=\left\{v_{h}\in \mathrm{H}^1\left(\Omega\right)\mid v_{h}\in \mathbf{V}^{\mathrm{Sn}}_{k+1,l_K+1}\left(K\right)\ \forall K\in \mathcal{T}_{h}\right\}.
  \end{aligned}
\end{equation}

We further denote the divergence-free and conforming subspaces by:
\begin{equation}
  \begin{aligned}
     & \mathbf{V}^{\mathrm{Sf}}_{\operatorname{div}0,h}\left(\Omega\right):=\left\{\boldsymbol{v}_{h}\in \mathbf{V}^{\mathrm{Sf}}_{h}\left(\Omega\right)\mid \nabla\cdot \boldsymbol{v}_{h}=0\right\}, \\
     & \mathbf{V}^{\mathrm{Se}}_{0,\operatorname{div}0,h}\left(\Omega\right):=\mathbf{V}^{\mathrm{Se}}_{h}\left(\Omega\right)\cap \mathbf{H}_{0,\operatorname{div}0}\left(\operatorname{\mathbf{curl}},\Omega\right).
  \end{aligned}
\end{equation}

With an order \(l_K\) chosen so that \(\bfitgreek{\Pi}_{l_K}\) is injective on the
local serendipity space for each element \(K\), we can define the stable inner
product as:
\begin{equation}
  \begin{aligned}
     & \left[\boldsymbol{u}_{h},\boldsymbol{v}_{h}\right]_{\mathrm{S},K}:= \left(\bfitgreek{\Pi}_{l_K}\boldsymbol{u}_{h},\bfitgreek{\Pi}_{l_K}\boldsymbol{v}_{h}\right)_{K} & \forall\boldsymbol{u}_{h},\boldsymbol{v}_{h}\in \mathbf{V}^{\mathrm{Se}}_{k,l_K}\left(K\right),                                        \\
     & \left[\boldsymbol{u}_{h},\boldsymbol{v}_{h}\right]_{\mathrm{S}}:=\sum_{K\in \mathcal{T}_{h}}\left[\boldsymbol{u}_{h},\boldsymbol{v}_{h}\right]_{\mathrm{S},K}        & \forall\boldsymbol{u}_{h},\boldsymbol{v}_{h}\in \mathbf{V}_{h}^{\mathrm{Se}}(\Omega) \text{ or } \mathbf{V}_{h}^{\mathrm{Sf}}(\Omega).
  \end{aligned}
\end{equation}

When applied to \(\mathbf{V}^{\mathrm{Se}}_{h}(\Omega)\) and
\(\mathbf{V}^{\mathrm{Sf}}_{h}(\Omega)\), we write the corresponding stable inner
products as \([\cdot,\cdot]_{\mathrm{Se}}\) and \([\cdot,\cdot]_{\mathrm{Sf}}\),
respectively.

For simplicity, we denote the associated global projector to the broken polynomial space as \(\bfitgreek{\Pi}_{h}\).

Define the equivalent discrete norm for \(\boldsymbol{v}_{h}\in \mathbf{V}_{h}^{\mathrm{Se}}\left(\Omega\right) \text{ or } \mathbf{V}_{h}^{\mathrm{Sf}}\left(\Omega\right)\) as:
\begin{equation}
  \begin{aligned}
     & \normmm{\boldsymbol{v}_{h}}^2_{\mathrm{S},K}:=\left[\boldsymbol{v}_{h},\boldsymbol{v}_{h}\right]_{\mathrm{S},K},     \\
     & \normmm{\boldsymbol{v}_{h}}^2_{\mathrm{S}}:=\sum_{K\in \mathcal{T}_{h}}\normmm{\boldsymbol{v}_{h}}^2_{\mathrm{S},K}.
  \end{aligned}
\end{equation}

Finally, we approximate the variational form with the stabilization-free VEM. Assuming the following discrete Helmholtz decomposition, for \(\boldsymbol{v}_h\in \mathbf{V}^{\mathrm{Se}}_{h}\left(\Omega\right)\),
\begin{equation}
  \label{eq:orthdecomp}
  \begin{aligned}
& \boldsymbol{v}_h=\bfitgreek{\psi}+\nabla \varphi, \\
& \bfitgreek{\psi} \in \mathbf{H}^1(\Omega),\left\{\begin{array}{l}
\nabla \times \bfitgreek{\psi} \in \mathbf{V}^{\mathrm{Sf}}_{\operatorname{div}0,h}\left(\Omega\right) \\
\nabla \cdot \bfitgreek{\psi}=0\\
\bfitgreek{\psi} \times \boldsymbol{n}=0
\end{array}\right., \\
& \varphi \in \mathbf{V}^{\mathrm{Sn}}_{h}\left(\Omega\right).
\end{aligned}
\end{equation}

So we can just solve: find eigenvalues \(\lambda_{h} \in \mathbb{R}\) and eigenfunctions \(\boldsymbol{u}_{h}\neq 0 \in \mathbf{V}^{\mathrm{Se}}_{0,\operatorname{div}0, h}\left(\Omega\right)\), satisfying:
\begin{equation}
    a_h(\boldsymbol{u}_{h}, \boldsymbol{v}_{h}):=\left[\nabla \times\boldsymbol{u}_{h},\nabla \times\boldsymbol{v}_{h}\right]_{\mathrm{Sf}}=\lambda_{h}\left[\boldsymbol{u}_{h},\boldsymbol{v}_{h}\right]_{\mathrm{Se}}  \quad \forall \boldsymbol{v}_{h}\in \mathbf{V}^{\mathrm{Se}}_{0,\operatorname{div}0, h}\left(\Omega\right).
  \label{eq:discrete}
\end{equation}

For simplicity, we briefly conduct the error analysis from the discrete formulation following the standard procedure of analyzing the mixed finite element method (c.f. \cite{boffiMixedFiniteElement2013}) and eigenvalue problem (c.f. \cite{babuskaEigenvalueProblems1991}).

\begin{lemma}[Coercivity]
  \label{lem:coercivity}
  \begin{equation}
    \normmm{\nabla \times \boldsymbol{v}_h}_{\mathrm{Sf}} \gtrsim \left\|\boldsymbol{v}_h\right\| \quad \forall \boldsymbol{v}_h \in \mathbf{V}^{\mathrm{Se}}_{0,\operatorname{div}0, h}\left(\Omega\right).
  \end{equation}
\end{lemma}

\begin{proof}
  By definition,
  \(\boldsymbol{v}_{h}\in \mathbf{V}^{\mathrm{Se}}_{0,\operatorname{div}0, h}\left(\Omega\right)
  \subseteq \mathbf{H}_{0,\operatorname{div}0}\left(\operatorname{\mathbf{curl}},\Omega\right)\).
  Therefore the Maxwell estimate (see, e.g., \cite[Proposition 3.7]{amroucheVectorPotentialsThreedimensional1998}) yields
  \begin{equation}
    \left\|\boldsymbol{v}_{h}\right\|\lesssim\left\|\nabla \times \boldsymbol{v}_{h}\right\|.
  \end{equation}

  On the other hand, by the exact sequence structure of the local spaces,
  \(\nabla \times \boldsymbol{v}_{h}\in \mathbf{V}^{\mathrm{Sf}}_{\operatorname{div}0,h}\left(\Omega\right)\subseteq \mathbf{V}^{\mathrm{Sf}}_{h}\left(\Omega\right)\).
  Since \(l_K\) is chosen so that \(\bfitgreek{\Pi}_{l_K}\) is injective on each local
  face serendipity space, which is guaranteed for sufficiently large \(l_K\) by
  \cref{thm:st}, finite dimensional norm equivalence gives the local norm equivalence
  \begin{equation}
    \left\|\nabla \times \boldsymbol{v}_{h}\right\|_{K}\approx \normmm{\nabla \times \boldsymbol{v}_{h}}_{\mathrm{S},K}.
  \end{equation}
  Summing over all elements gives
  \begin{equation}
    \left\|\nabla \times \boldsymbol{v}_{h}\right\|\approx \normmm{\nabla \times \boldsymbol{v}_{h}}_{\mathrm{Sf}},
  \end{equation}
  and the claim follows.
\end{proof}

Now define the solution operator \(\boldsymbol{T},\boldsymbol{T}_{h}\), such that:
\begin{equation}
  \begin{aligned}
    a\left(\boldsymbol{T}\boldsymbol{u},\boldsymbol{v}\right)             & =\left(\boldsymbol{u},\boldsymbol{v}\right)                  & \forall \boldsymbol{u}, \boldsymbol{v}\in \mathbf{H}_{0,\operatorname{div}0}\left(\operatorname{\mathbf{curl}},\Omega\right),                                                         \\
    a_{h}\left(\boldsymbol{T}_{h}\boldsymbol{u},\boldsymbol{v}_{h}\right) & =\left[\boldsymbol{u},\boldsymbol{v}_{h}\right]_{\mathrm{Se}} & \forall \boldsymbol{u}\in \mathbf{H}_{0,\operatorname{div}0}\left(\operatorname{\mathbf{curl}},\Omega\right),\boldsymbol{v}_{h}\in \mathbf{V}^{\mathrm{Se}}_{0,\operatorname{div}0, h}\left(\Omega\right).
  \end{aligned}
\end{equation}

\begin{lemma}
  \label{lem:Terror}
  Assume that \(\boldsymbol{u}\in \mathbf{H}^{s}\left(\Omega\right)\cap \mathbf{H}^{r}\left(\operatorname{\mathbf{curl}},\Omega\right), s,r > 1/2\) and is a solution to \cref{eq:continuous}, we have:
  \begin{equation}
    \left\|\boldsymbol{T}\boldsymbol{u}-\boldsymbol{T}_{h}\boldsymbol{u}\right\|\lesssim h^{s}|\boldsymbol{u}|_{s}+h^{1+r}|\nabla \times\boldsymbol{u}|_{r}.
  \end{equation}
\end{lemma}

\begin{proof}
  Take \(\boldsymbol{v}_{h}\in \mathbf{V}^{\mathrm{Se}}_{0,\operatorname{div}0, h}\left(\Omega\right)\), by using \cref{eq:continuous,eq:discrete}, we have:
  \begin{equation}
    \begin{aligned}
      a_{h}\left(\boldsymbol{I}_{h}\boldsymbol{T}\boldsymbol{u}-\boldsymbol{T}_{h}\boldsymbol{u},\boldsymbol{v}_{h}\right) & =a_{h}\left(\boldsymbol{I}_{h}\boldsymbol{T}\boldsymbol{u},\boldsymbol{v}_{h}\right)-\left[\boldsymbol{u},\boldsymbol{v}_{h}\right]_{\mathrm{Se}}                                                                \\
                                                                                                                           & =\left[a_{h}-a\right]\left(\boldsymbol{I}_{h}\boldsymbol{T}\boldsymbol{u},\boldsymbol{v}_{h}\right)+a\left(\boldsymbol{I}_{h}\boldsymbol{T}\boldsymbol{u}-\boldsymbol{T}\boldsymbol{u},\boldsymbol{v}_{h}\right) \\
                                                                                                                           & \quad +\left(\boldsymbol{u},\boldsymbol{v}_{h}\right)-\left[\boldsymbol{u},\boldsymbol{v}_{h}\right]_{\mathrm{Se}}.
    \end{aligned}
  \end{equation}

  Then using \cref{lem:coercivity}, we have:
  \begin{equation}
    \begin{aligned}
      \left\|\boldsymbol{I}_{h}\boldsymbol{T}\boldsymbol{u}-\boldsymbol{T}_{h}\boldsymbol{u}\right\| & \lesssim \sup_{\left\|\boldsymbol{v}_{h}\right\|=1}a_{h}\left(\boldsymbol{I}_{h}\boldsymbol{T}\boldsymbol{u}-\boldsymbol{T}_{h}\boldsymbol{u},\boldsymbol{v}_{h}\right)                                                                                                                                                      \\
                                                                                                     & \lesssim \left(\sum_{K\in \mathcal{T}_{h}}\left(\left\|\left(\bfitgreek{\Pi}_{h}-\boldsymbol{E}\right)\nabla \times\boldsymbol{I}_{h}\boldsymbol{T}\boldsymbol{u}\right\|^{2}_{K}+\left\|\nabla \times\left(\boldsymbol{I}_{h}\boldsymbol{T}\boldsymbol{u}-\boldsymbol{T}\boldsymbol{u}\right)\right\|^{2}_{K}\right.\right. \\
                                                                                                     & \quad\left.\left.+\left\|\left(\bfitgreek{\Pi}_{h}-\boldsymbol{E}\right)\boldsymbol{u}\right\|^{2}_{K}\right)\right)^{\frac{1}{2}}                                                                                                                                                                                           \\
                                                                                                     & \lesssim \left(\sum_{K\in \mathcal{T}_{h}}\left(\left\|\left(\bfitgreek{\Pi}_{h}-\boldsymbol{E}\right)\nabla \times\boldsymbol{T}\boldsymbol{u}\right\|^{2}_{K}+\left\|\nabla \times\left(\boldsymbol{I}_{h}\boldsymbol{T}\boldsymbol{u}-\boldsymbol{T}\boldsymbol{u}\right)\right\|^{2}_{K}\right.\right.                   \\
                                                                                                     & \quad\left.\left.+\left\|\left(\bfitgreek{\Pi}_{h}-\boldsymbol{E}\right)\boldsymbol{u}\right\|^{2}_{K}\right)\right)^{\frac{1}{2}}                                                                                                                                                                                           \\
                                                                                                     & \lesssim h^{s}|\boldsymbol{u}|_{s}+h^{1+r}|\nabla \times\boldsymbol{T}\boldsymbol{u}|_{r}                                                                                                                                                                                                                                    \\
                                                                                                     & \lesssim h^{s}|\boldsymbol{u}|_{s}+h^{1+r}|\nabla \times\boldsymbol{u}|_{r},
    \end{aligned}
  \end{equation}
  where \cref{thm:errorse}, boundedness of \(\boldsymbol{I}_{h}\) and invariance of \(\boldsymbol{T}\) are used.

  Finally by a triangle inequality and \cref{thm:ve}, the claim follows.
\end{proof}

\begin{theorem}[A Priori Error Estimates]
  \label{thm:error}
  Let \(\lambda_{h}, \boldsymbol{u}_{h}\) be the solution to \cref{eq:discrete}, and \(\lambda, \boldsymbol{u}\) be the solution to \cref{eq:continuous}. Then, there exists some \(1/2<s\) and \(\nabla \times \mathbb{P}_{k+1}^3\left(K\right)\subseteq \mathbf{V}^{\mathrm{Se}}_{\operatorname{div}0}\left(K\right) \ \forall K\in \mathcal{T}_{h}\), such that \(\boldsymbol{u}\in \mathbf{H}^s\left(\Omega\right)\), and:
  \begin{subequations}
    \begin{align}
      \left\|\boldsymbol{u}-\boldsymbol{u}_{h}\right\| & \lesssim h^{r},\label{eq:thmerror1}  \\
      \left|\lambda-\lambda_{h}\right|                 & \lesssim h^{2r},\label{eq:thmerror2}
    \end{align}
  \end{subequations}
  where \(r=\min \left\{k+1,s\right\}\).
\end{theorem}

\begin{proof}
  Since \(\boldsymbol{u}\in \mathbf{H}_{0}\left(\operatorname{\mathbf{curl}}, \Omega\right)\cap\mathbf{H}\left(\operatorname{div},\Omega\right)\), and \(\nabla \times \boldsymbol{u}\in \mathbf{H}\left(\operatorname{\mathbf{curl}}, \Omega\right)\cap\mathbf{H}_{0}\left(\operatorname{div},\Omega\right)\), by \cite[Proposition 3.7]{amroucheVectorPotentialsThreedimensional1998}, \(\boldsymbol{u}, \nabla \times \boldsymbol{u}\in \mathbf{H}^s\left(\Omega\right)\) for some \(s>1/2\). Then we can apply the standard analysis as in \cite[Section 6]{babuskaEigenvalueProblems1991}, and assertion \cref{eq:thmerror2} follows from \cite[Theorems 8.3 and 8.4]{babuskaEigenvalueProblems1991}, providing \cref{lem:Terror}.
  Since \cref{lem:Terror} gives the operator convergence estimate, the standard
  spectral approximation theory for compact self-adjoint operators yields both
  \cref{eq:thmerror1} and \cref{eq:thmerror2}; see \cite[Section 6]{babuskaEigenvalueProblems1991}.
\end{proof}

\section{Numerical Experiments}
In this section, we present the numerical experiments to verify the theoretical results. For the sake of convenience, we use SFVEM, and FEM as the abbreviations of the stabilization-free virtual element method proposed in this work, and the conventional finite element method, respectively. The SFVEM is implemented in Python with Numba\cite{lam2015numba} and JaxFEM\cite{Xue2023JaxFEM}, while it is also conducted with GPU acceleration (Nvidia® Tesla® A800). While in comparison, the conventional FEM (including Tetrahedral mesh generation) is conducted with COMSOL Multiphysics® 6.3\cite{multiphysics1998introduction} with CPU (Intel® Xeon® 6242R). The Voronoi polyhedral mesh for SFVEM is generated with ANSYS® Fluent® 2025 R1. The numerical results obtained with commercial FEM software serve as reference.

With cylindrical coordinates \(\left[\hat{\mathbf{e}}_{\rho},\hat{\mathbf{e}}_{\phi},\hat{\mathbf{e}}_{z}\right]\) adopted (and hereafter in this section), an sectorial cylindrical domain \(\Omega = \left[0,1\right] \times \left[0,3\pi/2\right] \times \left[0,1\right]\) is set, so that analytical solutions with infinite regularity and singularities can be easily derived. The mesh is partitioned into Voronoi polyhedra for SFVEM, and tetrahedra for FEM, as shown in \cref{fig:mesh}.

\begin{figure}[htbp]            
  \centering                   
  \begin{minipage}{0.33\textwidth}
    \centering
    \includegraphics[width=\textwidth]{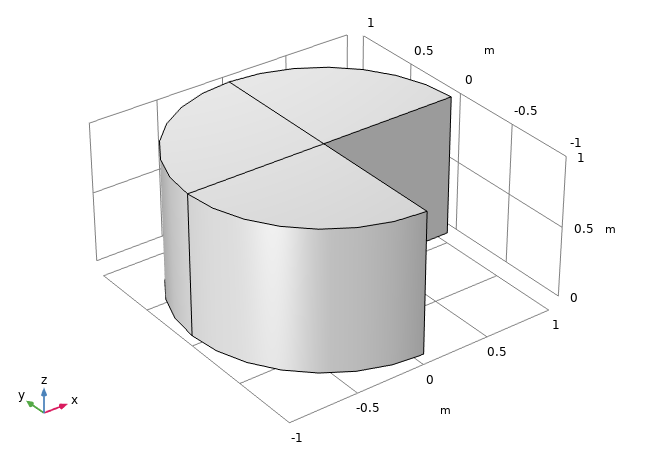} 
    \subcaption{}  
  \end{minipage}\hfill           
  \begin{minipage}{0.33\textwidth}
    \centering
    \includegraphics[width=\textwidth]{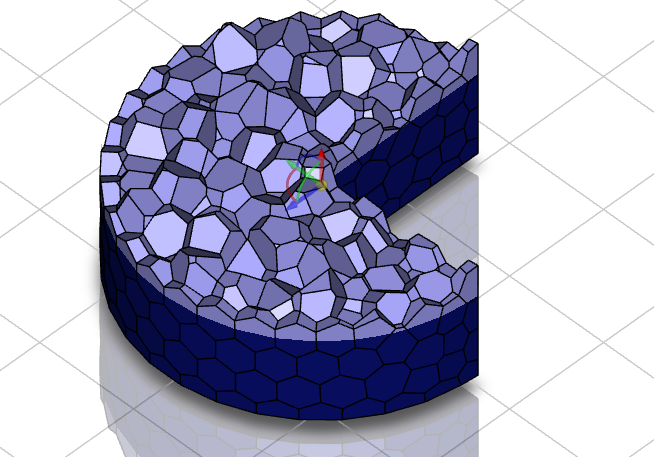} 
    \subcaption{}  
  \end{minipage}\hfill
  \begin{minipage}{0.33\textwidth}
    \centering
    \includegraphics[width=\textwidth]{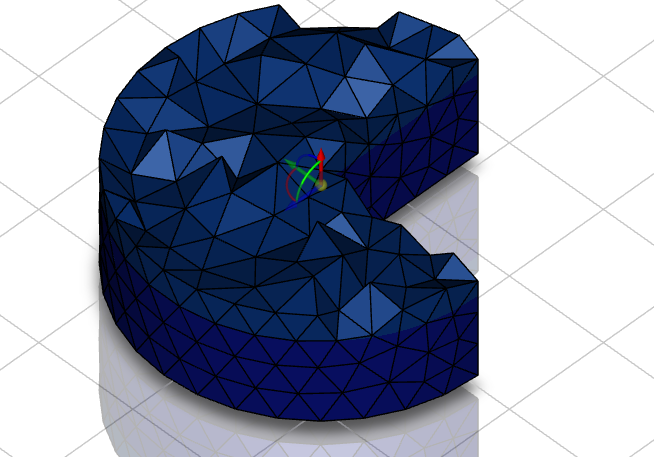} 
    \subcaption{}  
  \end{minipage}
  \caption{Geometry and Mesh. (a): Domain overview. (b): Voronoi polyhedral mesh. (c): Tetrahedral mesh.} 
  \label{fig:mesh}
\end{figure}

\subsection{Analytical Solutions and Regularity} For simplicity, constant coefficients are omitted in the following derivations. Considering only TE modes, the Borgnis potential\cite{broglieProblemesPropagationsGuidees1941} satisfying the boundary condition is:
\begin{equation}
  V= J_{\nu}\left(T\rho\right)\cos\left(\nu\phi\right)\sin\left(\beta z\right),
\end{equation}
where \(J_{\nu}\) is the Bessel function of the first kind, \(T\) is the radial wavenumber, \(\nu\) is the azimuthal wavenumber, and \(\beta\) is the axial wavenumber.

The eigenvalues are given by:
\begin{equation}
  \lambda=T^2+\beta^2,
\end{equation}
and the eigenfunctions are given by:
\begin{equation}
  \boldsymbol{u}=\left[
    \begin{matrix}
      -\frac{1}{\rho}\frac{\partial V}{\partial \phi} \\
      \frac{\partial V}{\partial \rho}                \\
      0
    \end{matrix}
    \right].
\end{equation}
Possible singularity occurs at \(\rho=0\), where we have the following asymptotic form:
\begin{equation}
  J_{\nu}\left(T\rho\right) \sim \frac{1}{\Gamma\left(\nu+1\right)}\left(\frac{T\rho}{2}\right)^{\nu} \text{ as } T\rho\rightarrow 0^+.
  \label{eq:asymp}
\end{equation}

It can be seen that \(J_{\nu}\left(T\rho\right)\) dominates and the integrals of \(\cos\left(\nu\phi\right)\sin\left(\beta z\right)\) (and their derivatives) are bounded by constants near \(\rho=0\). Then let \(\epsilon>0\) be an arbitrarily small value, \(0<s<1\), and consider the \(s\)-order seminorm\cite{dinezzaHitchhikersGuideFractional2012} of the eigenfunction in \(\Omega_{\epsilon}:= \left[0,\epsilon\right]\times\left[0,3\pi/2\right]\times\left[0,1\right]\):
\begin{equation}
  \begin{aligned}
    \left|\boldsymbol{u}\right|_{s,\Omega_{\epsilon}}^2 & =\int_{\Omega_{\epsilon}}\int_{\Omega_{\epsilon}}\frac{\left|\boldsymbol{u}_{\rho}\left(\boldsymbol{x}\right)-\boldsymbol{u}_{\rho}\left(\boldsymbol{y}\right)\right|^2+\left|\boldsymbol{u}_{\phi}\left(\boldsymbol{x}\right)-\boldsymbol{u}_{\phi}\left(\boldsymbol{y}\right)\right|^2}{\left|\boldsymbol{x}-\boldsymbol{y}\right|^{3+2s}}\mathrm{~d}\boldsymbol{x}\mathrm{d}\boldsymbol{y} \\
                                                        & \stackrel{\cref{eq:asymp}}{\approx} \int_{\Omega_{\epsilon}}\int_{\Omega_{\epsilon}}\frac{\left(\rho_{\boldsymbol{x}}^{\nu-1}-\rho_{\boldsymbol{y}}^{\nu-1}\right)^2}{\left|\boldsymbol{x}-\boldsymbol{y}\right|^{3+2s}}\mathrm{~d}\boldsymbol{x}\mathrm{d}\boldsymbol{y}                                                                                                                     \\
                                                        & \stackrel{\text{(Scaling)}}{\approx}\int_{0}^{\epsilon}\int_{0}^{\epsilon}\frac{\left(\rho_{\boldsymbol{x}}^{\nu-1}-\rho_{\boldsymbol{y}}^{\nu-1}\right)^2}{\left|\rho_{\boldsymbol{x}}-\rho_{\boldsymbol{y}}\right|^{2+2s}}\rho_{\boldsymbol{x}}\rho_{\boldsymbol{y}}\mathrm{~d}\rho_{\boldsymbol{x}}\mathrm{d}\rho_{\boldsymbol{y}}                                                         \\
                                                        & \approx\int_{0}^{\epsilon}\rho^{2\nu-1-2s}\mathrm{~d}\rho.
  \end{aligned}
\end{equation}
The integral converges if \(2\nu-1-2s>-1\). Indeed, for fractional \(\nu\), we have \(s<\nu\), and for integer \(\nu\), we have \(s=\infty\).

Now we investigate two solutions. The first one \(\boldsymbol{u}_{1}\) is infinitely regular, with \(\nu_{1}=4,\lambda_{1}\approx3.86500\pi^2\), and the second one \(\boldsymbol{u}_2\) is singular, with \(\nu_{2}=2/3,\lambda_{2}\approx3.38395\pi^2\). See \cref{fig:eigfunc} for the visualization of the two eigenfunctions, where the correspondence between the wavenumber and periodicity in each direction is clear, and in \cref{fig:eigfunc2}, the field diverges at \(\rho=0\) as expected.

\begin{figure}[htbp]            
  \centering                   
  \begin{minipage}{0.5\textwidth}
    \centering
    \includegraphics[width=\textwidth]{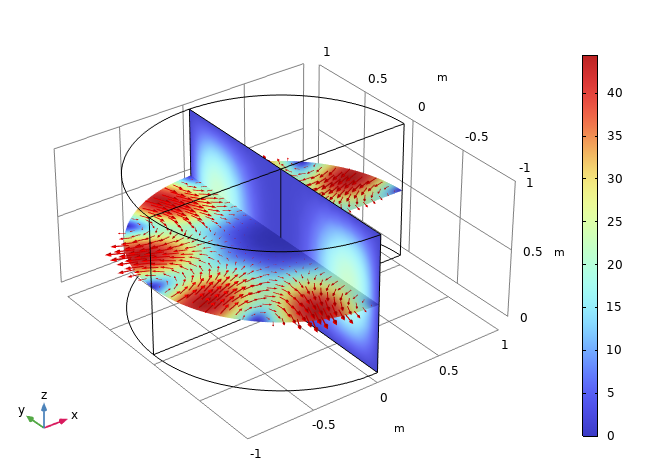} 
    \subcaption{}  
  \end{minipage}\hfill           
  \begin{minipage}{0.5\textwidth}
    \centering
    \includegraphics[width=\textwidth]{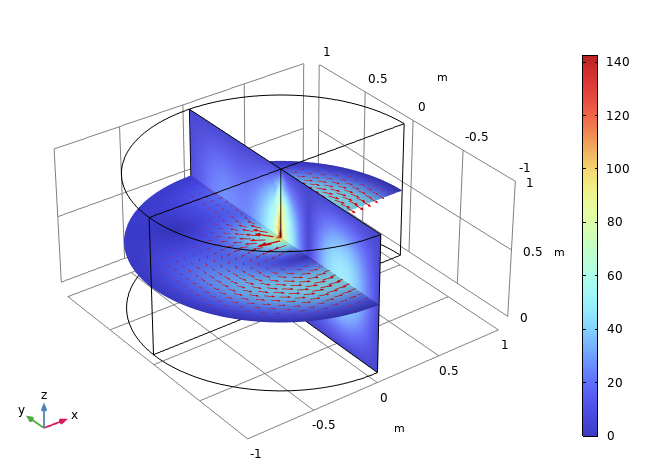} 
    \subcaption{}  
    \label{fig:eigfunc2}
  \end{minipage}
  \caption{Field Visualization. (a): \(\boldsymbol{u}_{1}\). (b): \(\boldsymbol{u}_{2}\). The direction of arrows represents the direction, and the length of arrows with background color represents the magnitude.}
  \label{fig:eigfunc}
\end{figure}

\subsection{Implementation and Assembling Costs}\label{sec:exp-cost} Since \(\mathbf{V}^{\mathrm{Se}}_{\operatorname{div}0}\) is essentially \(\mathbf{V}^{\mathrm{Se}}\) with \(k_{\mathrm{d}}=-1\), so the corresponding subspace to keep is \(\nabla \times \mathbb{P}_{k+1}^3\).

In \cref{fig:histo}, we show the histogram of polytopal shapes and corresponding average \(l\) for cells, in the mesh with \(h=0.2\), to give the reader an impression of how large \(l\) should be in practice. It can be seen that most cells are with 9 to 16 faces, and most faces are with 4 to 6 edges. And \(l\) is just moderately larger than \(k\) in most cells.

In \cref{fig:asstime}, we show the assembling time comparison between SFVEM and stabilized serendipity VEM. It can be seen that the assembling time of SFVEM is only moderately larger than that of FEM, which verifies the efficiency of the proposed method.

\begin{figure}[htbp]            
  \centering                   
  \begin{minipage}{0.5\textwidth}
    \centering
    \includegraphics[width=\textwidth]{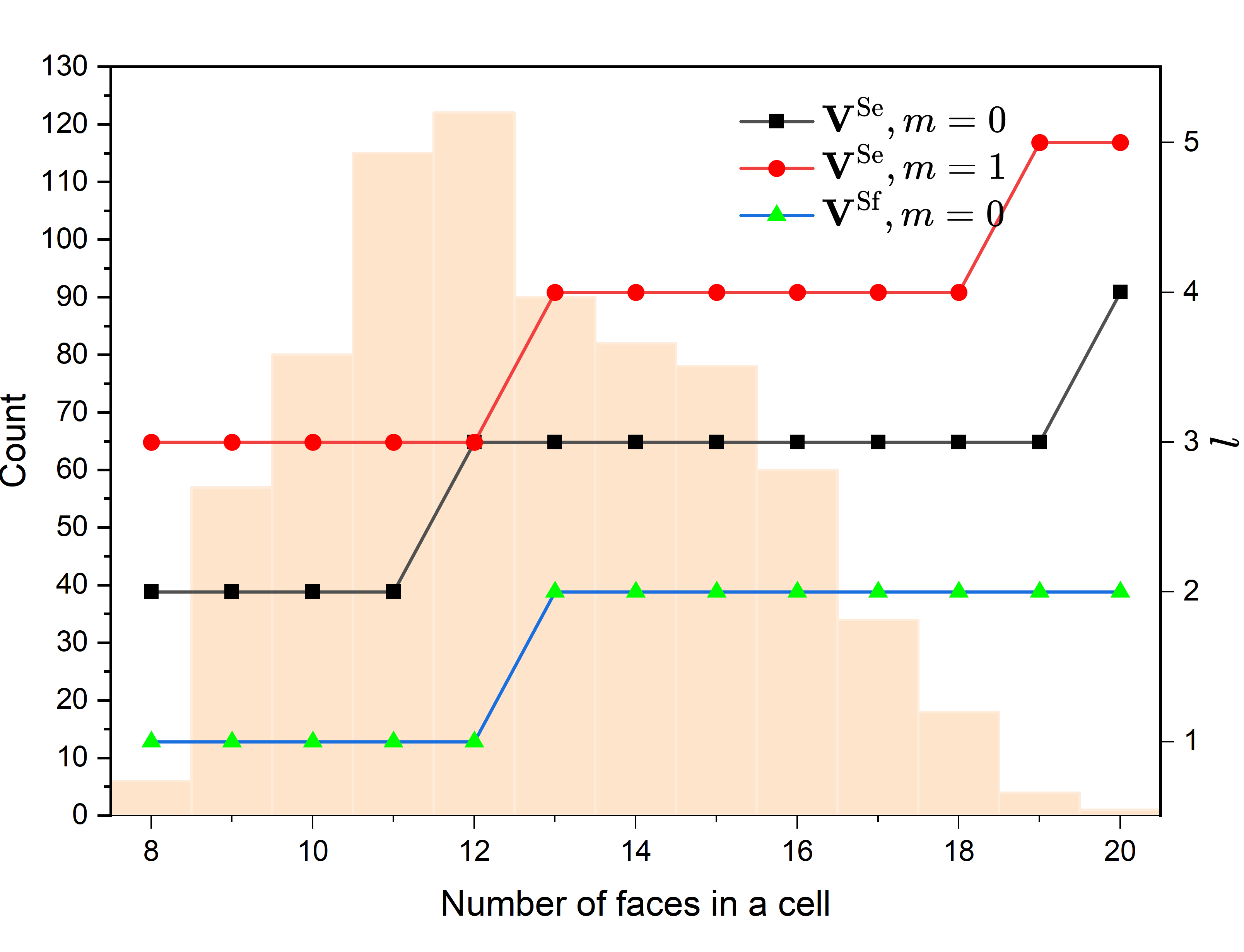} 
    \subcaption{}  
  \end{minipage}\hfill           
  \begin{minipage}{0.5\textwidth}
    \centering
    \includegraphics[width=\textwidth]{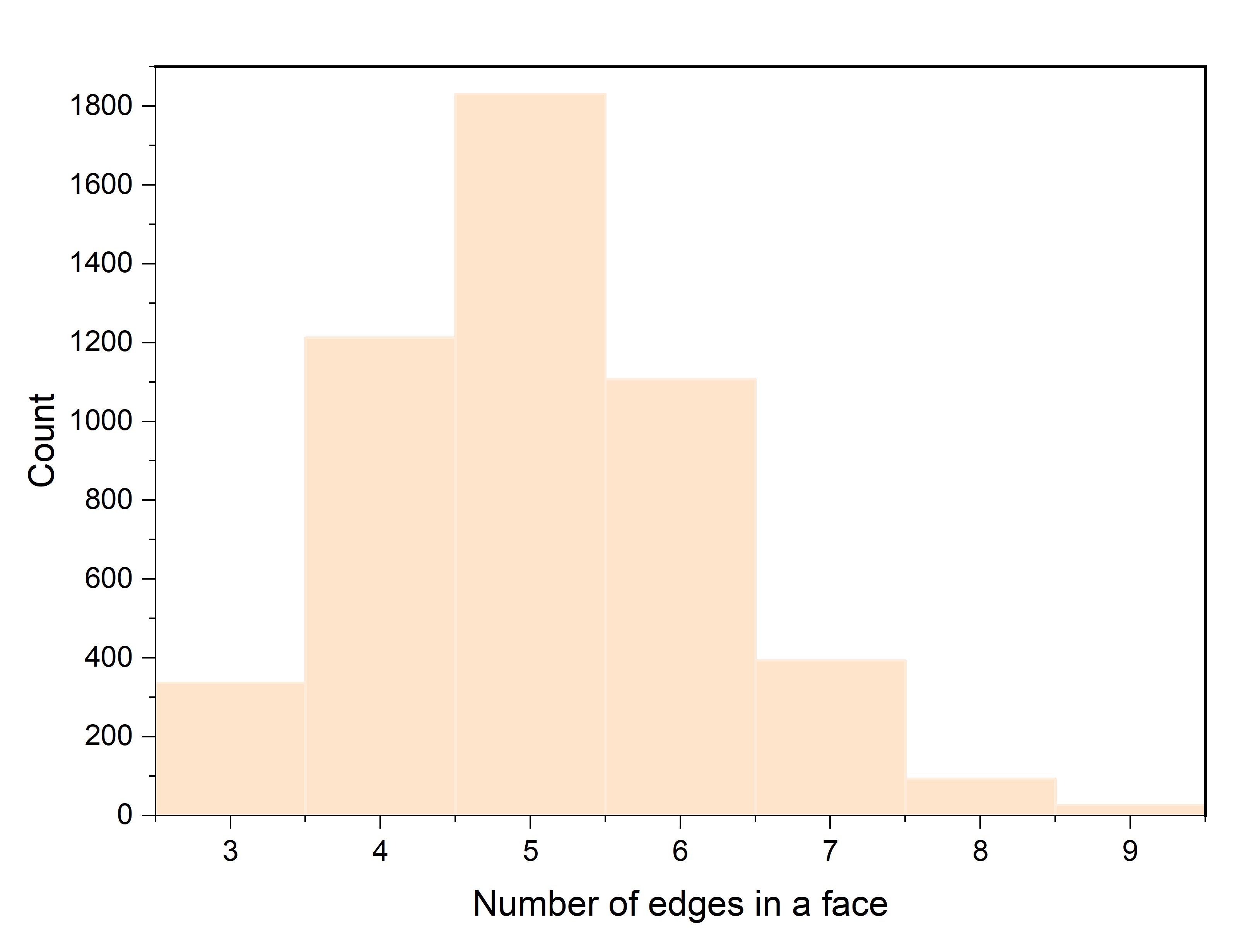} 
    \subcaption{}  
  \end{minipage}
  \caption{Histogram of polyhedral shapes and corresponding average \(l\) in the mesh with \(h=0.2\).}
  \label{fig:histo}
\end{figure}

\begin{figure}
  \centering                   
  \begin{minipage}{0.5\textwidth}
    \centering
    \includegraphics[width=\textwidth]{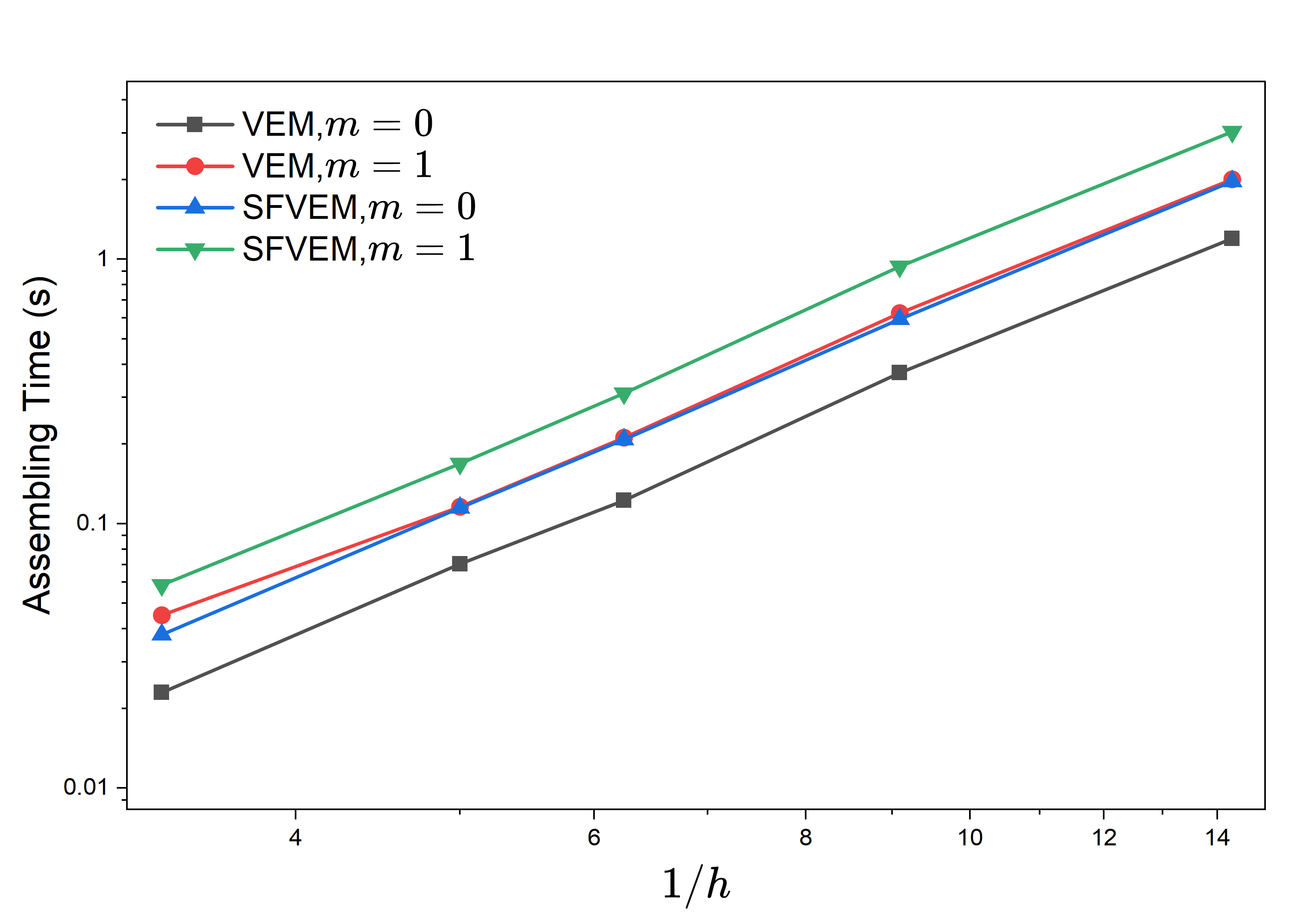} 
    \subcaption{}  
  \end{minipage}          
  \caption{Assembling Time Comparison.} 
  \label{fig:asstime}
\end{figure}

\subsection{Convergence Analysis} Since the \(\boldsymbol{u}_{1}\) is infinitely regular, by \cref{thm:error}, the error is expected to converge at the rate of the approximation space order, that is \(\left\|\boldsymbol{u}_{1}-\boldsymbol{u}_{1h}\right\|=O\left(h^{m+1}\right), \left|\lambda_{1}-\lambda_{1h}\right|=O\left(h^{2m+2}\right)\). Similarly, since \(\boldsymbol{u}_{2}\in \mathbf{H}^s\left(\Omega\right)\) with \(1/2<s<1\), and the convergence rate is expected to be \(\left\|\boldsymbol{u}_{2}-\boldsymbol{u}_{2h}\right\|=O\left(h^{s}\right), \left|\lambda_{2}-\lambda_{2h}\right|=O\left(h^{2s}\right)\), regardless of the approximation space order. Numerically, we measure the \(\mathrm{L}^2\)-norm by the equivalent discrete norm calculated from DoFs. The \(\mathrm{L}^2\)-norm errors and relative eigenvalue errors are plotted in \cref{fig:conv1,fig:conv2}, where the slope of the log-log plot agrees with the theoretical convergence rates.

Note that in \cref{fig:conv1,fig:conv2}, \(m\)-th order SFVEM corresponds to \(\left(m+1\right)\)-th order FEM due to differences in the definition of polynomial order in the respective methods. Since the memory consumption of 3D problem grows cubically with both \(m\) and \(1/h\), we can only afford to simulate cases with \(m=0\) and \(m=1\) in this work. Indeed, higher order spaces are unnecessary or even unfavorable for the model problem, since singularities exist for some solutions.

\begin{figure}[htbp]            
  \centering                   
  \begin{minipage}{0.5\textwidth}
    \centering
    \includegraphics[width=\textwidth]{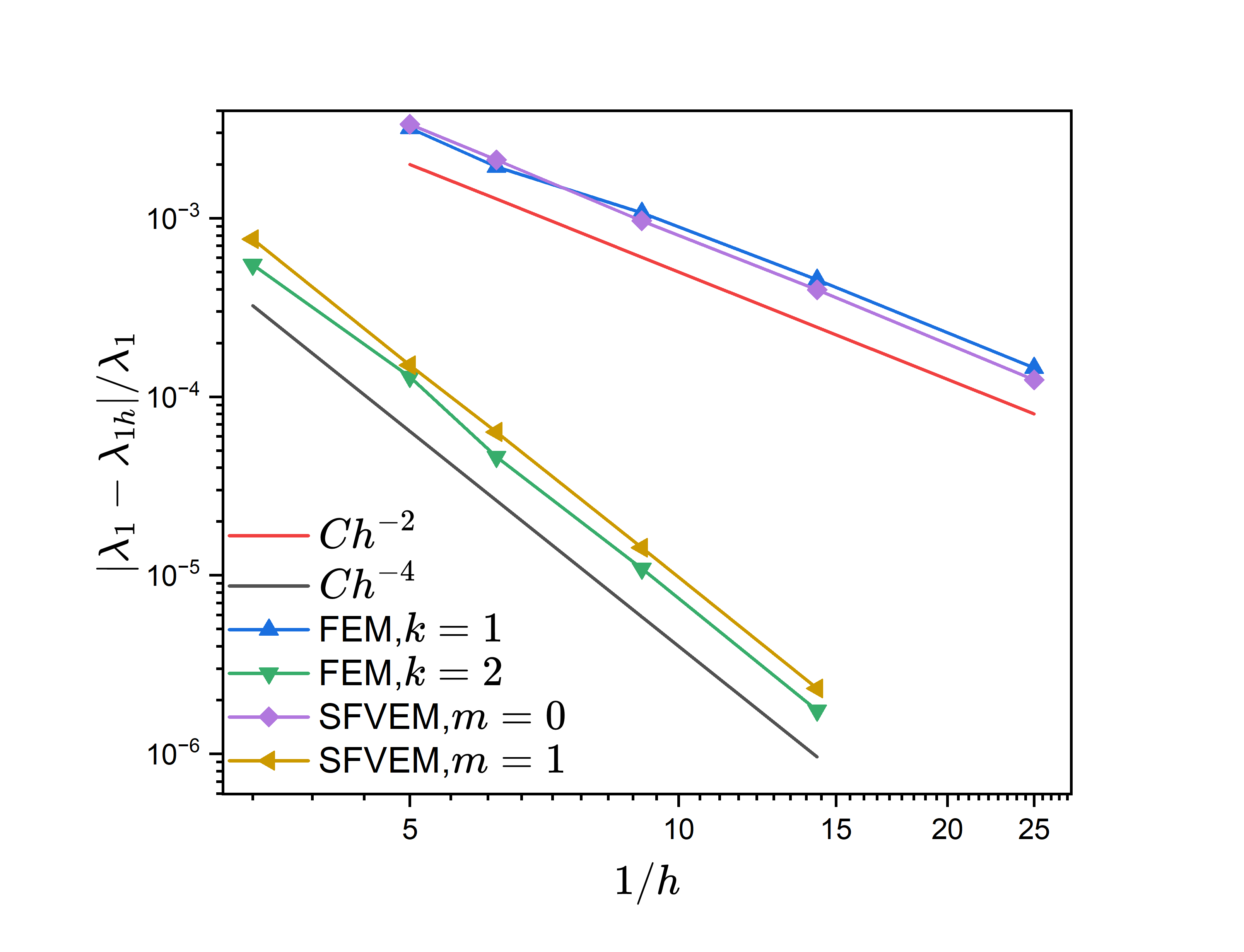} 
    \subcaption{}  
  \end{minipage}\hfill           
  \begin{minipage}{0.5\textwidth}
    \centering
    \includegraphics[width=\textwidth]{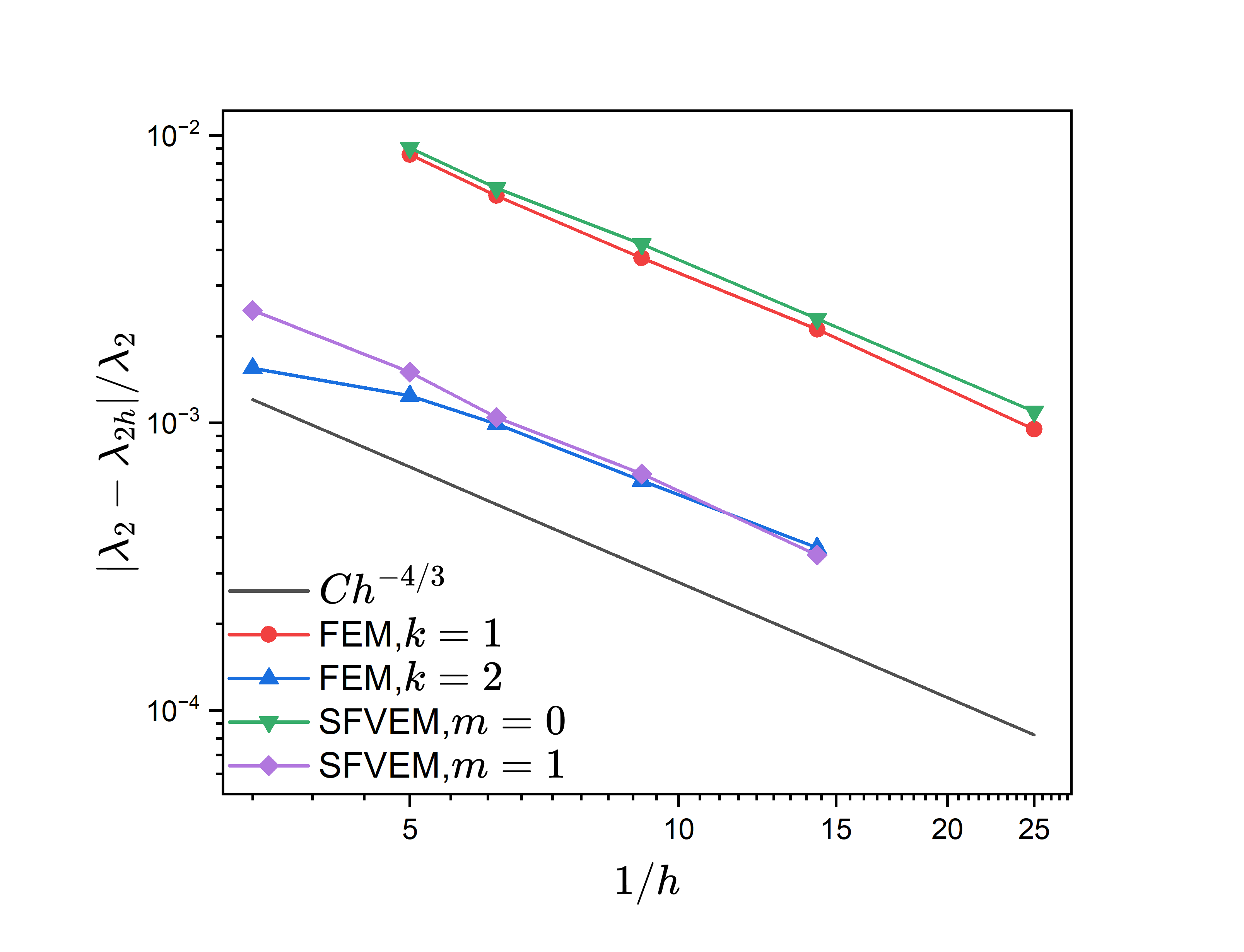} 
    \subcaption{}  
  \end{minipage}
  \caption{Relative Eigenvalue Errors. (a): \(\left|\lambda_{1}-\lambda_{1h}\right|/\lambda_1\). (b): \(\left|\lambda_{2}-\lambda_{2h}\right|/\lambda_2\).} 
  \label{fig:conv1}
\end{figure}

\begin{figure}[htbp]            
  \centering                   
  \begin{minipage}{0.5\textwidth}
    \centering
    \includegraphics[width=\textwidth]{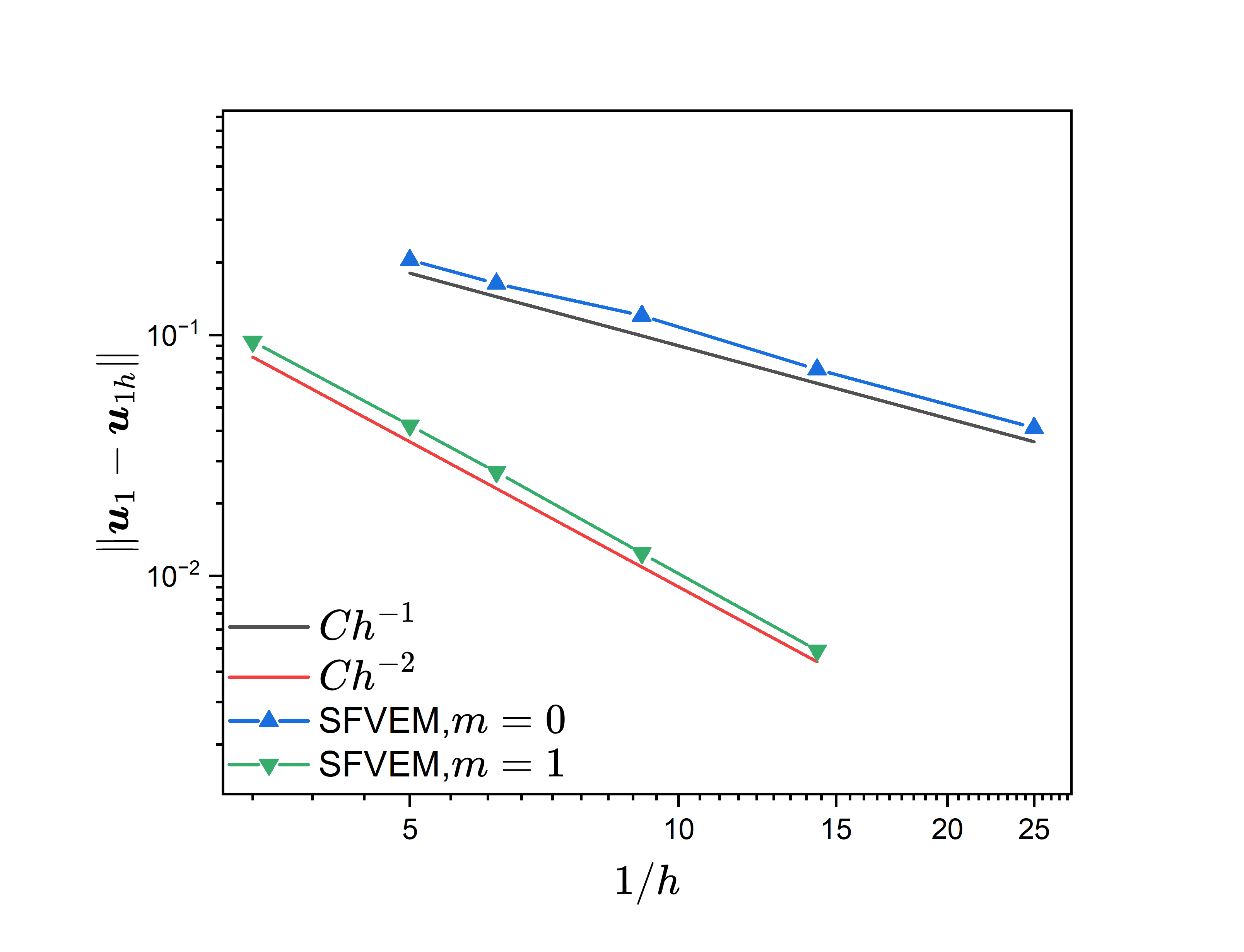} 
    \subcaption{}  
  \end{minipage}\hfill           
  \begin{minipage}{0.5\textwidth}
    \centering
    \includegraphics[width=\textwidth]{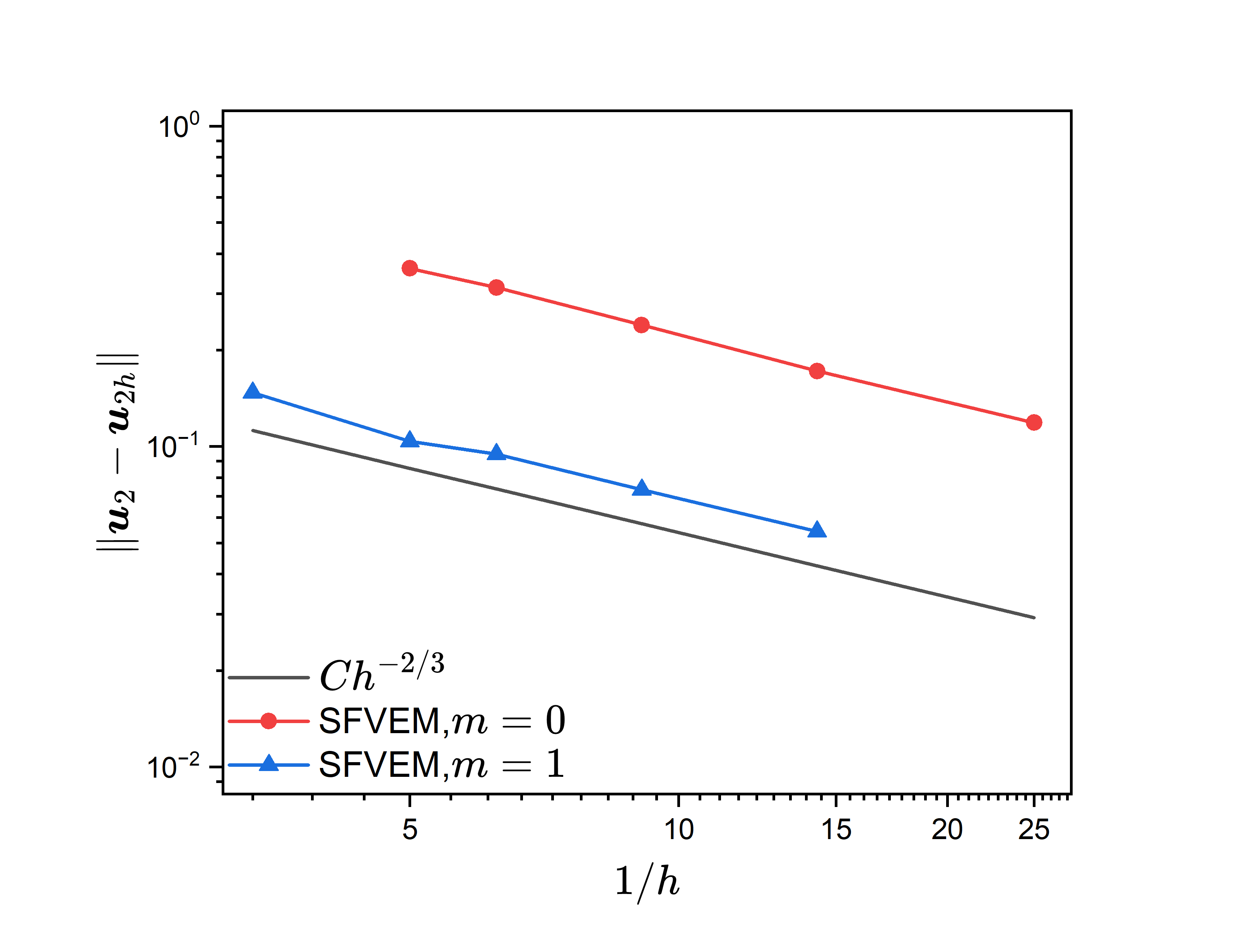} 
    \subcaption{}  
  \end{minipage}
  \caption{\(\mathrm{L}^2\)-norm Error. (a): \(\left\|\boldsymbol{u}_{1}-\boldsymbol{u}_{1h}\right\|\). (b): \(\left\|\boldsymbol{u}_{2}-\boldsymbol{u}_{2h}\right\|\).} 
  \label{fig:conv2}
\end{figure}

\section{Conclusion}

In this work, we have presented a novel stabilization-free virtual element method for general order \(\mathbf{H}(\operatorname{\mathbf{curl}})\) and \(\mathbf{H}(\operatorname{div})\)-conforming spaces. By constructing innovative serendipity projectors and corresponding serendipity spaces, our approach successfully addresses the fundamental challenges posed by stabilization terms in traditional virtual element methods.

Our key contributions include: (1) complete treatment of the De Rham complex chain in \(\mathbb{R}^3\) with exact commutativity of differential operators preserved; (2) construction of serendipity spaces that rigorously maintain boundary continuity, operator conformity, and B-compatibility; (3) achievement of stabilization-free computation with minimum number of degrees of freedom, significantly reducing computational overhead.

The theoretical analysis rigorously establishes optimal approximation properties of our method, while numerical experiments on Maxwell eigenvalue problems validate the theoretical convergence rates. For infinitely regular solutions, convergence rates of \(O(h^{m+1})\) are achieved for \(m\)-th order approximation spaces, while for singular solutions, convergence rates of \(O(h^s)\) are obtained, where \(s\) is determined by the solution regularity. These results are consistent with conventional finite element methods, yet our approach eliminates the difficulties associated with stabilization parameter selection.

The proposed stabilization-free virtual element method provides a theoretically rigorous and computationally efficient alternative for numerical solution of partial differential equations, particularly valuable for mixed formulations, eigenvalue problems, nonlinear problems, and applications requiring robust numerical behavior across diverse mesh configurations. This development opens new avenues for the advancement and application of virtual element methods in computational mathematics and engineering.

Finally, we must point out that though our method can be generalized to the full chain of De Rham complex on \(\mathbb{R}^d\), by recursively applying the same technique to the boundary of each dimension. However, on higher dimensional space, the VEM itself faces implementation complexity that outweighs its advantage. Furthermore, the rapid increase in amount of the boundary DoFs as the spatial dimension grows makes stabilization-free technique unacceptably inefficient. These practical considerations lead us to limit our discussion to the current setting.

\appendix
\section{Polynomial Properties}(\cite[Section 2.3]{daveigaInterpolationStabilityEstimates2022})

The following polynomial inverse estimates in a polytopal domain \(D \subset \mathbb{R}^d(d=2,3)\) are valid: for all \(p_k \in \mathbb{P}_k(D)\),
\begin{equation}\label{eq:polyinv1}
  \left\|p_k\right\|_{\partial D} \lesssim h_D^{-\frac{1}{2}}\left\|p_k\right\|_D, \quad\left|p_k\right|_{1, D} \lesssim h_D^{-1}\left\|p_k\right\|_D, \quad\left\|p_k\right\|_D \lesssim h_D^{-1}\left\|p_k\right\|_{-1, D}.
\end{equation}

Furthermore, for each piecewise polynomial \(p_k\) of degree at most \(k\) over \(\partial D\), we have:
\begin{equation}\label{eq:polyinv2}
  \left\|p_k\right\|_{\partial D} \lesssim h_D^{-\frac{1}{2}}\left\|p_k\right\|_{-\frac{1}{2}, \partial D}.
\end{equation}

Moreover, the following properties hold for the polynomial spaces:

Direct sum decomposition:
\begin{subequations}
  \begin{align}
    \mathbb{P}^3_{k,3} & =\nabla \mathbb{P}_{k+1,3}\oplus \boldsymbol{x}\times \mathbb{P}^3_{k-1,3},   \label{eq:pdcomp1}   \\
    \mathbb{P}^3_{k,3} & =\nabla \times \mathbb{P}^3_{k+1,3}\oplus \boldsymbol{x}\mathbb{P}_{k-1,3},   \label{eq:pdcomp2}   \\
    \mathbb{P}^2_{k,2} & =\nabla_{F} \mathbb{P}_{k+1,2}\oplus \boldsymbol{x}^\perp\mathbb{P}_{k-1,2},    \label{eq:pdcomp3} \\
    \mathbb{P}^2_{k,2} & =\nabla_{F} \times \mathbb{P}_{k+1,2}\oplus \boldsymbol{x}\mathbb{P}_{k-1,2}.\label{eq:pdcomp4}
  \end{align}
\end{subequations}

Isomorphism:
\begin{subequations}
  \begin{align}
     & \nabla:\mathbb{P}_{k+1}/\mathbb{R}\cong \nabla \mathbb{P}_{k+1}, \label{eq:iso1}                          \\
     & \nabla \times:\boldsymbol{x}\times \mathbb{P}_{k}^3\cong\nabla \times \mathbb{P}^3_{k+1}, \label{eq:iso2} \\
     & \nabla \cdot:\boldsymbol{x}\mathbb{P}_{k}\cong \mathbb{P}_{k}. \label{eq:iso3}
  \end{align}
\end{subequations}

Dimension:

Define:
\begin{equation}\label{eq:dimdef}
  \begin{aligned}
    \pi_{k,d}    & :=\binom{d+k}{k},         \\
    \gamma_{k,d} & :=\pi_{k+1,d}-1,          \\
    \rho_{k,3}   & :=3\pi_{k,3}-\pi_{k-1,3},
  \end{aligned}
\end{equation}
we have:
\begin{equation}
  \begin{aligned}
     & \dim\mathbb{P}_{k,d}=\pi_{k,d},                                                                \\
     & \dim\nabla \mathbb{P}_{k+1,3}=\gamma_{k,3},                                                    \\
     & \dim\nabla_{F} \mathbb{P}_{k+1,2}=\dim\nabla_{F} \times \mathbb{P}_{k+1,2}=\gamma_{k,2},       \\
     & \dim\nabla \times \mathbb{P}^3_{k+1,3}=\dim\boldsymbol{x}\times \mathbb{P}^3_{k,3}=\rho_{k,3}, \\
     & \dim\boldsymbol{x}\mathbb{P}_{k,3}=\pi_{k,3},                                                  \\
     & \dim\boldsymbol{x}\mathbb{P}_{k,2}=\dim\boldsymbol{x}^\perp\mathbb{P}_{k,2}=\pi_{k,2}.
  \end{aligned}
\end{equation}

\section{Trace Inequalities}(\cite[Section 2.4]{daveigaInterpolationStabilityEstimates2022})

The following trace inequalities are valid: given a polytopal domain \(D\), representing either an element or a face of the mesh, there hold:
\begin{subequations}
  \begin{align}
     & \|v\|_{\partial D} \lesssim h_D^{-\frac{1}{2}}\|v\|_D+h_D^{\delta-\frac{1}{2}}|v|_{\delta, D}                              & \forall v \in H^\delta(D), \frac{1}{2}<\delta<\frac{3}{2}\label{eq:trace1}     \\
     & |v|_{\varepsilon, \partial D} \lesssim h_D^{-\left(\varepsilon+\frac{1}{2}\right)}\|v\|_D+|v|_{\varepsilon+\frac{1}{2}, D} & \forall v \in H^{\varepsilon+\frac{1}{2}}(D), 0<\varepsilon<1\label{eq:trace2}
  \end{align}
\end{subequations}

Let \(F\) be a polygon and \(K\) be a polyhedron, respectively, representing either a face \(F\) or an element \(K\) of the mesh, thus satisfying the above assumptions (M). For \(\boldsymbol{v} \in \mathbf{H}\left(\operatorname{rot}_F, F\right)\), \(\boldsymbol{\phi} \in \mathbf{H}(\operatorname{div}, K), \boldsymbol{\psi} \in \mathbf{H}(\mathbf{curl}, K)\), and \(\boldsymbol{\chi} \in \mathbf{H}(\operatorname{div}, K) \cap \mathbf{H}(\mathbf{curl}, K)\), the following trace inequalities are valid:
\begin{subequations}
  \begin{align}
     & \left\|\boldsymbol{v} \cdot \mathbf{t}\right\|_{-\frac{1}{2}, \partial F} \lesssim\|\boldsymbol{v}\|_F+h_F\left\|\nabla_F \times \boldsymbol{v}\right\|_F\label{eq:vtrace1}                                                                                                                                      \\
     & \left\|\boldsymbol{\phi} \cdot \mathbf{n}\right\|_{-\frac{1}{2}, \partial K} \lesssim\|\boldsymbol{\phi}\|_K+h_K\|\nabla \cdot \boldsymbol{\phi}\|_K\label{eq:vtrace2}                                                                                                                                           \\
     & \left\|\boldsymbol{\psi} \times \mathbf{n}\right\|_{-\frac{1}{2}, \partial K} \lesssim\|\boldsymbol{\psi}\|_K+h_K\|\nabla \times \boldsymbol{\psi}\|_K\label{eq:vtrace3}                                                                                                                                         \\
     & \left\|\boldsymbol{\chi} \times \mathbf{n}\right\|_{\partial K} \lesssim h_K^{-\frac{1}{2}}\|\boldsymbol{\chi}\|_K+h_K^{\frac{1}{2}}\|\nabla \cdot \boldsymbol{\chi}\|_K+h_K^{\frac{1}{2}}\|\nabla \times \boldsymbol{\chi}\|_K+\left\|\boldsymbol{\chi} \cdot \mathbf{n}\right\|_{\partial K}\label{eq:vtrace4}
  \end{align}
\end{subequations}

\bibliographystyle{siamplain}
\bibliography{references}

\end{document}


\maketitle

\section{Stable Polynomial Projections}

This section proves \cref{thm:st} by first reducing stable polynomial
projection to two structural ingredients: moment compatibility across orders
and eventual separation of the preserved DoFs by polynomial moments. For the
face space, the required moment identities come from integration by parts, and
the large-order injectivity is obtained from the finite-degree separation of
the preserved boundary data by polynomial moments together with volume-bubble
tests. For the edge space, the
argument combines a low-order/high-order splitting with rotational boundary
testing to separate tangential traces face by face, and then closes the last
step by a low-order nonreentry argument based on the same face-supported lifts
together with a zero-boundary bubble correction that removes the bulk curl
term.

\begin{proposition}[Generic DoF Criterion for Stable Polynomial Projection]
  \label{prop:supp-criterion}
  Let \(\mathbf{V}^{\mathrm{S}}_{l}\) denote either the serendipity face space
  \(\mathbf{V}^{\mathrm{Sf}}_{k,l}(K)\) or the serendipity edge space
  \(\mathbf{V}^{\mathrm{Se}}_{k,l}(K)\), and assume that for every admissible
  \(l\) the preserved DoFs
  \(\left\{\mathcal{F}_{i}\right\}_{i=1}^{N_{\mathrm{S}}}\) determine a unique
  function in \(\mathbf{V}^{\mathrm{S}}_{l}\). Let
  \(\mathscr{R}_{l}:\mathbb{R}^{N_{\mathrm{S}}}\rightarrow \mathbf{V}^{\mathrm{S}}_{l}\)
  denote the corresponding reconstruction operator, namely
  \(\mathcal{D}_{N_{\mathrm{S}}}(\mathscr{R}_{l}\mathbf{d})=\mathbf{d}\).

  Assume moreover that:
  \begin{enumerate}
    \def\labelenumi{(\roman{enumi})}
    \item for every \(m\leq l\), every \(\mathbf{d}\in\mathbb{R}^{N_{\mathrm{S}}}\), and every
          \(\boldsymbol{p}_{m}\in\mathbb{P}^{3}_{m}(K)\),
          \begin{equation}
            \left(\mathscr{R}_{l}\mathbf{d},\boldsymbol{p}_{m}\right)_{K}=\left(\mathscr{R}_{m}\mathbf{d},\boldsymbol{p}_{m}\right)_{K};
          \end{equation}
    \item for every nonzero \(\mathbf{d}\in\mathbb{R}^{N_{\mathrm{S}}}\), there exist
          some \(l\) and some \(\boldsymbol{p}_{l}\in\mathbb{P}^{3}_{l}(K)\) such that
          \begin{equation}
            \left(\mathscr{R}_{l}\mathbf{d},\boldsymbol{p}_{l}\right)_{K}\neq 0.
          \end{equation}
  \end{enumerate}

  Then there exists \(l_{*}\) such that \(\bfitgreek{\Pi}_{l}\) is injective on
  \(\mathbf{V}^{\mathrm{S}}_{l}\) for every \(l\geq l_{*}\). Equivalently, for each
  such \(l\) there exists a constant \(c_{l}>0\) satisfying
  \begin{equation}
    c_{l}\|\boldsymbol{v}\|_{K}\leq\left\|\bfitgreek{\Pi}_{l}\boldsymbol{v}\right\|_{K}\leq\|\boldsymbol{v}\|_{K}
    \quad \forall \boldsymbol{v}\in\mathbf{V}^{\mathrm{S}}_{l}.
  \end{equation}
\end{proposition}

\begin{proof}
  For each admissible \(l\), define the linear map
  \(\mathscr{M}_{l}:\mathbb{R}^{N_{\mathrm{S}}}\rightarrow \left(\mathbb{P}^{3}_{l}(K)\right)'\) by
  \begin{equation}
    \mathscr{M}_{l}(\mathbf{d})(\boldsymbol{p}_{l}):=\left(\mathscr{R}_{l}\mathbf{d},\boldsymbol{p}_{l}\right)_{K}
    \quad \forall \boldsymbol{p}_{l}\in\mathbb{P}^{3}_{l}(K).
  \end{equation}
  Let \(\mathscr{K}_{l}:=\ker \mathscr{M}_{l}\). By assumption (i), for every
  \(m\leq l\) and every \(\mathbf{d}\in \mathscr{K}_{l}\), we have
  \begin{equation}
    \left(\mathscr{R}_{m}\mathbf{d},\boldsymbol{p}_{m}\right)_{K}=\left(\mathscr{R}_{l}\mathbf{d},\boldsymbol{p}_{m}\right)_{K}=0
    \quad \forall \boldsymbol{p}_{m}\in\mathbb{P}^{3}_{m}(K),
  \end{equation}
  hence \(\mathscr{K}_{l}\subseteq \mathscr{K}_{m}\). Therefore
  \(\{\mathscr{K}_{l}\}_{l}\) is a descending chain of subspaces of the fixed finite
  dimensional space \(\mathbb{R}^{N_{\mathrm{S}}}\).

  By assumption (ii), for every nonzero \(\mathbf{d}\in\mathbb{R}^{N_{\mathrm{S}}}\)
  there exists some \(l\) such that \(\mathbf{d}\notin \mathscr{K}_{l}\). Hence
  \begin{equation}
    \bigcap_{l}\mathscr{K}_{l}=\{0\}.
  \end{equation}
  Since a descending chain of subspaces in a finite dimensional space must
  stabilize, there exists \(l_{*}\) such that
  \begin{equation}
    \mathscr{K}_{l}=\{0\}\quad \forall l\geq l_{*}.
  \end{equation}

  Now let \(l\geq l_{*}\) and \(\boldsymbol{v}\in\mathbf{V}^{\mathrm{S}}_{l}\). Write
  \(\boldsymbol{v}=\mathscr{R}_{l}\mathbf{d}\) with
  \(\mathbf{d}=\mathcal{D}_{N_{\mathrm{S}}}(\boldsymbol{v})\). Then
  \(\bfitgreek{\Pi}_{l}\boldsymbol{v}=0\) is equivalent to
  \begin{equation}
    \left(\boldsymbol{v},\boldsymbol{p}_{l}\right)_{K}=0
    \quad \forall \boldsymbol{p}_{l}\in\mathbb{P}^{3}_{l}(K),
  \end{equation}
  namely \(\mathbf{d}\in\mathscr{K}_{l}\). Since \(\mathscr{K}_{l}=\{0\}\), we obtain
  \(\mathbf{d}=0\), hence \(\boldsymbol{v}=0\). Therefore \(\bfitgreek{\Pi}_{l}\) is injective on
  \(\mathbf{V}^{\mathrm{S}}_{l}\).

  Finally, for each fixed \(l\geq l_{*}\), injectivity of \(\bfitgreek{\Pi}_{l}\) and
  finite dimensionality of \(\mathbf{V}^{\mathrm{S}}_{l}\) imply the norm equivalence
  \begin{equation}
    c_{l}\|\boldsymbol{v}\|_{K}\leq\left\|\bfitgreek{\Pi}_{l}\boldsymbol{v}\right\|_{K}\leq\|\boldsymbol{v}\|_{K}
    \quad \forall \boldsymbol{v}\in\mathbf{V}^{\mathrm{S}}_{l},
  \end{equation}
  for some constant \(c_{l}>0\).
\end{proof}

We now prove \cref{thm:st}.

\begin{proof}
  The face-space assertion is exactly \cref{thm:stf}. The edge-space assertion
  follows from \cref{cor:edge-full-space-stability}.
\end{proof}

\begin{proposition}[Moment Compatibility Across Orders]
  \label{prop:moment-compatibility}
  Let \(\mathbf{d}\in\mathbb{R}^{N_{\mathrm{S}}}\) be a preserved DoF vector.

  For the face space, let \(\boldsymbol{v}_{l}=\mathscr{R}_{l}\mathbf{d}\in
  \mathbf{V}^{\mathrm{Sf}}_{k,l}(K)\) and \(\boldsymbol{p}_{m}\in\mathbb{P}^{3}_{m}(K)\)
  with decomposition
  \(\boldsymbol{p}_{m}=\nabla q_{m+1}+\boldsymbol{x}_{K}\times \boldsymbol{r}_{m-1}\).
  Then
  \begin{equation}\label{eq:stbf}
    \begin{aligned}
      \left(\boldsymbol{v}_{l},\boldsymbol{p}_{m}\right)_{K}
       & =\left(\boldsymbol{v}_{l},\nabla q_{m+1}+\boldsymbol{x}_{K}\times\boldsymbol{r}_{m-1}\right)_{K} \\
       & =-\left(\nabla \cdot \boldsymbol{v}_{l},q_{m+1}\right)_{K}+\left(\boldsymbol{v}_{l}\cdot \boldsymbol{n},q_{m+1}\right)_{\partial K}+\left(\boldsymbol{v}_{l},\boldsymbol{x}_{K}\times\boldsymbol{r}_{m-1}\right)_{K},
    \end{aligned}
  \end{equation}
  and the right-hand side depends only on the preserved DoFs. Hence, for every
  admissible \(m\leq l\), the moment \(\left(\boldsymbol{v}_{l},\boldsymbol{p}_{m}\right)_{K}\)
  is independent of \(l\).

  For the edge space, let \(\boldsymbol{v}_{l}=\mathscr{R}_{l}\mathbf{d}\in
  \mathbf{V}^{\mathrm{Se}}_{k,l}(K)\) and \(\boldsymbol{p}_{m}\in\mathbb{P}^{3}_{m}(K)\)
  with decomposition
  \(\boldsymbol{p}_{m}=\nabla \times \boldsymbol{q}_{m+1}+\boldsymbol{x}_{K}r_{m-1}\).
  Writing on each face
  \(\boldsymbol{n}\times\boldsymbol{q}_{m+1}=\nabla_{F}\times a_{F,m+2}+\boldsymbol{x}_{F}b_{F,m}\),
  we have
  \begin{equation}\label{eq:stbe}
    \begin{aligned}
      \left(\boldsymbol{v}_{l},\boldsymbol{p}_{m}\right)_{K}
       & =\left(\boldsymbol{v}_{l},\nabla \times \boldsymbol{q}_{m+1}+\boldsymbol{x}_{K}r_{m-1}\right)_{K} \\
       & =\sum_{F\subseteq\partial K}\left(\left(\nabla_F \times \boldsymbol{v}_{l},a_{F,m+2}\right)_{F}-\left(\boldsymbol{v}_{l}\cdot \boldsymbol{t},a_{F,m+2}\right)_{\partial F}+\left(\boldsymbol{v}_{l},\boldsymbol{x}_{F}b_{F,m}\right)_{F}\right) \\
       & \quad +\left(\nabla \times \boldsymbol{v}_{l},\boldsymbol{q}_{m+1}\right)_{K}+\left(\boldsymbol{v}_{l},\boldsymbol{x}_{K}r_{m-1}\right)_{K},
    \end{aligned}
  \end{equation}
  and again the right-hand side depends only on the preserved DoFs. Hence, for
  every admissible \(m\leq l\), the moment \(\left(\boldsymbol{v}_{l},\boldsymbol{p}_{m}\right)_{K}\)
  is independent of \(l\).
\end{proposition}

\begin{proof}
  For the face space, identity \cref{eq:stbf} is the integration-by-parts formula
  applied to the decomposition
  \(\boldsymbol{p}_{m}=\nabla q_{m+1}+\boldsymbol{x}_{K}\times
  \boldsymbol{r}_{m-1}\). The divergence term, the normal-trace term, and the
  interior moment against \(\boldsymbol{x}_{K}\times\boldsymbol{r}_{m-1}\) are all
  determined by the preserved DoFs of
  \(\mathbf{V}^{\mathrm{Sf}}_{k,l}(K)\). Hence the right-hand side depends only on
  \(\mathbf{d}\), and is therefore independent of the reconstruction order \(l\)
  whenever \(m\leq l\).

  For the edge space, identity \cref{eq:stbe} is obtained by first integrating
  by parts in the bulk, and then decomposing the tangential trace of
  \(\boldsymbol{q}_{m+1}\) on each face into the form
  \(\boldsymbol{n}\times\boldsymbol{q}_{m+1}=\nabla_{F}\times a_{F,m+2}+\boldsymbol{x}_{F}b_{F,m}\),
  followed by face-wise integration by parts. The resulting terms involve only
  the preserved face and cell moments of
  \(\mathbf{V}^{\mathrm{Se}}_{k,l}(K)\), hence again depend only on
  \(\mathbf{d}\). Therefore the moments are independent of \(l\) for every
  admissible \(m\leq l\).
\end{proof}

\begin{lemma}[Finite-Degree Separation of Polynomial Normal Traces]
  \label{lem:trace-separation}
  Define the boundary polynomial space
  \begin{equation}
    \mathbb{B}_{k}(\partial K):=\left\{g\in L^2(\partial K)\mid g_{\mid F}\in\mathbb{P}_{k}(F)\ \forall F\subseteq\partial K\right\}.
  \end{equation}
  Then there exists an integer \(m_{\partial}=m_{\partial}(K,k)\) such that, for
  every \(m\geq m_{\partial}\),
  \begin{equation}
    \left(g,q_{m}\right)_{\partial K}=0\ \forall q_{m}\in\mathbb{P}_{m}(K)
    \implies g=0\quad \forall g\in\mathbb{B}_{k}(\partial K).
  \end{equation}
\end{lemma}

\begin{proof}
  Let
  \(\mathscr{T}_{m}:=\left\{q_{m\mid \partial K}\mid q_{m}\in\mathbb{P}_{m}(K)\right\}\subseteq C(\partial K)\).
  The union \(\bigcup_{m}\mathscr{T}_{m}\) is an algebra of continuous functions on
  the compact set \(\partial K\), contains the constants, and separates points.
  Hence, by the Stone-Weierstrass theorem,
  \(\bigcup_{m}\mathscr{T}_{m}\) is dense in \(C(\partial K)\), therefore also dense
  in \(L^2(\partial K)\).

  Define
  \begin{equation}
    \mathscr{Z}_{m}:=\left\{g\in\mathbb{B}_{k}(\partial K)\mid \left(g,q_{m}\right)_{\partial K}=0\ \forall q_{m}\in\mathbb{P}_{m}(K)\right\}.
  \end{equation}
  Then \(\{\mathscr{Z}_{m}\}_{m}\) is a descending chain of subspaces of the fixed
  finite-dimensional space \(\mathbb{B}_{k}(\partial K)\). If
  \(g\in\bigcap_{m}\mathscr{Z}_{m}\), then \(g\) is orthogonal to a dense subset of
  \(L^2(\partial K)\), hence \(g=0\). Therefore
  \(\bigcap_{m}\mathscr{Z}_{m}=\{0\}\). Since descending chains of subspaces in a
  finite-dimensional space stabilize, there exists
  \(m_{\partial}=m_{\partial}(K,k)\) such that \(\mathscr{Z}_{m}=\{0\}\) for all
  \(m\geq m_{\partial}\).
\end{proof}

\subsection{Face Spaces}

\begin{theorem}[Existence of Stable Polynomial Projection for Face Spaces]
  \label{thm:stf}
  There exists an integer \(l_{\mathrm{f},*}=l_{\mathrm{f},*}(K,k)\) such that
  \(\bfitgreek{\Pi}_{l}\) is injective on \(\mathbf{V}^{\mathrm{Sf}}_{k,l}(K)\) for every
  \(l\geq l_{\mathrm{f},*}\). Equivalently, for each such \(l\) there exists a
  constant \(c_{l}>0\) satisfying
  \begin{equation}
    c_{l}\|\boldsymbol{v}\|_{K}\leq\left\|\bfitgreek{\Pi}_{l}\boldsymbol{v}\right\|_{K}\leq\|\boldsymbol{v}\|_{K}
    \quad \forall \boldsymbol{v}\in\mathbf{V}^{\mathrm{Sf}}_{k,l}(K).
  \end{equation}
\end{theorem}

\begin{proof}
  Let \(\eta\) be the number of faces of \(K\), let
  \(m_{\partial}=m_{\partial}(K,k)\) be given by
  \cref{lem:trace-separation}, and choose
  \begin{equation}
    l_{\mathrm{f},*}\geq \max\left\{m_{\partial}-1,2\eta+k-2\right\}.
  \end{equation}

  Let \(l\geq l_{\mathrm{f},*}\) and suppose that
  \(\boldsymbol{v}\in\mathbf{V}^{\mathrm{Sf}}_{k,l}(K)\) satisfies
  \(\bfitgreek{\Pi}_{l}\boldsymbol{v}=0\). Then
  \begin{equation}
    \left(\boldsymbol{v},\boldsymbol{p}_{l}\right)_{K}=0
    \quad \forall \boldsymbol{p}_{l}\in\mathbb{P}^{3}_{l}(K).
  \end{equation}
  Using the decomposition
  \(\mathbb{P}^{3}_{l}(K)=\nabla \mathbb{P}_{l+1}(K)\oplus \boldsymbol{x}_{K}\times \mathbb{P}^{3}_{l-1}(K)\),
  we obtain
  \begin{subequations}
    \begin{align}
      \left(\boldsymbol{v},\boldsymbol{x}_{K}\times \boldsymbol{r}_{l-1}\right)_{K}=0
      \quad & \forall \boldsymbol{r}_{l-1}\in\mathbb{P}^{3}_{l-1}(K), \label{eq:supp-stfproof1} \\
      -\left(\nabla\cdot \boldsymbol{v},q_{l+1}\right)_{K}+\left(\boldsymbol{v}\cdot \boldsymbol{n},q_{l+1}\right)_{\partial K}=0
      \quad & \forall q_{l+1}\in\mathbb{P}_{l+1}(K). \label{eq:supp-stfproof2}
    \end{align}
  \end{subequations}

  Set \(d:=\nabla\cdot \boldsymbol{v}\in\mathbb{P}_{k-1}(K)\). For each face
  \(F\subseteq\partial K\), let \(\ell_{F}\) be an affine function satisfying
  \(\ell_{F\mid F}=0\), and define
  \begin{equation}
    b_{K}:=\prod_{F\subseteq\partial K}\ell_{F}^{2}.
  \end{equation}
  Then \(b_{K}\) is a polynomial vanishing on \(\partial K\) and strictly positive
  in the interior of \(K\). Since
  \(\deg(b_{K}d)\leq 2\eta+k-1\leq l+1\), taking
  \(q_{l+1}=b_{K}d\) in \cref{eq:supp-stfproof2} yields
  \begin{equation}
    0=-\left(d,b_{K}d\right)_{K}.
  \end{equation}
  Hence \(d=0\), namely
  \begin{equation}
    \nabla\cdot \boldsymbol{v}=0.
  \end{equation}

  Therefore \cref{eq:supp-stfproof2} reduces to
  \begin{equation}
    \left(\boldsymbol{v}\cdot \boldsymbol{n},q_{l+1}\right)_{\partial K}=0
    \quad \forall q_{l+1}\in\mathbb{P}_{l+1}(K).
  \end{equation}
  Since \(\boldsymbol{v}\cdot \boldsymbol{n}\in\mathbb{B}_{k}(\partial K)\) and
  \(l+1\geq m_{\partial}\), \cref{lem:trace-separation} implies
  \begin{equation}
    \boldsymbol{v}\cdot \boldsymbol{n}=0\quad \text{on }\partial K.
  \end{equation}

  Finally, applying \cref{eq:boundf} to
  \(\boldsymbol{v}\in\mathbf{V}^{\mathrm{Sf}}_{k,l}(K)\subseteq
  \mathbf{V}^{\mathrm{f}}_{k,k-1,l-1}(K)\), and using
  \cref{eq:supp-stfproof1}, we obtain
  \begin{equation}
    \|\boldsymbol{v}\|_{K}\lesssim h_{K}\|\nabla\cdot \boldsymbol{v}\|_{K}+h^{\frac{1}{2}}_{K}\|\boldsymbol{v}\cdot \boldsymbol{n}\|_{\partial K}+\sup_{\boldsymbol{r}_{l-1}\in\mathbb{P}^{3}_{l-1}(K)}\frac{\left(\boldsymbol{v},\boldsymbol{x}_{K}\times \boldsymbol{r}_{l-1}\right)_{K}}{\left\|\boldsymbol{x}_{K}\times \boldsymbol{r}_{l-1}\right\|_{K}}=0.
  \end{equation}
  Thus \(\boldsymbol{v}=0\), so \(\bfitgreek{\Pi}_{l}\) is injective on
  \(\mathbf{V}^{\mathrm{Sf}}_{k,l}(K)\). The asserted norm equivalence follows from
  finite dimensionality.
\end{proof}

\subsection{Edge Spaces: Facewise Tangential Separation}

Following the notation of the main text, for every face
\(F\subseteq\partial K\) and for the whole boundary \(\partial K\), we write
\(\boldsymbol{v}^{F}\) and \(\boldsymbol{v}^{\partial K}\) for the tangential traces
of \(\boldsymbol{v}\) on \(F\) and \(\partial K\), respectively.

For a face \(F\subseteq\partial K\), define the tangential polynomial space
\begin{equation}
  \mathbb{P}_{m,t}(F):=\left\{\bfitgreek{\psi}_{m}\in\mathbb{P}^{3}_{m}(F)\mid \bfitgreek{\psi}_{m}\cdot \boldsymbol{n}_{F}=0\right\}.
\end{equation}

\begin{lemma}[Face-Local Tangential Polynomial Lifting]
  \label{lem:face-lift}
  Let \(F\subseteq\partial K\) be a face, let \(\eta\) be the number of faces of
  \(K\), and let the affine functions \(\ell_{G}\) be as in the proof of
  \cref{thm:stf}. Define
  \begin{equation}
    b_{F}:=\prod_{G\subseteq\partial K,\ G\neq F}\ell_{G\mid F}^{2}\in \mathbb{P}_{2(\eta-1)}(F).
  \end{equation}
  Then \(b_{F}=0\) on \(\partial F\) and \(b_{F}>0\) in the interior of \(F\).
  Moreover, for every \(\bfitgreek{\psi}_{m}\in\mathbb{P}_{m,t}(F)\), there exists
  \(\boldsymbol{q}_{m+2(\eta-1)}\in\mathbb{P}^{3}_{m+2(\eta-1)}(K)\) such that
  \begin{equation}
    \boldsymbol{q}_{m+2(\eta-1)\mid G}=0\quad \forall G\subseteq\partial K,\ G\neq F,
    \qquad
    \boldsymbol{q}_{m+2(\eta-1)\mid F}=b_{F}\bfitgreek{\psi}_{m}.
  \end{equation}
\end{lemma}

\begin{proof}
  Since \(F\) is planar, every tangential polynomial field
  \(\bfitgreek{\psi}_{m}\in\mathbb{P}_{m,t}(F)\) admits a polynomial extension
  \(\mathscr{E}_{F}\bfitgreek{\psi}_{m}\in\mathbb{P}^{3}_{m}(K)\) that is constant
  along lines orthogonal to \(F\). In particular,
  \(\mathscr{E}_{F}\bfitgreek{\psi}_{m\mid F}=\bfitgreek{\psi}_{m}\).

  Define
  \begin{equation}
    \boldsymbol{q}_{m+2(\eta-1)}:=\left(\prod_{G\subseteq\partial K,\ G\neq F}\ell_{G}^{2}\right)\mathscr{E}_{F}\bfitgreek{\psi}_{m}.
  \end{equation}
  Its degree is \(m+2(\eta-1)\). On the face \(F\), the product restricts to
  \(b_{F}\), hence
  \(\boldsymbol{q}_{m+2(\eta-1)\mid F}=b_{F}\bfitgreek{\psi}_{m}\). On every other
  face \(G\neq F\), the factor \(\ell_{G}^{2}\) vanishes identically, so
  \(\boldsymbol{q}_{m+2(\eta-1)\mid G}=0\).

  Finally, if \(x\in\partial F\), then \(x\in F\cap G\) for some face \(G\neq F\),
  hence \(\ell_{G\mid F}(x)=0\) and therefore \(b_{F}(x)=0\). If
  \(x\in F\setminus\partial F\), then \(x\) lies in the interior of \(K\) relative
  to every face \(G\neq F\), so \(\ell_{G\mid F}(x)>0\), which implies
  \(b_{F}(x)>0\).
\end{proof}

\begin{lemma}[Density of Weighted Tangential Polynomial Fields on a Face]
  \label{lem:face-weighted-density}
  For every face \(F\subseteq\partial K\), the space
  \begin{equation}
    \mathscr{T}_{F}:=\bigcup_{m\geq 0} b_{F}\mathbb{P}_{m,t}(F)
  \end{equation}
  is dense in \(\mathbf{L}^{2}_{t}(F)\).
\end{lemma}

\begin{proof}
  Since \(C_{c}^{\infty}(F)^{3}\cap \mathbf{L}^{2}_{t}(F)\) is dense in
  \(\mathbf{L}^{2}_{t}(F)\), it suffices to approximate a tangential field
  \(\bfitgreek{\varphi}\in C_{c}^{\infty}(F)^{3}\cap \mathbf{L}^{2}_{t}(F)\).
  Let \(F'\Subset F\) be a compact set containing the support of
  \(\bfitgreek{\varphi}\). Since \(b_{F}>0\) in the interior of \(F\), there exists
  \(c_{F'}>0\) such that \(b_{F}\geq c_{F'}\) on \(F'\). Therefore
  \(\bfitgreek{\psi}:=\bfitgreek{\varphi}/b_{F}\in C(F')^{3}\cap
  \mathbf{L}^{2}_{t}(F)\).

  Choose an orthonormal basis \((\bfitgreek{\tau}_{1},\bfitgreek{\tau}_{2})\) of the
  tangent plane of \(F\), and write
  \(\bfitgreek{\psi}=\psi_{1}\bfitgreek{\tau}_{1}+\psi_{2}\bfitgreek{\tau}_{2}\).
  By the Weierstrass approximation theorem, there exist scalar polynomials
  \(p_{1,m},p_{2,m}\in\mathbb{P}_{m}(F)\) such that
  \(p_{i,m}\to \psi_{i}\) uniformly on \(F'\). Setting
  \begin{equation}
    \bfitgreek{\psi}_{m}:=p_{1,m}\bfitgreek{\tau}_{1}+p_{2,m}\bfitgreek{\tau}_{2}\in \mathbb{P}_{m,t}(F),
  \end{equation}
  we obtain
  \begin{equation}
    \|b_{F}\bfitgreek{\psi}_{m}-\bfitgreek{\varphi}\|_{F}=\|b_{F}(\bfitgreek{\psi}_{m}-\bfitgreek{\psi})\|_{F'}\to 0.
  \end{equation}
  This proves the claim.
\end{proof}

\begin{proposition}[Finite-Degree Separation on a Face]
  \label{prop:face-separation}
  For each face \(F\subseteq\partial K\), there exists an integer
  \(m_{F,*}=m_{F,*}(F,k)\) such that, for every admissible face order \(r\), every
  \(m\geq m_{F,*}\), and every
  \(\boldsymbol{w}\in\mathbf{V}^{\mathrm{Se}}_{k,r}(F)\),
  \begin{equation}
    \left(\boldsymbol{w},b_{F}\bfitgreek{\psi}_{m}\right)_{F}=0
    \quad \forall \bfitgreek{\psi}_{m}\in\mathbb{P}_{m,t}(F)
    \implies \boldsymbol{w}=0.
  \end{equation}
\end{proposition}

\begin{proof}
  Let \(N_{F}\) be the number of preserved face DoFs, and let
  \(\mathscr{R}^{F}_{r}:\mathbb{R}^{N_{F}}\to\mathbf{V}^{\mathrm{Se}}_{k,r}(F)\)
  denote the corresponding face-local reconstruction operator.

  For each \(m\geq 0\), define the linear map
  \(\mathscr{M}^{F}_{m}:\mathbb{R}^{N_{F}}\to (\mathbb{P}_{m,t}(F))'\) by
  \begin{equation}
    \mathscr{M}^{F}_{m}(\mathbf{d}_{F})(\bfitgreek{\psi}_{m}):=\left(\mathscr{R}^{F}_{r}\mathbf{d}_{F},b_{F}\bfitgreek{\psi}_{m}\right)_{F},
    \quad \forall \bfitgreek{\psi}_{m}\in\mathbb{P}_{m,t}(F),
  \end{equation}
  where \(r\) is any admissible face order such that
  \(m+2(\eta-1)\leq r\). Since \(b_{F}\bfitgreek{\psi}_{m}\) is a polynomial
  tangential field on the polygon \(F\), the same lower-order moment
  compatibility argument used in \cref{prop:moment-compatibility} shows that the
  right-hand side is independent of the chosen \(r\).

  Let \(\mathscr{K}^{F}_{m}:=\ker\mathscr{M}^{F}_{m}\). Then
  \(\{\mathscr{K}^{F}_{m}\}_{m}\) is a descending chain of subspaces of the fixed
  finite-dimensional space \(\mathbb{R}^{N_{F}}\). If
  \(\mathbf{d}_{F}\in \bigcap_{m}\mathscr{K}^{F}_{m}\), then for every admissible
  \(r\) the reconstructed field
  \(\boldsymbol{w}_{r}:=\mathscr{R}^{F}_{r}\mathbf{d}_{F}\) is orthogonal to
  \(\mathscr{T}_{F}\). By \cref{lem:face-weighted-density}, this implies
  \(\boldsymbol{w}_{r}=0\) in \(\mathbf{L}^{2}_{t}(F)\). Since the preserved face DoFs
  determine a unique function in \(\mathbf{V}^{\mathrm{Se}}_{k,r}(F)\), we obtain
  \(\mathbf{d}_{F}=0\). Therefore
  \begin{equation}
    \bigcap_{m}\mathscr{K}^{F}_{m}=\{0\}.
  \end{equation}
  Since descending chains of subspaces in a finite-dimensional space stabilize,
  there exists \(m_{F,*}\) such that
  \(\mathscr{K}^{F}_{m}=\{0\}\) for all \(m\geq m_{F,*}\). This is exactly the
  asserted separation property.
\end{proof}

\subsection{Full Edge Space}
\label{sec:edge-draft-route}

We now turn to the full edge space. After splitting by the low-order edge
projector, reduced-kernel fields automatically satisfy a family of rotational
boundary orthogonality relations. The same rotational family yields both the
boundary vanishing step and the low-order nonreentry step, so the full proof
closes entirely in this subsection.

For an admissible order \(l\), define the reduced high-order kernel
\begin{equation}
  \mathcal{Z}_{l}(K):=\left\{\boldsymbol{z}\in\mathbf{V}^{\mathrm{Se}}_{k,l}(K)\mid
  \bfitgreek{\Pi}^{\mathrm{Se}}_{k}\boldsymbol{z}=0,
  \ \bfitgreek{\Pi}_{l}\boldsymbol{z}=0\right\}.
\end{equation}

\begin{proposition}[Rotational Boundary Testing Implies Vanishing Tangential Trace]
  \label{prop:edge-direct-reduction}
  There exists an integer \(l_{\mathrm{t},*}=l_{\mathrm{t},*}(K,k)\) such that,
  for every admissible order \(l\geq l_{\mathrm{t},*}\), if
  \(\boldsymbol{z}\in\mathbf{V}^{\mathrm{Se}}_{k,l}(K)\) satisfies
  \begin{equation}
    \left(\boldsymbol{z}^{\partial K},\boldsymbol{n}\times
    \left(\boldsymbol{x}_{K}\times\boldsymbol{r}_{l-2}\right)\right)_{\partial K}=0
    \quad \forall \boldsymbol{r}_{l-2}\in\mathbb{P}^{3}_{l-2}(K),
  \end{equation}
  then
  \begin{equation}
    \boldsymbol{z}^{\partial K}=0\quad\text{on }\partial K.
  \end{equation}
\end{proposition}

\begin{proof}
  Let \(\eta\) be the number of faces of \(K\), and choose
  \begin{equation}
    l_{\mathrm{t},*}\geq
    \max_{F\subseteq\partial K}\left(m_{F,*}+2\eta\right).
  \end{equation}

  Let \(l\geq l_{\mathrm{t},*}\), let
  \(\boldsymbol{z}\in\mathbf{V}^{\mathrm{Se}}_{k,l}(K)\) satisfy the stated
  boundary orthogonality, and fix a face \(F\subseteq\partial K\). Set
  \begin{equation}
    c_{F}:=\boldsymbol{x}_{K}\cdot\boldsymbol{n}_{F},
    \qquad
    m:=l-2\eta\geq m_{F,*}.
  \end{equation}
  Here \(c_{F}\) is a nonzero constant on \(F\). Let
  \(\bfitgreek{\psi}_{m}\in\mathbb{P}_{m,t}(F)\). By \cref{lem:face-lift},
  there exists \(\widetilde{\boldsymbol{r}}_{l-2}\in\mathbb{P}^{3}_{l-2}(K)\) such that
  \begin{equation}
    \widetilde{\boldsymbol{r}}_{l-2\mid G}=0
    \quad \forall G\subseteq\partial K,\ G\neq F,
    \qquad
    \widetilde{\boldsymbol{r}}_{l-2\mid F}=b_{F}\bfitgreek{\psi}_{m}.
  \end{equation}
  Define
  \begin{equation}
    \boldsymbol{r}_{l-2}:=-c_{F}^{-1}\widetilde{\boldsymbol{r}}_{l-2}
    \in\mathbb{P}^{3}_{l-2}(K).
  \end{equation}
  On every face \(G\neq F\), one has
  \begin{equation}
    \boldsymbol{n}_{G}\times\left(\boldsymbol{x}_{K}\times\boldsymbol{r}_{l-2}\right)=0,
  \end{equation}
  because \(\boldsymbol{r}_{l-2}=0\) on \(G\). On the face \(F\), the trace
  of \(\boldsymbol{r}_{l-2}\) is tangential, so
  \(\boldsymbol{n}_{F}\cdot\boldsymbol{r}_{l-2}=0\). Hence the triple-product
  identity gives
  \begin{equation}
    \begin{aligned}
      \boldsymbol{n}_{F}\times\left(\boldsymbol{x}_{K}\times\boldsymbol{r}_{l-2}\right)
       & =\left(\boldsymbol{n}_{F}\cdot\boldsymbol{r}_{l-2}\right)\boldsymbol{x}_{K}
      -\left(\boldsymbol{n}_{F}\cdot\boldsymbol{x}_{K}\right)\boldsymbol{r}_{l-2} \\
       & =-c_{F}\boldsymbol{r}_{l-2}
      =\widetilde{\boldsymbol{r}}_{l-2}
      =b_{F}\bfitgreek{\psi}_{m}
      \quad\text{on }F.
    \end{aligned}
  \end{equation}
  Therefore the assumed boundary orthogonality yields
  \begin{equation}
    0=\left(\boldsymbol{z}^{\partial K},\boldsymbol{n}\times
    \left(\boldsymbol{x}_{K}\times\boldsymbol{r}_{l-2}\right)\right)_{\partial K}
    =\left(\boldsymbol{z}^{F},b_{F}\bfitgreek{\psi}_{m}\right)_{F}.
  \end{equation}
  Since \(\boldsymbol{z}^{F}\) belongs to the face-local serendipity edge family
  on \(F\), \cref{prop:face-separation} implies
  \begin{equation}
    \boldsymbol{z}^{F}=0\quad\text{on }F.
  \end{equation}
  As \(F\) was arbitrary, we conclude that
  \begin{equation}
    \boldsymbol{z}^{\partial K}=0\quad\text{on }\partial K.
  \end{equation}
\end{proof}

\begin{proposition}[Rotational Boundary Orthogonality on the Reduced Edge Kernel]
  \label{prop:edge-direct-identity}
  Let \(l\) be admissible and let \(\boldsymbol{z}\in\mathcal{Z}_{l}(K)\). Then,
  for every \(\boldsymbol{r}_{l-2}\in\mathbb{P}^{3}_{l-2}(K)\), one has
  \begin{equation}
    \left(\boldsymbol{z}^{\partial K},\boldsymbol{n}\times
    \left(\boldsymbol{x}_{K}\times\boldsymbol{r}_{l-2}\right)\right)_{\partial K}=0
    \quad \forall \boldsymbol{r}_{l-2}\in\mathbb{P}^{3}_{l-2}(K).
  \end{equation}
\end{proposition}

\begin{proof}
  Let \(\boldsymbol{q}_{l-1}:=\boldsymbol{x}_{K}\times\boldsymbol{r}_{l-2}\in
  \mathbb{P}^{3}_{l-1}(K)\). Since
  \(\nabla\times\boldsymbol{q}_{l-1}\in\mathbb{P}^{3}_{l-2}(K)\subseteq
  \mathbb{P}^{3}_{l}(K)\) and \(\bfitgreek{\Pi}_{l}\boldsymbol{z}=0\), we obtain
  \begin{equation}
    0=\left(\boldsymbol{z},\nabla\times\boldsymbol{q}_{l-1}\right)_{K}.
  \end{equation}
  Set
  \begin{equation}
    \boldsymbol{w}:=\nabla\times\boldsymbol{z}
    \in\mathbf{V}^{\mathrm{Sf}}_{k-1,l-1}(K).
  \end{equation}
  Integrating by parts gives
  \begin{equation}
    0=\left(\boldsymbol{w},\boldsymbol{q}_{l-1}\right)_{K}
    +\left(\boldsymbol{z}^{\partial K},\boldsymbol{n}\times\boldsymbol{q}_{l-1}\right)_{\partial K}.
  \end{equation}
  By the exact commuting property,
  \begin{equation}
    \bfitgreek{\Pi}^{\mathrm{Sf}}_{k-1}\boldsymbol{w}
    =\bfitgreek{\Pi}^{\mathrm{Sf}}_{k-1}(\nabla\times\boldsymbol{z})
    =\nabla\times\left(\bfitgreek{\Pi}^{\mathrm{Se}}_{k}\boldsymbol{z}\right)=0.
  \end{equation}
  Hence the defining condition of the face serendipity kernel yields
  \begin{equation}
    \left(\boldsymbol{w},\boldsymbol{x}_{K}\times\boldsymbol{r}_{l-2}\right)_{K}=0.
  \end{equation}
  Therefore
  \begin{equation}
    \left(\boldsymbol{z}^{\partial K},\boldsymbol{n}\times
    \left(\boldsymbol{x}_{K}\times\boldsymbol{r}_{l-2}\right)\right)_{\partial K}=0,
  \end{equation}
  proving the claim.
\end{proof}

We next prove the low-order nonreentry statement. Once this step is available,
the two preceding propositions close the full edge-space argument.

The point of the low-order step is to show that once
\(\bfitgreek{\Pi}^{\mathrm{Se}}_{k}\boldsymbol{z}=0\), the apparently missing
low-order information of \(\boldsymbol{p}_{k}:=\bfitgreek{\Pi}_{l}\boldsymbol{z}\in
\mathbb{P}^{3}_{k}(K)\) can be recovered from the same rotational testing family,
after a zero-boundary bubble correction that removes the bulk term involving
\(\nabla\times\boldsymbol{p}_{k}\).

\begin{proposition}[Low-Order Nonreentry on the High-Order Kernel of the Edge Projector]
  \label{prop:edge-nonreentry-from-transfer}
  There exists an integer \(l_{\mathrm{nr},*}=l_{\mathrm{nr},*}(K,k)\) such that,
  for every admissible order \(l\geq l_{\mathrm{nr},*}\) and every
  \(\boldsymbol{z}\in\mathbf{V}^{\mathrm{Se}}_{k,l}(K)\) with
  \(\bfitgreek{\Pi}^{\mathrm{Se}}_{k}\boldsymbol{z}=0\), one has
  \begin{equation}
    \bfitgreek{\Pi}_{l}\boldsymbol{z}\in \mathbb{P}^{3}_{k}(K)
    \implies \bfitgreek{\Pi}_{l}\boldsymbol{z}=0.
  \end{equation}
  Equivalently,
  \begin{equation}
    \bfitgreek{\Pi}_{l}\left(
    \ker \bfitgreek{\Pi}^{\mathrm{Se}}_{k}
    \cap \mathbf{V}^{\mathrm{Se}}_{k,l}(K)\right)
    \cap \mathbb{P}^{3}_{k}(K)=\{0\}.
  \end{equation}
\end{proposition}

\begin{proof}
  Let \(\eta\) be the number of faces of \(K\), let the affine functions
  \(\ell_{F}\) be as in the proof of \cref{thm:stf}, and set
  \begin{equation}
    b_{K}:=\prod_{F\subseteq\partial K}\ell_{F}^{2}.
  \end{equation}
  Choose
  \begin{equation}
    l_{\mathrm{nr},*}\geq
    \max\left\{l_{\mathrm{f},*}+1,\ k+2\eta+2,\ 
    \max_{F\subseteq\partial K}(m_{F,*}+2\eta)\right\}.
  \end{equation}

  Let \(l\geq l_{\mathrm{nr},*}\), and let
  \(\boldsymbol{z}\in\mathbf{V}^{\mathrm{Se}}_{k,l}(K)\) satisfy
  \begin{equation}
    \bfitgreek{\Pi}^{\mathrm{Se}}_{k}\boldsymbol{z}=0,
    \qquad
    \boldsymbol{p}_{k}:=\bfitgreek{\Pi}_{l}\boldsymbol{z}\in\mathbb{P}^{3}_{k}(K).
  \end{equation}
  We prove that \(\boldsymbol{p}_{k}=0\).

  Set
  \begin{equation}
    \boldsymbol{s}_{k-1}:=\nabla\times\boldsymbol{p}_{k}
    \in\mathbb{P}^{3}_{k-1}(K).
  \end{equation}
  Fix a face \(F\subseteq\partial K\), define
  \begin{equation}
    c_{F}:=\boldsymbol{x}_{K}\cdot\boldsymbol{n}_{F},
    \qquad
    m:=l-2\eta\geq m_{F,*},
  \end{equation}
  and let \(\bfitgreek{\psi}_{m}\in\mathbb{P}_{m,t}(F)\). By
  \cref{lem:face-lift}, there exists
  \(\widetilde{\boldsymbol{r}}_{l-2}\in\mathbb{P}^{3}_{l-2}(K)\) such that
  \begin{equation}
    \widetilde{\boldsymbol{r}}_{l-2\mid G}=0
    \quad \forall G\subseteq\partial K,\ G\neq F,
    \qquad
    \widetilde{\boldsymbol{r}}_{l-2\mid F}=b_{F}\bfitgreek{\psi}_{m}.
  \end{equation}
  Exactly as in the proof of \cref{prop:edge-direct-reduction}, the polynomial
  field
  \begin{equation}
    \boldsymbol{q}^{\,0}_{l-1}:=\boldsymbol{x}_{K}\times
    \left(-c_{F}^{-1}\widetilde{\boldsymbol{r}}_{l-2}\right)
    \in \boldsymbol{x}_{K}\times\mathbb{P}^{3}_{l-2}(K)
  \end{equation}
  satisfies
  \begin{equation}
    \boldsymbol{n}\times\boldsymbol{q}^{\,0}_{l-1}=0
    \quad\text{on }\partial K\setminus F,
    \qquad
    \boldsymbol{n}_{F}\times\boldsymbol{q}^{\,0}_{l-1}=b_{F}\bfitgreek{\psi}_{m}
    \quad\text{on }F.
  \end{equation}

  If \(\boldsymbol{s}_{k-1}=0\), set
  \(\boldsymbol{q}_{l-1}:=\boldsymbol{q}^{\,0}_{l-1}\). Otherwise,
  \(\boldsymbol{x}_{K}\times\boldsymbol{s}_{k-1}\not\equiv 0\), for if
  \(\boldsymbol{x}_{K}\times\boldsymbol{s}_{k-1}\equiv 0\), then
  \(\boldsymbol{s}_{k-1}=\boldsymbol{x}_{K}\sigma\) for some scalar polynomial
  \(\sigma\), and the identity
  \(\nabla\cdot\boldsymbol{s}_{k-1}=0\) would give
  \(3\sigma+\boldsymbol{x}_{K}\cdot\nabla\sigma=0\), hence
  \(\sigma=0\), a contradiction. Define
  \begin{equation}
    \boldsymbol{r}^{\,b}_{k+2\eta}:=b_{K}\,\boldsymbol{x}_{K}\times
    \boldsymbol{s}_{k-1}\in \mathbb{P}^{3}_{k+2\eta}(K)
    \subseteq\mathbb{P}^{3}_{l-2}(K),
    \qquad
    \boldsymbol{q}^{\,b}_{k+2\eta+1}:=\boldsymbol{x}_{K}\times
    \boldsymbol{r}^{\,b}_{k+2\eta}.
  \end{equation}
  Since \(b_{K}=0\) on \(\partial K\), one has
  \(\boldsymbol{n}\times\boldsymbol{q}^{\,b}_{k+2\eta+1}=0\) on \(\partial K\). Moreover,
  \begin{equation}
    \begin{aligned}
      \left(\boldsymbol{s}_{k-1},\boldsymbol{q}^{\,b}_{k+2\eta+1}\right)_{K}
       &= \left(\boldsymbol{s}_{k-1},\boldsymbol{x}_{K}\times
       \left(b_{K}\,\boldsymbol{x}_{K}\times\boldsymbol{s}_{k-1}\right)\right)_{K} \\
       &= -\left\|b_{K}^{\frac12}\,\boldsymbol{x}_{K}\times
       \boldsymbol{s}_{k-1}\right\|_{K}^{2}<0.
    \end{aligned}
  \end{equation}
  Therefore the field
  \begin{equation}
    \boldsymbol{q}_{l-1}:=\boldsymbol{q}^{\,0}_{l-1}
    -\frac{\left(\boldsymbol{s}_{k-1},\boldsymbol{q}^{\,0}_{l-1}\right)_{K}}
    {\left(\boldsymbol{s}_{k-1},\boldsymbol{q}^{\,b}_{k+2\eta+1}\right)_{K}}
    \boldsymbol{q}^{\,b}_{k+2\eta+1}
  \end{equation}
  belongs to \(\boldsymbol{x}_{K}\times\mathbb{P}^{3}_{l-2}(K)\), has the same
  boundary trace as \(\boldsymbol{q}^{\,0}_{l-1}\), and satisfies
  \begin{equation}
    \left(\boldsymbol{s}_{k-1},\boldsymbol{q}_{l-1}\right)_{K}=0.
  \end{equation}

  Set \(\boldsymbol{w}:=\nabla\times\boldsymbol{z}\in
  \mathbf{V}^{\mathrm{Sf}}_{k-1,l-1}(K)\). By the exact commuting property,
  \begin{equation}
    \bfitgreek{\Pi}^{\mathrm{Sf}}_{k-1}\boldsymbol{w}
    =\nabla\times\left(\bfitgreek{\Pi}^{\mathrm{Se}}_{k}\boldsymbol{z}\right)=0.
  \end{equation}
  Hence, exactly as in the proof of \cref{prop:edge-direct-identity}, the
  defining condition of the face serendipity kernel gives
  \begin{equation}
    \left(\boldsymbol{w},\boldsymbol{q}_{l-1}\right)_{K}=0.
  \end{equation}
  Since \(\nabla\times\boldsymbol{q}_{l-1}\in\mathbb{P}^{3}_{l-2}(K)
  \subseteq\mathbb{P}^{3}_{l}(K)\) and
  \(\bfitgreek{\Pi}_{l}\boldsymbol{z}=\boldsymbol{p}_{k}\), we obtain
  \begin{equation}
    0=\left(\boldsymbol{z}-\boldsymbol{p}_{k},
    \nabla\times\boldsymbol{q}_{l-1}\right)_{K}.
  \end{equation}
  Integrating by parts yields
  \begin{equation}
    0=\left(\boldsymbol{w}-\boldsymbol{s}_{k-1},\boldsymbol{q}_{l-1}\right)_{K}
    +\left((\boldsymbol{z}-\boldsymbol{p}_{k})^{\partial K},
    \boldsymbol{n}\times\boldsymbol{q}_{l-1}\right)_{\partial K}
    =\left((\boldsymbol{z}-\boldsymbol{p}_{k})^{F},
    b_{F}\bfitgreek{\psi}_{m}\right)_{F}.
  \end{equation}
  Since \((\boldsymbol{z}-\boldsymbol{p}_{k})^{F}\) belongs to the face-local
  serendipity edge family on \(F\), \cref{prop:face-separation}
  implies
  \begin{equation}
    (\boldsymbol{z}-\boldsymbol{p}_{k})^{F}=0\quad\text{on }F.
  \end{equation}
  As \(F\) was arbitrary, we conclude that
  \begin{equation}
    (\boldsymbol{z}-\boldsymbol{p}_{k})^{\partial K}=0\quad\text{on }\partial K.
  \end{equation}

  Set
  \begin{equation}
    \boldsymbol{u}:=\nabla\times(\boldsymbol{z}-\boldsymbol{p}_{k})
    \in\mathbf{V}^{\mathrm{Sf}}_{k-1,l-1}(K).
  \end{equation}
  Then \(\nabla\cdot\boldsymbol{u}=0\), and the vanishing tangential trace on each
  face implies
  \begin{equation}
    \boldsymbol{u}\cdot\boldsymbol{n}=0\quad\text{on }\partial K.
  \end{equation}
  Moreover, for every \(\boldsymbol{r}_{l-2}\in\mathbb{P}^{3}_{l-2}(K)\),
  \begin{equation}
    0=\left(\boldsymbol{z}-\boldsymbol{p}_{k},
    \nabla\times(\boldsymbol{x}_{K}\times\boldsymbol{r}_{l-2})\right)_{K}
    =\left(\boldsymbol{u},\boldsymbol{x}_{K}\times\boldsymbol{r}_{l-2}\right)_{K},
  \end{equation}
  because the boundary term vanishes. Applying \cref{eq:boundf} at order
  \(l-1\) therefore gives
  \begin{equation}
    \boldsymbol{u}=0\quad\text{in }K,
  \end{equation}
  that is,
  \begin{equation}
    \nabla\times\boldsymbol{z}=\nabla\times\boldsymbol{p}_{k}.
  \end{equation}

  We now verify the defining conditions
  \cref{eq:se1,eq:se2,eq:se3,eq:se4} for \(\boldsymbol{p}_{k}\). The first follows
  from \((\boldsymbol{z}-\boldsymbol{p}_{k})^{\partial K}=0\) on \(\partial K\). The second
  and fourth follow from
  \(\nabla\times\boldsymbol{z}=\nabla\times\boldsymbol{p}_{k}\) together with
  \(\bfitgreek{\Pi}^{\mathrm{Se}}_{k}\boldsymbol{z}=0\). For the third, let
  \(p_{k_{\mathrm{d}}}\in\mathbb{P}_{k_{\mathrm{d}}}(K)\). Since
  \(\boldsymbol{x}_{K}p_{k_{\mathrm{d}}}\in\mathbb{P}^{3}_{k}(K)
  \subseteq\mathbb{P}^{3}_{l}(K)\), the identity
  \(\bfitgreek{\Pi}_{l}\boldsymbol{z}=\boldsymbol{p}_{k}\) gives
  \begin{equation}
    \left(\boldsymbol{p}_{k},\boldsymbol{x}_{K}p_{k_{\mathrm{d}}}\right)_{K}
    =\left(\boldsymbol{z},\boldsymbol{x}_{K}p_{k_{\mathrm{d}}}\right)_{K}=0,
  \end{equation}
  again because \(\bfitgreek{\Pi}^{\mathrm{Se}}_{k}\boldsymbol{z}=0\). Hence
  \begin{equation}
    \bfitgreek{\Pi}^{\mathrm{Se}}_{k}\boldsymbol{p}_{k}=0.
  \end{equation}
  Since \(\boldsymbol{p}_{k}\in\mathbb{P}^{3}_{k}(K)\), \cref{prop:se} implies
  that \(\bfitgreek{\Pi}^{\mathrm{Se}}_{k}\) acts as the identity on
  \(\boldsymbol{p}_{k}\). Therefore
  \begin{equation}
    \boldsymbol{p}_{k}=\bfitgreek{\Pi}^{\mathrm{Se}}_{k}\boldsymbol{p}_{k}=0.
  \end{equation}
  This proves the stated low-order nonreentry property.
\end{proof}

\begin{corollary}[Full-Space Stable Polynomial Projection for Edge Spaces]
  \label{cor:edge-full-space-stability}
  There exists an integer \(l_{\mathrm{e},*}=l_{\mathrm{e},*}(K,k)\) such that
  \(\bfitgreek{\Pi}_{l}\) is injective on \(\mathbf{V}^{\mathrm{Se}}_{k,l}(K)\)
  for every sufficiently large admissible order \(l\).
\end{corollary}

\begin{proof}
  Let \(l\geq \max\{l_{\mathrm{nr},*},l_{\mathrm{t},*}\}\), and let
  \(\boldsymbol{v}\in\mathbf{V}^{\mathrm{Se}}_{k,l}(K)\) satisfy
  \(\bfitgreek{\Pi}_{l}\boldsymbol{v}=0\). Split
  \begin{equation}
    \boldsymbol{v}=\bfitgreek{\Pi}^{\mathrm{Se}}_{k}\boldsymbol{v}+\boldsymbol{z},
    \qquad
    \bfitgreek{\Pi}^{\mathrm{Se}}_{k}\boldsymbol{z}=0.
  \end{equation}
  Since \(\bfitgreek{\Pi}^{\mathrm{Se}}_{k}\boldsymbol{v}\in\mathbb{P}^{3}_{k}(K)\)
  and \(\bfitgreek{\Pi}_{l}\) is the identity on \(\mathbb{P}^{3}_{k}(K)\), we have
  \begin{equation}
    0=\bfitgreek{\Pi}_{l}\boldsymbol{v}
    =\bfitgreek{\Pi}^{\mathrm{Se}}_{k}\boldsymbol{v}+\bfitgreek{\Pi}_{l}\boldsymbol{z}.
  \end{equation}
  Hence \(\bfitgreek{\Pi}_{l}\boldsymbol{z}\in \mathbb{P}^{3}_{k}(K)\), and
  \cref{prop:edge-nonreentry-from-transfer} yields
  \begin{equation}
    \bfitgreek{\Pi}_{l}\boldsymbol{z}=0.
  \end{equation}
  Therefore also
  \begin{equation}
    \bfitgreek{\Pi}^{\mathrm{Se}}_{k}\boldsymbol{v}=0.
  \end{equation}
  Therefore \(\boldsymbol{z}\in\mathcal{Z}_{l}(K)\). By
  \cref{prop:edge-direct-identity,prop:edge-direct-reduction}, we obtain
  \begin{equation}
    \boldsymbol{z}^{\partial K}=0\quad\text{on }\partial K.
  \end{equation}

  Since the trace of an edge field on each face is tangential, the identity
  \(\boldsymbol{z}^{\partial K}=0\) implies that
  \(\boldsymbol{z}^{F}=0\) on every face. Consequently,
  \begin{equation}
    \nabla_{F}\times\boldsymbol{z}=0\quad\text{on each }F\subseteq\partial K,
  \end{equation}
  hence
  \begin{equation}
    \left(\nabla\times\boldsymbol{z}\right)\cdot\boldsymbol{n}=0
    \quad\text{on }\partial K.
  \end{equation}
  Moreover,
  \begin{equation}
    \boldsymbol{z}\cdot\boldsymbol{t}=0\text{ on }\partial F,
    \qquad
    \left(\boldsymbol{z},\boldsymbol{x}_{F}p_{\beta_{\mathrm{d}}}\right)_{F}=0
    \quad \forall F\subseteq\partial K,
  \end{equation}
  because \(\boldsymbol{x}_{F}\) is tangential to \(F\).

  Finally, the cell moments already vanish on \(\mathcal{Z}_{l}(K)\):
  \begin{equation}
    \left(\nabla\times\boldsymbol{z},\boldsymbol{x}_{K}\times\boldsymbol{p}_{\mu_{\mathrm{r}}}\right)_{K}=0
    \quad \forall \boldsymbol{p}_{\mu_{\mathrm{r}}}\in\mathbb{P}^{3}_{\mu_{\mathrm{r}}}(K),
  \end{equation}
  by \cref{eq:se4,eq:equivse}, and
  \begin{equation}
    \left(\boldsymbol{z},\boldsymbol{x}_{K}p_{k_{\mathrm{d}}}\right)_{K}=0
    \quad \forall p_{k_{\mathrm{d}}}\in\mathbb{P}_{k_{\mathrm{d}}}(K),
  \end{equation}
  by \cref{eq:se3}. Therefore every term on the right-hand side of
  \cref{eq:bounde} vanishes when applied to \(\boldsymbol{z}\), and we obtain
  \begin{equation}
    \|\boldsymbol{z}\|_{K}=0.
  \end{equation}
  Hence \(\boldsymbol{z}=0\).
  Therefore \(\boldsymbol{v}=0\), proving the injectivity of
  \(\bfitgreek{\Pi}_{l}\) on \(\mathbf{V}^{\mathrm{Se}}_{k,l}(K)\).
\end{proof}